\documentclass{book}[12pt, final]
\usepackage{graphicx}
\usepackage{amsmath, amsfonts,amssymb}
\usepackage{eucal}
\usepackage{amsthm}
\usepackage{mathrsfs}
\usepackage{etex}
\usepackage[all,cmtip,arrow,2cell]{xy}
\usepackage[colorlinks,
             linkcolor=black,
             citecolor=blue,
             bookmarksnumbered=true]{hyperref}
\usepackage{appendix}
\usepackage{latexsym, amscd}
\usepackage{float}
\usepackage{soul}
\restylefloat{table}

\newtheorem{prop}{Proposition}[section]
\newtheorem{thm}{Theorem}[section]
\newtheorem{cor}{Corollary}[section]
\newtheorem{lem}[thm]{Lemma}

\theoremstyle{definition}
\newtheorem{dfn}{Definition}[section]
\newtheorem{rmk}{Remark}[section]
\newtheorem*{notat}{Notation}
\newtheorem{ex}{Example}[section]

\theoremstyle{remark}
\theoremstyle{remark}

\theoremstyle{remark}
\def\Ha{\mathscr{H}\!}
\def\HA{\mathbb{H}}
\def\Ei{\mathbb{E}_\i}
\def\Affg{\mathbf{Aff}_\g}
\def\Aff{\mathbf{Aff}}
\def\Specg{\mathbf{Spec}_\g}
\def\Pro{\mathbf{Pro}}
\def\gsch{\mathbf{Sch}\left(\g\right)}
\newcommand{\sch}[1]{\mathbf{Sch}\left(#1\right)}
\newcommand{\Dm}[1]{\mathfrak{D\mbox{-}MSt}\left(#1\right)}
\def\Dmg{\mathfrak{D\mbox{-}MSt}\left(\g\right)}
\def\Spec{\mathbf{Spec}}
\def\dm{\mathfrak{D\mbox{-}MSt}}
\def\et{\acute{e}t}
\def\Symp{\mathcal{S}ymp}

\def\sA{\mathscr A}
\def\sB{\mathscr B}
\def\sE{\mathscr E}
\def\bE{\mathbf E}
\def\bF{\mathbf F}
\def\bA{\mathbf A}

\def\cB{\mathcal B}
\def\sD{\mathscr D}
\def\cD{\mathcal D}
\def\G{\mathcal G}
\def\g{\mathscr G}
\def\Uni{\mathbb{U}}

\def\Mfd{\mathit{Mfd}}

\def\sL{\mathscr L}
\def\cL{\mathcal L}
\def\C{\mathscr C}
\def\X{\mathscr X}
\def\Y{\mathscr Y}

\def\g{\mathscr G}
\def\K{\mathscr K}

\def\lbar{\overline{\sL}}
\def\dbar{\overline{\sD}}
\def\lbare{\overline{\sL}^{\et}}
\def\dbare{\overline{\sD}^{\et}}
\def\Kk{\mathcal{K}}
\def\M{\mathcal{M}}
\def\Fun{\mathbf{Fun}}

\def\cE{\mathcal E}
\def\cF{\mathcal F}
\def\G{\mathcal G}

\def\U{\mathcal U}
\def\sU{\mathscr U}
\def\ubar{\overline{\sU}}
\def\ubare{\ubar^{\et}}
\def\V{\mathcal V}

\def\sl{\mathcal{S}l}
\def\cosl{\mathcal{C}o\mathcal{S}l}
\def\i{\infty}
\def\ni{\left(n,1\right)}
\def\n1i{\left(n+1,1\right)}
\def\nGpd{\mathbf{Gpd}_n}
\def\iGpd{\mathbf{Gpd}_\i}
\def\iCat{\widehat{\mathbf{Cat}}_\i}
\def\LiGpd{\widehat{\mathbf{Gpd}}_\i}
\def\longlongrightarrow{-\!\!\!-\!\!\!-\!\!\!-\!\!\!-\!\!\!-\!\!\!\longrightarrow}
\def\longlonglongrightarrow{-\!\!\!-\!\!\!-\!\!\!-\!\!\!-\!\!\!-\!\!\!\longlongrightarrow}

\newcommand*{\longhookrightarrow}{\ensuremath{\lhook\joinrel\relbar\joinrel\rightarrow}}
\newcommand*{\longlonghookrightarrow}{\ensuremath{\lhook\joinrel\relbar\joinrel\relbar\joinrel\rightarrow}}
\def\fdot{\mspace{3mu} \cdot \mspace{3mu}}
\def\T{\mathfrak{Top}_\i}
\def\Tn{\mathfrak{Top}_n}
\def\Tm{\mathfrak{Top}_m}

\def\Geo{\mathfrak{Geom}}
\def\Str{\mathfrak{Str}}
\def\Sset{\Set^{\Delta^{op}}}
\def\Pshi{\Psh_\i}
\def\LPshi{\widehat{\Psh}_\i}
\def\icat{\left(\i,1\right)\mbox{-category}}
\def\icat{\mbox{$\i$-category}}
\def\qcats{\mbox{$\i$-categories}}

\def\cinf{{\mathbf{C}^\i}}
\def\lbarcinf{\overline{\sL_{\cinf}}}
\def\tp{\underset{\widetilde{}} \varphi}
\def\wt{\widetilde{\varphi}}
\def\colim{\underrightarrow{\mbox{colim}\vspace{0.5pt}}\mspace{4mu}}
\renewcommand{\lim}{\varprojlim}

\newcommand{\Hom}{\operatorname{Hom}}
\DeclareMathOperator{\St}{St}
\DeclareMathOperator{\Set}{Set}
\DeclareMathOperator{\Sh}{Sh}
\DeclareMathOperator{\Shi}{Sh_\i}
\def\LShi{\widehat{\Sh}_\i}
\DeclareMathOperator{\Psh}{Psh}

\DeclareMathOperator{\Lan}{Lan}

\newcommand{\Etd}{\cinf\mbox{-}\mathscr{\acute{E}}\!\mathit{tendues}_\i}
\newcommand{\Etdd}{\cinf\mbox{-}\mathscr{\acute{E}}\!\mathit{tendues}}
\newcommand{\Etds}{\mathfrak{EtDiffSt}_\i}
\newcommand{\Etdsd}{\mathfrak{EtDiffSt}}
\newcommand{\Etdsn}{n\mbox{-}\mathfrak{EtDiffSt}_\i}

\def\RR{\mathbb{R}}

\def\lllarrows{\hspace{.05cm}\mbox{\,\put(0,-3){$\leftarrow$}\put(0,1){$\leftarrow$}\put(0,5){$\leftarrow$}\hspace{.45cm}}}

\renewcommand{\O}{\mathcal{O}}
\newcommand{\rt} {\mspace{4mu} \rotatebox[origin=C]{90}{$\perp$} \mspace{4mu}}

\newcommand{\Adj}[4]{\xymatrix@1{#2 \ar@<-0.5ex>[r]_-{#4} & #3 \ar@<-0.5ex>[l]_-{#1}}}
\usepackage[body={6.0in, 8.2in},left=1.25in,right=1.25in]{geometry}

\usepackage{fancyhdr}                    
\usepackage{setspace}

\makeatletter 
\def\cleardoublepage{\clearpage\if@twoside \ifodd\c@page\else%
    \hbox{}%
    \thispagestyle{empty}
    \newpage%
    \if@twocolumn\hbox{}\newpage\fi\fi\fi} 
\makeatother 

\begin{document}
\title{Higher Orbifolds and Deligne-Mumford Stacks as Structured Infinity-Topoi}
\author{David Carchedi}

\maketitle

\begin{center}
\section*{Abstract}
\end{center}
We develop a universal framework to study smooth higher orbifolds on the one hand and higher Deligne-Mumford stacks (as well as their derived and spectral variants) on the other, and use this framework to obtain a completely categorical description of which stacks arise as the functor of points of such objects. We choose to model higher orbifolds and Deligne-Mumford stacks as infinity-topoi equipped with a structure sheaf, thus naturally generalizing the work of Lurie in \cite{dag}, but our approach applies not only to different settings of algebraic geometry such as classical algebraic geometry, derived algebraic geometry, and the algebraic geometry of commutative ring spectra as in \cite{dag}, but also to differential topology, complex geometry, the theory of supermanifolds, derived manifolds etc., where it produces a theory of higher generalized orbifolds appropriate for these settings. This universal framework yields new insights into the general theory of Deligne-Mumford stacks and orbifolds, including a representability criterion which gives a categorical characterization of such generalized Deligne-Mumford stacks. This specializes to a new categorical description of classical Deligne-Mumford stacks, a result sketched in \cite{prol}, which extends to derived and spectral Deligne-Mumford stacks as well. 
\tableofcontents

\chapter{Introduction}


\emph{Deligne-Mumford stacks} were introduced in 1969 in the seminal paper \cite{DM}. They are locally modeled by quotients of schemes by finite group actions, and are used to model moduli spaces of objects which have interesting automorphisms, such as elliptic curves. These automorphisms can be interpreted as automorphisms of the points of the Deligne-Mumford stack in question. The theory of Deligne-Mumford stacks has lead to substantial developments, in particular in algebraic geometry \cite{laff} and Gromov-Witten theory \cite{kontt}. Higher categorical generalizations of Deligne-Mumford stacks have been defined by To\"en and Vezzosi in \cite{Hagy2} and also by Lurie in \cite{dag}, and they apply to far reaching generalizations of algebraic geometry, e.g. the algebraic geometry of simplicial commutative rings, and the algebraic geometry of $\Ei$-ring spectra, both of which are versions of \emph{derived algebraic geometry.} 

Predating Deligne-Mumford stacks by over a decade \cite{satake}, the notion of a \emph{smooth orbifold} grew out of differential topology and the study of automorphic forms, and also appeared in the study of 3-manifolds, and in foliation theory. Orbifolds are locally modeled on quotients of $\mathbb{R}^n$ by a linear action of a finite group, and are analogues of Deligne-Mumford stacks. Although initially defined as topological spaces with an orbifold atlas, this point of view leads to problems with the notion of morphism. This can be remedied by encoding orbifolds as differentiable stacks \cite{Dorette}, which one may think of geometrically as generalized manifolds whose points are allowed to have automorphisms. \emph{\'Etale differentiable stacks} are differentiable stacks which are mild generalizations of smooth orbifolds whose examples include not only orbifolds, but more generally, leaf spaces of foliated manifolds and quotients of manifolds by almost-free Lie group actions. They have been studied by various authors \cite{Ie,Dorette,stacklie,hepworth,morsifold,Wockel,Giorgio,etalspme,prol}. In this manuscript, we lay down the basic framework to study \emph{higher \'etale differentiable stacks}. Morally, such a higher stack can be thought of as a generalization of a smooth manifold which allows points to have automorphisms, and also allows automorphisms to have automorphisms themselves, and for these automorphisms between automorphisms to have automorphisms ad infinitum.

In this manuscript, we develop a universal framework to study higher \'etale differentiable stacks and orbifolds on the one hand, and (higher) algebraic Deligne-Mumford stacks, as well as their derived and spectral analogues, on the other. This framework was not developed just to provide a useful language that puts orbifolds and Deligne-Mumford stacks on the same footing, although this is a certainly a nice consequence, but instead was developed as a tool to be used to reveal new insights about these objects and their generalizations. In particular, the framework developed facilitates the proof of a very general representability theorem; we give categorical characterizations of higher generalized orbifolds and Deligne-Mumford stacks, as well as their functors of points, which are broad generalizations of the results in \cite{prol}. In particular, this yields a new characterization of classical Deligne-Mumford stacks which extends to the derived and spectral setting as well. In the differentiable setting, this characterization shows that there is a natural correspondence between $n$-dimensional higher \'etale differentiable stacks, and classical fields for $n$-dimensional field theories, in the sense of \cite{field}.

\section{Overview of our approach}
An \emph{$\infty$-topos} is an $\left(\infty,1\right)$-category of $\infty$-sheaves, where the latter are higher categorical versions of sheaves and stacks taking values in the $\left(\infty,1\right)$-category $\iGpd$ of $\infty$-groupoids rather than sets or groupoids. An \emph{$\infty$-groupoid} is a model for the homotopy type of a topological space, so $\infty$-sheaves are essentially homotopy sheaves of spaces. The homeomorphism type of a (sober) topological space $X$ can be recovered from its topos of sheaves of sets, $\Sh\left(X\right),$ and a general topos can be thought of as a generalized space, whose points can have automorphisms. The homeomorphism type of $X$ is also encoded by the $\i$-topos $\Shi\left(X\right)$ of $\i$-sheaves of $\i$-groupoids over $X,$ and a general $\i$-topos, in a similar manner as for an ordinary topos, may be regarded as a generalized space whose points can posses automorphisms, but with the additional possibility of these automorphisms having automorphisms themselves, and the automorphisms between automorphisms having automorphisms, and so on. As this is the same geometric intuition behind how the underlying ``space'' of a higher orbifold or higher Deligne-Mumford stack should behave, $\i$-topoi become natural candidates for modeling the topology of such higher stacks.

All the geometric objects in our framework are modeled as $\infty$-topoi with an appropriate structure sheaf. We start by specifying a collection of such geometric objects, our \emph{local models}, and then define generalized orbifolds and Deligne-Mumford stacks as those objects which are locally equivalent to our local models, in much the same way one builds manifolds out of Euclidean spaces, or schemes out of affine schemes. This is a natural simultaneous generalization of the approach of Lurie in \cite{dag} for modeling Deligne-Mumford stacks and their derived and spectral analogues, and the approach of Pronk in \cite{Dorette} for modeling \'etale differentiable stacks. For example, if we take the local models to be all open submanifolds of $\RR^n,$ $n\ge 0,$ regarded as locally ringed spaces (which can in turn be viewed as locally ringed $\infty$-topoi), the resulting theory is the theory of higher \'etale differentiable stacks, and if instead the local models are all those of the form $\mbox{Spec}_{\et}\left(\mathbf{E}\right),$ the small \'etale spectrum of a connected $\Ei$-ring spectrum $\mathbf{E}$, the resulting theory is the theory of spectral Deligne-Mumford stacks in the sense of \cite{spectral}.

\section{Organization and Main Results}
Chapter \ref{chap:prelim} offers a brief review of higher topos theory, which is used extensively in the rest of the manuscript.

In Chapter \ref{chap:etale} we develop the role of $n$-topoi (for $0 \le n \le \i$) as generalized spaces, paying particular attention to the notion of a local homeomorphism between $n$-topoi. We introduce a Grothendieck topology, called the \'etale topology, on higher topoi, which provides an elegant framework for formalizing what one means by gluing higher topoi together along local homeomorphisms.

Chapter \ref{chap:structured} introduces the concept of \emph{structured $\i$-topoi} and reviews the notions of \emph{geometry} and \emph{geometric structure} introduced in Lurie in \cite{dag}. The role of these notions is to provide a natural generalization of the concept of locally ringed space suitable for $\i$-topoi which carry structure sheaves with interesting algebraic structure, e.g., a structure sheaf of commutative rings, simplicial commutative rings, $\Ei$-ring spectra, etc. We also review the notion of an \emph{\'etale morphism} of such structured $\i$-topoi. These are an appropriate generalization of the notion of a local homeomorphism of $\i$-topoi, which take into account the structure sheaf. For example, for smooth manifolds, it agrees with the concept of a local diffeomorphism, for schemes with the Zariski topology, it agrees the concept of an open immersion, and for schemes with the \'etale topology, it agrees with the standard definition of an \'etale map of schemes used in algebraic geometry. Finally, we introduce a variant of the \'etale Grothendieck topology on $\i$-topoi suitable for these structured $\i$-topoi.

The cornerstone of this manuscript is Chapter \ref{chap:etendues}. We develop a natural framework to model generalized higher orbifolds and Deligne-Mumford stacks as structured $\i$-topoi. The general set up is as follows: We start with an $\i$-category $\sL$ of structured $\i$-topoi, whose objects are to be considered our local models. We then define an $\i$-category $\lbar$ of \textbf{$\sL$-\'etendues}; an $\sL$-\'etendue, in a precise sense, is a structured $\i$-topos which is locally modeled on objects in $\sL.$ We reserve the terminology \textbf{Deligne-Mumford stack} for the $\i$-sheaf on $\sL$ one gets by considering the functor of points of such an $\sL$-\'etendue. \footnote{In order to have a good notion of functor of points, one needs to impose a certain local smallness condition on the $\i$-category of local models $\sL,$ and also one may need to slightly enlarge $\sL$ in such a way that does not change the resulting $\i$-category $\lbar$, but we will ignore these technical details in this introduction.} For example, if $\sL$ is all affine derived schemes (realized by their small \'etale spectra), then $\sL$-\'etendues are what Lurie calls \emph{derived Deligne-Mumford stacks} in \cite{dag}.

We then construct the $\i$-topos of $\i$-sheaves over the local models $\sL,$ $\Shi\left(\sL\right),$ and prove the following, which essentially states that an $\sL$-\'etendue is determined by its functor of points on the local models:

\begin{thm}
There is a full and faithful functor $$\lbar \hookrightarrow \Shi\left(\sL\right)$$ between the $\left(\infty,1\right)$-category of $\sL$-\'etendues and the $\left(\infty,1\right)$-category of $\i$-sheaves over $\sL.$
\end{thm}

We then define an $\i$-sheaf $\X$ on $\sL$ to be \textbf{Deligne-Mumford} if it is in the essential image of the above full and faithful functor. In particular, the $\left(\infty,1\right)$-categories of $\sL$-\'etendues and of Deligne-Mumford stacks on $\sL$ are canonically equivalent.

Recall that an \emph{\'etale morphism} of $\sL$-\'etendues is essentially a local homeomorphism of underlying $\i$-topoi which respects their structure sheaves. Our first major result is the following:

\begin{thm}\label{thm:int2}
If the $\left(\infty,1\right)$-category-category of local models $\sL$ is essentially small, then there is a canonical equivalence of $\left(\infty,1\right)$-categories
$$\lbare\simeq \Shi\left(\sL^{\et}\right)$$
between the $\left(\infty,1\right)$-category of $\sL$-\'etendues and only their \'etale morphisms, and the $\left(\infty,1\right)$-category of $\i$-sheaves over the $\left(\infty,1\right)$-category consisting of the objects of $\sL$ and only their \'etale morphisms. In particular, $\lbare$ is an $\i$-topos. Moreover, the $\i$-topos $\Uni:=\Shi\left(\sL^{\et}\right)$ carries a canonical structure sheaf $\O_\Uni,$ making the pair $\left(\Uni,\O_\Uni\right)$ an $\sL$-\'etendue. Furthermore, $\left(\Uni,\O_\Uni\right)$ is the terminal object in $\lbare.$ (In particular, such a terminal object exists).
\end{thm}

Put another way: the $\left(\infty,1\right)$-category of $\i$-sheaves over the site of local models and their \'etale maps (local homeomorphisms which respect the structure sheaves) is equivalent to the $\left(\infty,1\right)$-category consisting of the \'etendues (schemes relative to these local models) and their \'etale maps, whenever the local models form a set. This general principle was first noticed in \cite{prol}, in the $2$-categorical context, with the local models constrained to be structured topological spaces. Here is what this means concretely:

An $\sL$-\'etendue $\cE$ is determined by its functor of points:
$$\Hom\left(\mspace{3mu}\cdot\mspace{3mu},\cE\right):\sL^{op} \to \iGpd.$$
The above functor is an $\i$-sheaf on $\sL$, and is in fact the Deligne-Mumford stack associated to $\cE$. However, $\cE$ is also determined by its 
\'etale points, which form an $\i$-sheaf on the $\left(\infty,1\right)$-category $\sL^{\et}$ of local models \emph{and only their \'etale morphisms}. The \'etale points of $\cE$ are the $\i$-sheaf denoted by $\tilde y^{\et}\left(\cE\right),$ which assigns to an object $\cL$ of $\sL$ the space
$$\tilde y^{\et}\left(\cE\right)\left(\cL\right)=\left\{f:\cL \to \cE\mspace{3mu}|\mspace{3mu} f\mbox{ \'etale}\right\}$$
of \'etale morphisms from $\cE$ to $\cL.$

The above theorem, in particular, states the converse, namely that if $F$ is an \emph{arbitrary} $\i$-sheaf on $\sL^{\et},$ then it is the $\i$-sheaf of \'etale points of some $\sL$-\'etendue (equivalently the \'etale points of a Deligne-Mumford stack): there exists an $\sL$-\'etendue $\Theta\left(F\right)$ unique with the property that for any object $\cL$ in $\sL,$ space
$F\left(\cL\right)$ is homotopy equivalent to the space of \'etale maps from  $\cL$ to $\Theta\left(F\right)$
$$F\left(\cL\right)\simeq \left\{f:\cL \to \Theta\left(F\right)\mspace{3mu}|\mspace{3mu} f\mbox{ \'etale}\right\}.$$
These two constructions $\tilde y^{\et}$ and $\Theta$ are mutually inverse to each other.

Even more surprisingly, the collection of all $\sL$-\'etendues (or equivalently Deligne-Mumford stacks)  together with all their \'etale morphisms forms an $\i$-topos, and this $\i$-topos is actually the underlying ``space'' of a certain universal \'etendue: the terminal object in $\lbare.$ At the same time, this $\i$-topos is also canonically equivalent to $\Shi\left(\sL^{\et}\right).$ These are the main new insights of this manuscript about the theory of Deligne-Mumford stacks.

Given an essentially small $\left(\infty,1\right)$-category of local models $\sL,$ we denote by $j$ the canonical functor $$j:\sL^{\et} \to \sL$$ between the $\left(\infty,1\right)$-category consisting of the local models and only their \'etale morphisms and the $\left(\infty,1\right)$-category of local models and \emph{all} their morphisms. There is an obvious restriction functor $$j^*:\Shi\left(\sL\right) \to \Shi\left(\sL^{\et}\right)$$ between their associated $\i$-topoi of $\i$-sheaves which, when $\sL$ is small, we show admits a left adjoint $j_!$ called the \textbf{\'etale prolongation functor}. This functor gives a natural way of \emph{prolonging} an $\i$-sheaf that is a priori only functorial with respect to \'etale morphisms to produce a new $\i$-sheaf which is functorial with respect to all morphisms.

Our next major result is the following:

\begin{thm}\label{thm:int3}
An $\i$-sheaf $\X$ on $\sL$ is Deligne-Mumford if and only if it is in the essential image of the \'etale prolongation functor $$j_!:\Shi\left(\sL^{\et}\right) \to \Shi\left(\sL\right).$$
\end{thm}

The above may be regarded as a representability theorem for Deligne-Mumford stacks. Notice that its only input is the canonical functor $$j:\sL^{\et} \to \sL$$ between the $\left(\infty,1\right)$-category of local models and \'etale morphisms and the $\left(\infty,1\right)$-category of local models and all morphisms, and the Grothendieck topology on $\sL.$ The above theorem states that the $\left(\infty,1\right)$-category of Deligne-Mumford stacks, which a priori was defined by geometric means by gluing together local models along \'etale morphisms, is actually determined by the \emph{the functor $j$ and this Grothendieck topology alone} ($j_!$ can be constructed directly from $j$ as a certain left Kan extension). Not only this, but the functor $j$ gives a precise characterization of how this $\left(\infty,1\right)$-category of Deligne-Mumford stacks sits inside of the $\left(\infty,1\right)$-category of $\i$-sheaves $\Shi\left(\sL\right),$ namely, it is precisely the image of $j_!$.

We then turn to the case where the local models in question need not form a set, or more precisely, the case where the $\left(\infty,1\right)$-category $\sL$ of local models need not be essentially small. Nonetheless, each object $\cL$ of $\sL$ is contained in some small subcategory of $\sL,$ and for each such subcategory $\sU,$ one has the canonical functor $$\sU^{\et} \to \sU \to \sL,$$ which we denote by $j_\sU,$ where $\sU^{\et}$ consists of the objects of $\sU$ and their \'etale morphisms, and one has a restriction functor $$j_\sU^*:\Shi\left(\sL\right) \to \Shi\left(\sU^{\et}\right)$$ between their associated $\left(\infty,1\right)$-categories of $\i$-sheaves, which we show admits a left adjoint $j^\sU_!$ called the \textbf{relative \'etale prolongation functor with respect to $\sU$}.

We generalize Theorem \ref{thm:int3} to the following:

\begin{thm}\label{thm:int4}
An $\i$-sheaf $\X$ on $\sL$ is Deligne-Mumford if and only if it is in the essential image of a relative prolongation functor $$j^\sU_!:\Shi\left(\sU^{\et}\right) \to \Shi\left(\sL\right),$$ for some small subcategory $\sU$ of $\sL.$
\end{thm}

In Chapter \ref{chap:examples}, we spell out the general machinery developed in Chapter \ref{chap:etendues} in the setting of differential topology, and in the cases of classical, derived, and spectral algebraic geometry.

Specifically, in Section \ref{sec:orbifolds} we define the concept of a differentiable $\i$-orbifold and an \'etale differentiable $\i$-stack. If $\sL$ is the category of smooth manifolds, regarded as locally ringed spaces spaces, we call an $\sL$-\'etendue a \textbf{smooth $\i$-\'etendue} and we define an \textbf{\'etale differentiable $\i$-stack} to be a Deligne-Mumford stack with respect to $\sL$. We then reconcile these definitions with the existing definitions of differentiable orbifolds and \'etale differentiable stack in the literature:

\begin{thm}
An $\i$-sheaf $\X$ on the category of smooth manifolds $\Mfd$ (with respect to the open cover topology) is an \'etale differentiable stack if and only if it is an \'etale differentiable $\i$-stack, and is $1$-truncated (i.e. is a stack of groupoids). Similarly, $\X$ is a differentiable orbifold if and only if it is a differentiable $\i$-orbifold and is $1$-truncated.
\end{thm}

We also prove:

\begin{thm}
An $\i$-sheaf $\X$ on $\Mfd$ is an \'etale differentiable $\i$-stack if and only if it arises from an \'etale simplicial manifold.
\end{thm}

The following theorem is then a special case of Theorem \ref{thm:int2}:

\begin{thm}
There is a canonical equivalence of $\left(\infty,1\right)$-categories
$$\Etds^{\et} \simeq \Shi\left(\Mfd^{\et}\right)$$
between the $\left(\infty,1\right)$-category of \'etale differentiable $\i$-stacks and local diffeomorphisms, and the $\left(\infty,1\right)$-category of $\i$-sheaves on the site of smooth manifolds and local diffeomorphisms. Similarly, there is a canonical equivalence of $\left(\infty,1\right)$-categories $$\Etdsn^{\et} \simeq \Shi\left(n\mbox{-}\Mfd^{\et}\right)$$ between the $\left(\infty,1\right)$-category of $n$-dimensional \'etale differentiable $\i$-stacks and local diffeomorphisms, and the $\left(\infty,1\right)$-category of $\i$-sheaves on the site of smooth $n$-manifolds and local diffeomorphisms.
\end{thm}

\begin{rmk}
The $\left(\infty,1\right)$-category $\Shi\left(n\mbox{-}\Mfd^{\et}\right)$ is the same as the $\left(\infty,1\right)$-category of \emph{classical fields for an $n$-dimensional field theory} as defined by Freed and Teleman in \cite{field}. Therefore, one may interpret the above theorem as saying that $n$-dimensional \'etale differentiable $\i$-stacks (generalized higher orbifolds) are geometric incarnations of classical fields for an $n$-dimensional field theory. However we will not develop this idea further in this manuscript.
\end{rmk}

Denote by $j:\Mfd^{\et} \to \Mfd$ the canonical functor from the category of smooth manifolds and local diffeomorphisms into the category of smooth manifolds and all smooth maps. The induced restriction functor $$j^*:\Shi\left(\Mfd\right)\to \Shi\left(\Mfd^{\et}\right)$$ has a left adjoint $j_!$ called the \'etale prolongation functor. The following theorem is a special case of Theorem \ref{thm:int3}:

\begin{thm}
An $\i$-sheaf $\X$ on $\Mfd$ is an \'etale differentiable $\i$-stack if and only if it is in the essential image of the \'etale prolongation functor $j_!$.
\end{thm}

The $2$-categorical analogue of the above theorem was proven in \cite{prol}.

It is then show how the language of smooth $\i$-\'etendues leads naturally to an extension of the \textbf{\'etal\'e space construction} for \'etale differentiable stacks, as developed in \cite{etalspme} to the setting of \'etale differentiable $\i$-stacks. That is, to any $\i$-sheaf $F$ on an \'etale differentiable $\i$-stack $\X,$ there exists a local diffeomorphism $$LF \to \X$$ from another \'etale differentiable $\i$-stack whose $\i$-sheaf of sections recovers $F.$  We then recall the notion of an \textbf{effective \'etale differentiable stack}, which is basically an \'etale differentiable stack whose isotropy groups act effectively.

Finally, we prove the following:

\begin{thm}
Any \'etale differentiable $\i$-stack $\X$ can be realized as the \'etal\'e space of an $\i$-gerbe over an effective \'etale differentiable stack.
\end{thm}

In Section \ref{sec:dmgeom}, we apply the results of Chapter \ref{chap:etendues} to the setting of $\g$-schemes for a \textbf{geometry} in the sense of \cite{dag}. In particular, we work out consequences of these results to the theory of higher Deligne-Mumford stacks, derived Deligne-Mumford stacks, and spectral Deligne-Mumford stacks.

Our first result is the following:

\begin{thm}
Let $k$ be a commutative ring and let $\sL$ be the category of affine $k$-schemes, viewed as a category of structured $\i$-topoi, with each commutative $k$-algebra $A$ being associated to its small \'etale $\i$-topos $\Shi\left(A_{\et}\right)$ (with an appropriate structure sheaf). An $\i$-sheaf on affine $k$-schemes with respect to the \'etale topology is a Deligne-Mumford stack, in the classical sense, if and only if it is a Deligne-Mumford stack for $\sL$ and is $1$-truncated, i.e. is a stack of groupoids.
\end{thm}

We then give a long list of theorems, which are all special cases of Theorems \ref{thm:int2}, \ref{thm:int3}, and \ref{thm:int4}. For brevity, we will only list a sample of these in this introduction. All of these theorems have variants which apply to (higher) Deligne-Mumford stacks, derived Deligne-Mumford stacks, and spectral Deligne-Mumford stacks. Most have variants for the Zariski topology as well. Let $k$ be a commutative ring and $\bE$ a connective $\Ei$-ring spectrum, and let $\Aff_k$, $\mathbf{DAff}_k$ and $\mathbf{SpAff}_\bE$ denote the $\left(\infty,1\right)$-categories of affine $k$-schemes, derived affine $k$-schemes, and connective spectral $\bE$-schemes respectively.

We prove the following:

\begin{thm}
There is a canonical equivalence $$\dm^{der}\left(\mathbf{LFP}_k\right)^{\et} \simeq \Shi\left(\mathbf{DAff}^{\mathit{lfp},\mspace{2mu}{\et}}_k\right)$$
between the $\left(\infty,1\right)$-category of derived Deligne-Mumford stacks locally of finite type over $k$ and their \'etale morphisms, and the $\left(\infty,1\right)$-category of $\i$-sheaves on the site of derived affine $k$-schemes locally of finite type and their \'etale morphisms (with respect to the \'etale topology).
\end{thm}

Given $\sU$ a subcategory of $\mathbf{SpAff}_\bE,$ we say that $\sU$ is \textbf{\'etale closed} if for all \'etale morphisms of $\Ei$-rings $$f:\bF \to \bA,$$ with $\Spec\left(\bF\right)$ in $\sU,$ $\Spec\left(\bA\right)$ is also in $\sU.$ We make similar conventions for $\Aff_k$ and $\mathbf{DAff}_k.$ For such a subcategory, denote by $j^\sU$ the canonical functor $$\sU^{\et} \to \sU \to \mathbf{SpAff}_\bE,$$ where $\sU^{\et}$ denotes the subcategory of $\sU$ spanned by only \'etale morphisms. There is a restriction functor $$j^*_\sU:\Shi\left(\mathbf{SpAff}_\bE,{\et}\right) \to \Shi\left(\sU^{\et},{\et}\right)$$ between $\left(\infty,1\right)$-categories of $\i$-sheaves (with respect to the \'etale topology) which has a left adjoint $j^\sU_!,$ again called the relative \'etale prolongation functor with respect to $\sU.$

\begin{thm}
An $\i$-sheaf in $\Shi\left(\mathbf{SpAff}_\bE,\acute{e}t\right)$ on (connective) affine spectral $\bE$-schemes with respect to the \'etale topology is a spectral Deligne-Mumford stack if and only if $\X$ is in the essential image of a relative prolongation functor $j^\sU_!,$ for $\sU$ a small \'etale closed subcategory of $\mathbf{SpAff}_\bE.$
\end{thm}

For an \'etale closed subcategory $\sU$ of $\Aff_k,$ denote again by $j^\sU_!$ the left adjoint to the restriction functor $$j^*_\sU:\St\left(\Aff_k,\acute{e}t\right) \to \St\left(\sU^{\et},\acute{e}t\right)$$ between stacks of groupoids (with respect to the \'etale topology), and call it again the relative \'etale prolongation functor with respect to $\sU.$

\begin{thm}
Let $k$ be a commutative ring. A stack of groupoids on affine $k$-schemes with respect to the \'etale topology is a (classical) Deligne-Mumford stack\footnote{We are using the definition of Deligne-Mumford stack introduced by Lurie in \cite{dag}, which is slightly unconventional in that it does not impose any separation conditions on the diagonal of such a stack.}, if and only if $\X$ is in the essential image of a relative prolongation functor $j^\sU_!,$ for $\sU$ a small \'etale closed subcategory of $\Aff_k.$
\end{thm}

\section{Conventions and Notations}
We will closely follow the notations and conventions of \cite{htt}. In particular, by an $\icat,$ we will mean a quasicategory or inner-Kan complex (hence an $\left(\infty,1\right)$-category). We will freely use the definitions and notational conventions of \cite{htt} and refer the reader to its index and notational index. There are a few instances where we will deviate from the notation of op. cit., most of which are explained in the body of this manuscript as we use them. The rest we will explain here:\\

\begin{itemize}
\item What Lurie writes as $\mathcal{S}$ and calls the $\icat$ of spaces, we will write as $\iGpd$ and call the $\icat$ of $\i$-groupoids.
\item For $C$ an object of an $\icat$ $\C,$ we will write the slice $\icat$ as $\C/C$ as opposed to $\C_{/C}.$
\item If $C$ and $D$ are objects of an $\icat$ $\C,$ we will write $\Hom_\C\left(C,D\right)$ as their mapping space as opposed to $\operatorname{Map}_\C\left(C,D\right).$
\end{itemize}

\section*{Acknowledgments}
I would like to thank to Rune Haugseng, Gijs Heuts, Jacob Lurie, and Urs Schreiber for useful conversation. In particular, I am indebted to Gijs for the proof of an essential lemma, and to Jacob for providing counterexamples and suggestions without which this manuscript may not have been realized. I would also like to thank Anton Fetisov and Zhen Lin Low for providing answers to questions of mine on Mathoverflow. I should also thank Ezra Getzler for catching some minor misprints in a previous draft of this manuscript, and Kai Behrend and Damien Calaque for useful discussions about some of the examples. I am grateful to the Max Planck Institute for Mathematics for providing the pleasant and mathematically stimulating environment in which most of this research was conducted, and I wish to thank Chris Schommer-Pries and Peter Teichner for organizing a year long seminar on higher category theory from which I benefited greatly. Finally, I would also like to thank the topology group of MIT for their hospitality, where this research was conducted in part during the summers of 2012 and 2013.

\chapter{Preliminaries on higher topos theory}\label{chap:prelim}
\begin{dfn}
Let $\nGpd$ denote the $\icat$ of $n$-groupoids, that is the full subcategory of the $\icat$ of $\i$-groupoids consisting of $n$-truncated objects. (This is an $\n1i$-category). For a given   $\icat$ $\C$, denote by $\Psh_n\left(\C\right)$ the $\icat$ of functors $$\C^{op} \to \nGpd.$$ Such functors will be called \textbf{$n$-presheaves}.
\end{dfn}

\begin{ex}
$\textbf{Gpd}_0$ is the category of sets, whereas $\textbf{Gpd}_1$ is the $\left(2,1\right)$-category of essentially small groupoids.
\end{ex}

\begin{dfn}
If $\C$ is an $\left(n,1\right)$-category, then each object $C$ of $\C$ gives rise to an $\left(n-1\right)$-presheaf $y\left(C\right)$ which assigns to each object $D$ of $\C$ the $\left(n-1\right)$-groupoid $\Hom_\C\left(D,C\right).$ (See Section 5.1.3 of \cite{htt} to see how one can make this precise.) This produces a functor $$y:\C \to \Psh_{n-1}\left(\C\right)$$ which is full and faithful by Proposition 5.3.1.1 of op. cit. The functor $y$ is called the \textbf{Yoneda embedding} of $\C.$
\end{dfn}

\begin{lem}\textbf{The $\i$-Yoneda Lemma} (Lemma 5.5.2.1 of \cite{htt})
Let $0 \le n \le \i$. If $\C$ is an $\left(n,1\right)$-category with $C$ an object of $\C,$ and $$F:\C^{op} \to \textbf{Gpd}_{\left(n-1\right)}$$ is an $\left(n-1\right)$-presheaf on $\C,$ then there is a canonical equivalence of $\left(n-1\right)$-groupoids $$\Hom\left(y\left(C\right),F\right) \simeq F\left(C\right).$$
\end{lem}

\begin{dfn}
An $\icat$ $\cE$ is an $n$-topos if and only if it is an accessible left exact localization of the form
$$\xymatrix@1{\cE\mspace{4mu} \ar@{^{(}->}[r]<-0.9ex>_-{i} & \Psh_{n-1}\left(\C\right) \ar@<-0.5ex>[l]_-{a}}$$ for some $\icat$ $\C$, that is to say, $a$ is accessible, left exact (i.e. preserves finite limits), and left adjoint to $i,$ and $i$ is full and faithful.
\end{dfn}

\begin{rmk}
An $n$-topos is automatically an $\ni$-category. Furthermore, in the above definition, there is no harm in requiring that $\C$ be an $\ni$-category, since for any $\icat$ $\sD,$
\begin{eqnarray*}
\Psh_{n-1}\left(\sD\right)&=&\Fun\left(\sD,\textbf{Gpd}_{\left(n-1\right)}\right)\\
&\simeq & \Fun\left(\tau_n\left(\sD\right),\textbf{Gpd}_{\left(n-1\right)}\right)\\
&=&\Psh_{n-1}\left(\tau_n\left(\sD\right)\right),
\end{eqnarray*}
where $\tau_n\left(\sD\right)$ is the truncation of $\sD$ to an $\left(n,1\right)$-category, since $\textbf{Gpd}_{\left(n-1\right)}$ is a an $\left(n,1\right)$-category.
\end{rmk}

\begin{dfn}
A geometric morphism $f:\cE \to \cF$ between two $n$-topoi $\cE$ and $\cF$, with $0\le n \le \i,$ is a pair of adjoint functors  $$\Adj{f^{\ast}}{\cE}{\cF}{f_{\ast}},$$
with $f^* \rt f_*,$ such that $f^*$ is left exact. The functor $f_*$ is called the \textbf{direct image functor} of $f$ whereas $f^*$ is called the \textbf{inverse image functor}.
\end{dfn}

Since either adjoint $f_*$ or $f^*$ determines the other up to a contractible space of choices, one may use only one adjoint to define an $\icat$ of $n$-topoi:

\begin{dfn}
Denote by $\mathfrak{Top}_n$ the $\icat$ of $n$-topoi. It is the subcategory of $\widehat{\mathcal{C}at}_\i$, the $\icat$ of not necessarily small\footnote{This can be made precise by using Grothendieck universes.} $\i$-categories, whose objects are $n$-topoi and whose morphisms $f:\cE \to \cF$ consist of only those functors which admit a left exact left adjoint. It is easy to see that when $n < \i,$ $\mathfrak{Top}_n$ is in fact an $\n1i$-category.
\end{dfn}

\begin{rmk}
When $n=\i$, this is what Lurie denotes by $\mathcal{R}{\mathcal T}op$ in \cite{htt}. It is canonically equivalent to the opposite category of $\mathcal{\cL}{\mathcal T}op$, which is the subcategory of $\widehat{\mathcal{C}at}_\i$ whose objects are $\i$-topoi and whose morphisms $f:\cE \to \cF$ are left exact functors $\cF \to \cE$ which admit a \emph{right} adjoint.
\end{rmk}

\begin{dfn}\label{dfn:geofun}
Let $\cE$ and $\cF$ be two $\i$-topoi. Denote by $\Geo\left(\cE,\cF\right)$ the $\icat$ of geometric morphisms $\cE \to \cF$ with \emph{not necessarily invertible} natural transformations as morphisms. It is the full subcategory of the functor $\icat$, $\Fun\left(\cF,\cE\right),$ consisting of those functors $f^*:\cF \to \cE$ which are left exact and admit a right adjoint.
\end{dfn}

\begin{ex}
Let $f:\C \to \sD$ be a functor between small $\i$-categories. There is a canonically induced restriction functor $$f^*:\Pshi\left(\sD\right) \to \Pshi\left(\C\right).$$ This functor has a left adjoint $f_!$ which assigns to each functor $F:\C^{op} \to \iGpd$ its left Kan extension along $f^{op}.$ The functor $f_!$ itself is a left Kan extension, namely $$f_!=\Lan_{y_\C} \left(f \circ y_\sD\right),$$ where $y_\C$ and $y_\sD$ are the Yoneda embeddings of $\C$ and $\sD$ respectively. The functor $f^*$ also has a right adjoint $f_*$, which sends a functor $F:\C^{op} \to \iGpd$ to its right Kan extension along $f^{op}.$ By the Yoneda Lemma, it follows that for each $D$ in $\sD,$ $$f_*\left(F\right)\left(D\right) \simeq \Hom\left(f^*y_\sD\left(D\right),F\right).$$ In particular, since $f^*$ has both a left and a right adjoint, it is left exact, so the pair $$\left(f_*,f^*\right)$$ is a geometric morphism.
\end{ex}

\begin{dfn}\label{dfn:nloc}
Let $n \le m \le \i.$ If $\cE$ is an $m$-topos, the full subcategory $\tau_{\le n-1}\left(\cE\right)$ of $\cE$ consisting of the $\left(n-1\right)$-truncated objects of $\cE$ is an $n$-topos, and this defines a functor $$\tau^m_{\le n-1}:\mathfrak{Top}_m \to \mathfrak{Top}_n.$$ This functor is left adjoint to a full and faithful embedding $$\nu^m_n:\mathfrak{Top}_n \to \mathfrak{Top}_m,$$ hence determines a localization. An $m$-topos equivalent to one of the form $\nu^{m}_n\left(\cF\right)$ for an $n$-topos $\cF,$ is called \textbf{$n$-localic}. If $\cE$ is an $m$-topos, the $n$-topos $\tau^m_{\le n-1}\left(\cE\right)$ is called the \textbf{$n$-localic reflection} of $\cE.$  When no confusion will arise, we will often denote $\nu^m_n$ by simply $\nu_n,$ and similarly will denote $\tau^m_{n-1}$ by $\tau_{n-1}$. See Section 6.4.5 of \cite{htt} for more details.
\end{dfn}

\begin{rmk}\label{rmk:loclocsmall}
On one hand, the $\icat$ $\T$ is not locally small. E.g., there exists a topos $\mathcal{A}$ called the \emph{classifying topos for abelian groups}, and it has the property that the category of geometric morphisms $\Geo\left(\Set,\mathcal{A}\right)$ is equivalent to the category of abelian groups, which is not small. Since the $\left(2,1\right)$-category of topoi is equivalent to the full subcategory of $\T$ on the $1$-localic objects, this implies that $\T$ is not locally small.

On the other hand, if $\cE$ and $\cF$ are two $\i$-topoi, then the $\i$-category $\Geo\left(\cE,\cF\right)$ is locally small. To see this, first notice that it is a full subcategory of $\Fun^{L}\left(\cF,\cE\right),$ the $\i$-category of colimit-preserving functors. Suppose that $\sD$ is a small $\i$-category for which $\cF$ is a left exact localization of $\Pshi\left(\sD\right).$ Then by Proposition 5.5.4.20 and Theorem 5.1.5.6 of \cite{htt}, $\Fun^{L}\left(\cF,\cE\right)$ is equivalent to a full subcategory of $$\Fun\left(\sD,\cE\right),$$ which is an $\i$-topos, hence locally small. Finally, observe that the $\i$-groupoid $$\Hom_{\T}\left(\cE,\cF\right)$$ is equivalent to the maximal sub Kan complex of $\Geo\left(\cE,\cF\right)$, so that $\T$ has locally small mapping spaces.
\end{rmk}

\section{The epi-mono factorization system}
\begin{dfn}
Let $f:C \to D$ be a morphism in an $\icat$ $\C.$ Then $f$ is a \textbf{monomorphism} if it is $\left(-1\right)$-truncated.
\end{dfn}

We now recall a basic construction we will use throughout the paper. Let $\C$ be an $\icat$ with pullbacks. A morphism $f:C \to D$ induces a functor $$C\left(f\right)_\bullet:\Delta^{op} \to \C$$ called its \textbf{\v{C}ech nerve} (which is unique up to a contractible space of choices). The simplicial object $C\left(f\right)_\bullet$ may be described concretely in terms of iterative pullbacks by the formula $$C\left(f\right)_n = C \times_{D} \cdots \times_{D} C,$$ where the pullback above is the $n^{\mbox{th}}$ iterated fibered product over $D.$ In fact, there is always canonical lift
$$\xymatrix{\Delta^{op} \ar@{-->}[rd]_-{\tilde C\left(f\right)_\bullet} \ar[r]^{C\left(f\right)_\bullet} & \C\\
& \C/D \ar[u],}$$
and $\tilde C\left(f\right)_\bullet$ may be identified with a cocone $$\widehat{C}\left(f\right)_\bullet:\left(\Delta^{op}\right)^{\triangleright}\cong \Delta_{+}^{op} \longlongrightarrow \C$$ with vertex $D.$

\begin{dfn}
A morphism $f:C \to D$ in an $\icat$ $\C$ with pullbacks is an \textbf{effective epimorphism} if the cocone $\widehat{C}\left(f\right)_\bullet$ is colimiting.
\end{dfn}

\begin{rmk}
The construction of the \v{C}ech nerve of a morphism may be carried out functorially as follows: a morphism $f:C \to D$ is an object of $$\C/D=\Fun\left(\Delta^{op}_{\le 0},\C/D\right),$$ where $\Delta_{\le 0}$ is the full subcategory of the simplex category $\Delta$ on $\left[0\right].$ The canonical inclusion $$i_0:\Delta_{\le 0} \hookrightarrow \Delta$$ induces a restriction functor $$i^*_0:\Fun\left(\Delta^{op},\C/D\right) \to \Fun\left(\Delta^{op}_{\le 0},\C/D\right)$$ which has a right adjoint $i^0_*$. If the composite is denote by $$\operatorname{cosk_0}:=i^0_* \circ i^*_0,$$ one has $$\tilde C\left(f\right)_\bullet=\operatorname{cosk_0}\left(f\right).$$
\end{rmk}

\begin{prop}\label{prop:factorization}
(\cite{htt}, Example 5.2.8.16) Let $\cE$ be an $\infty$-topos. Then effective epimorphisms and monomorphisms together form a factorization system on $\cE$, in the sense defined in \cite{htt}, Section 5.2.8.
\end{prop}

\begin{rmk}
This factorization system is a natural generalization of the classical epi/mono factorization system of ordinary $1$-topoi. Explicitly, if $f:E \to F$ is a morphism in an $\i$-topos $\cE,$ one can form its \v{C}ech nerve, $$C\left(f\right)_\bullet:\Delta^{op} \to \C.$$ The factorization of $f$ into an effective epimorphism followed by a monomorphism is given by $$E \to \colim C\left(f\right)_\bullet \to F.$$
\end{rmk}

\section{Grothendieck topologies}
\begin{dfn}
A \textbf{sieve} on an object $C$ in an $\icat$ $\C$ is a subobject in $\Psh_\i\left(\C\right)$ of $y\left(C\right),$ where $$y:\C \hookrightarrow \Psh_\i\left(\C\right)$$ is the Yoneda embedding.
\end{dfn}

\begin{dfn}
A Grothendieck topology \index{Grothendieck topology} $J$ on an $\icat$ $\C$ is an assignment to every object $C$ of $\C$, a set $Cov(C)$ of \textbf{covering sieves} \index{sieve} on $C$, such that

\begin{itemize}
\item[i)] The maximal subobject $y\left(C\right)$ is a covering sieve on $C.$
\item[ii)] If $R$ is a covering sieve on $C$ and $f:D \to C$, then $f^{*}\left(R\right)$ is a covering sieve on $D$.
\item[iii)] If $R$ is a sieve on $C$ and $S$ is a covering sieve on $C$ such that for each object $D$ and every arrow $$f \in S\left(D\right) \subseteq \Hom_{\C}\left(D,C\right)_0$$
$f^{*}\left(R\right)$ is a covering sieve on $D$, then $R$ is a covering sieve on $C$.
\end{itemize}
An $\icat$ together with a Grothendieck topology is called a \textbf{Grothendieck site}.
\end{dfn}

\begin{dfn}
A \textbf{Grothendieck pretopology} on an $\icat$ $\C$ is an assignment to each object $C$ of $\C$ a collection $\mathcal{B}\left(C\right)$ of families of arrows $\left(U_i \to C\right)_{i\in I}$, called \textbf{covering families}, such that

\begin{itemize}
\item[i)] If $D \to C$ is an equivalence, then $\left(D \to C\right)$ is in $\mathcal{B}\left(C\right).$
\item[ii)] If $\left(U_i \to C\right)_i$ is in $\mathcal{B}\left(C\right)$ and $f:D \to C$, then the fibered products $U_i \times_{C} D$ exist and the set of the induced maps $U_i \times_{C} D \to D$ is in $\mathcal{B}\left(D\right)$.
\item[iii)] If $\left(U_i \to C\right)_i$ is in $\mathcal{B}\left(C\right)$ and for each $i,$ $\left(V_{ij} \to U_i\right)_j$ is in $\mathcal{B}\left(U_i\right)$, then $\left(V_{ij} \to C\right)_{ij}$ is in $\mathcal{B}\left(C\right)$.
\end{itemize}
\end{dfn}

\begin{rmk}
Given a Grothendieck pretopology $J,$ and $\mathcal{U}=\left(U_i \to C\right)_{i\in I}$ a $J$-covering family, one can factor the induced map $$\coprod\limits_i y\left(U_i\right) \to y\left(C\right)$$ uniquely (up to a contractible space of choices) as an effective epimorphism followed by a monomorphism as in Proposition \ref{prop:factorization}, $$\coprod\limits_i y\left(U_i\right) \to S_{\mathcal{U}} \to y\left(C\right).$$ The monomorphism $S_{\mathcal{U}} \to y\left(C\right)$ corresponds to a sieve on $C.$ This assignment generates a Grothendieck topology by declaring a sieve $R$ on $C$ is in $Cov\left(C\right)$ if there exists a covering family $\mathcal{U}$ such that $S_{\mathcal{U}} \subseteq R$.
\end{rmk}

\begin{dfn}
If $\left(\C,J\right)$ is a Grothendieck site, then a functor $$F:\C^{op} \to \iGpd$$ is an \textbf{$\i$-sheaf} if for every object $C$ of $\C,$ and every covering sieve $S$ on $C,$ the induced morphism
$$F\left(C\right) \simeq \Hom\left(y\left(C\right),F\right) \to \Hom\left(S,F\right)$$ is an equivalence of $\i$-groupoids. An $\i$-sheaf is called an \textbf{$n$-sheaf} if in addition, it is $n$-truncated.
\end{dfn}

\begin{rmk}
A $0$-sheaf is a sheaf, in the classical sense, whereas a $1$-sheaf is a stack of groupoids.
\end{rmk}

For any site $\left(\C,J\right),$ the $\icat$ of $\i$-sheaves is a left exact localization
$$\xymatrix@1{\Shi\left(\C,J\right)\mspace{4mu} \ar@{^{(}->}[r]<-0.9ex>_-{i} & \Pshi\left(\C\right) \ar@<-0.5ex>[l]_-{a}}.$$ (See Lemma 6.2.2.7 of \cite{htt}.)
A localization of this form is called \textbf{topological.} Furthermore, by Proposition 6.4.3.6 of op. cit., the $\left(n+1,1\right)$-category of $n$-sheaves is a left exact localization
\begin{equation}\label{eq:grothp}
\xymatrix@1{\Sh_n\left(\C,J\right)\mspace{4mu} \ar@{^{(}->}[r]<-0.9ex>_-{i} & \Psh_{n}\left(\C\right) \ar@<-0.5ex>[l]_-{a}}.
\end{equation}

\begin{thm}\label{thm:nlocalicsheaf}
Let $n < \i$ and suppose that $$\xymatrix@1{\cE\mspace{4mu} \ar@{^{(}->}[r]<-0.9ex>_-{i} & \Psh_{n-1}\left(\C\right) \ar@<-0.5ex>[l]_-{a}}$$ is an expression of an $n$-topos as a left exact localization. Then there is a unique Grothendieck topology on $\C$ for which the localization is equivalent to one of the form (\ref{eq:grothp}).
\end{thm}

\begin{rmk}\label{rmk:nosheav}
The above theorem is \emph{not} true for $n=\i.$ For an $\icat$ $\C,$ there can exist left exact localizations of the form $$\xymatrix@1{\cE\mspace{4mu} \ar@{^{(}->}[r]<-0.9ex>_-{i} & \Psh_\i\left(\C\right) \ar@<-0.5ex>[l]_-{a}}$$
which do not arise from Grothendieck topologies on $\C.$ When this happens, there is always a unique Grothendieck topology $J$ on $\C$ and a factorization
$$\xymatrix{\cE\mspace{4mu} \ar@{^{(}->}[r]<-0.9ex>_-{i} & \Sh_\i\left(\C,J\right) \ar@{^{(}->}[r]<-0.9ex>_-{i} \ar@<-0.5ex>[l]_-{l} & \Psh_\i\left(\C\right) \ar@<-0.5ex>[l]_-{a}},$$
for which the localization $$\xymatrix@1{\cE\mspace{4mu} \ar@{^{(}->}[r]<-0.9ex>_-{i} & \Sh_\i\left(\C,J\right)  \ar@<-0.5ex>[l]_-{l}}$$ is \emph{cotopological}. (See Proposition 6.5.2.19 of \cite{htt}.)
\end{rmk}

\begin{dfn}\label{dfn:canonical}
Every $n$-topos $\cE$ carries a distinguished Grothendieck topology, called the \textbf{epimorphism topology}. It is generated by covering families $\left(E_\alpha \to E\right)$ such that $$\coprod\limits_\alpha E_\alpha \to E$$ is an effective epimorphism. We denote the associated Grothendieck site by $\left(\cE,can\right).$
\end{dfn}

Finally, we end with a few useful observations:

\begin{rmk}\label{rmk:5.3.5.4}
By the proof of Corollary 5.3.5.4 of \cite{htt}, one has a canonical identification
$$\Psh_\i\left(\C\right)/F \simeq \Psh_\i\left(\C/F\right).$$
\end{rmk}

\begin{prop}\label{prop:locslice}
Suppose that $$\xymatrix@1{\cE\mspace{4mu} \ar@{^{(}->}[r]<-0.9ex>_-{i} & \Psh_\i\left(\C\right) \ar@<-0.5ex>[l]_-{a}}$$ is a left exact localization. Then, for each $E$ in $\cE,$ $\cE/E$ is a left exact localization of $\Psh_\i\left(\C/E\right)$. A similar statement holds for $n$-topoi for any $n$.
\end{prop}

\begin{proof}
By Remark \ref{rmk:5.3.5.4} one has a canonical identification
\begin{equation}\label{eq:iddd}
\Psh_\i\left(\C\right)/E \simeq \Psh_\i\left(\C/E\right).
\end{equation}
It therefore suffices to show that $\cE/E$ is a left exact localization of $\Psh_\i\left(\C\right)/E$. The functor $i$ induces a fully faithful inclusion $$i/E:\cE/E \hookrightarrow \Psh_\i\left(\C\right)/iE$$ coming from $i,$ which is right adjoint to the functor $$a/E:\Psh_\i\left(\C\right)/iE \to \cE/E$$ induced by $a$. Since $a$ is left exact, it readily follows that $a/E$ is.
\end{proof}

\begin{rmk}
Let $0 \le n \le \i$ and let $\cE$ be an $n$-topos. It follows from Proposition \ref{prop:locslice} that for any object $E$ of $\cE,$ the slice category $\cE/E$ is an $n$-topos, called the \textbf{slice $n$-topos} of $E.$
\end{rmk}

\section{Sheaves on $\i$-categories of $\i$-topoi.}
Let $\U$ be the ambient Grothendieck universe of small sets. Fix $\V$ a larger Grothendieck universe such that $\U \in \V,$ whose elements are to be thought of as \emph{large sets}.

\begin{dfn}
Denote by $\LiGpd$ the $\icat$ of \textbf{large $\i$-groupoids} (i.e. $\i$-groupoids in the universe $\V$). For $\C$ an $\icat$, denote by $\LPshi\left(\C\right)$ the $\icat$ $\Fun\left(\C^{op},\LiGpd\right)$ of \textbf{large presheaves}.
\end{dfn}

\begin{dfn}\label{dfn:sheavesoverE}
Let $\cE$ be an $\i$-topos. Denote by $\LShi\left(\cE\right)$ the full subcategory of $\LPshi\left(\cE\right)$ consisting of those $$F:\cE^{op} \to \LiGpd$$ which preserve $\U$-small limits. Denote by $\Shi\left(\cE\right)$ the full subcategory thereof on those $F$ which factor through the inclusion of $\U$-small $\i$-groupoids $$\iGpd \hookrightarrow \LiGpd.$$
\end{dfn}

\begin{rmk}\label{rmk:limitpresrep}
By Remark 6.3.5.17 of \cite{htt}, $\LShi\left(\cE\right)$ may be regarded as an $\i$-topos in $\V.$ Moreover, the Yoneda embedding $$y:\cE \hookrightarrow \Shi\left(\cE\right)$$ is an equivalence, so that in particular, $\Shi\left(\cE\right)$ is an $\i$-topos in $\U.$ If $\cE$ were instead a $1$-topos, then limit preserving functors $$\cE^{op} \to \widehat{\Set}$$ could be identified with precisely those functors which are sheaves with respect to the epimorphism topology as in Definition \ref{dfn:canonical} (which hence must also be the canonical topology). The proof readily generalizes for functors $$\cE^{op} \to \widehat{\mathbf{Gpd}}_n$$ when $\cE$ is instead an $n$-topos if $n < \i,$ but not for $n=\i.$ When $n=\i,$ we do not know if this result still holds.
\end{rmk}



\section{The $\left(\i,2\right)$-category of $\i$-topoi.}\label{sec:lax}
It should be noted that $\i$-topoi more naturally form an $\left(\i,2\right)$-category than an $\left(\i,1\right)$-category, and in many situations it is useful to speak of \emph{non-invertible} natural transformations between geometric morphisms. We will not delve into the technical details of how to make this precise as it would distract too much from the theme of this manuscript. Instead, we will exploit $\left(\i,1\right)$-categorical techniques to get our hands on the non-invertible natural transformations we will need.

In the $\left(\i,2\right)$-category of $\i$-topoi, the $\icat$ of morphisms from $\cE$ to $\cF$ is the $\icat$ $\Geo\left(\cE,\cF\right)$ of Definition \ref{dfn:geofun}. Fix an $\i$-topos $\cB.$ Notice that the composite
$$\T^{op} \stackrel{\Geo\left(\fdot, \cB\right)}{-\!\!\!\longlongrightarrow} \mathcal{C}at_\i \stackrel{\left(\fdot \right)^{op}}{\longlongrightarrow}\mathcal{C}at_\i $$ classifies a Cartesian fibration of $\i$-categories $$p_\cB:\T^{\mathit{lax}}/\cB \to \T.$$ The $\icat$ $\T^{\mathit{lax}}/\cB,$ as the notation suggest, plays the role of the lax slice $\icat$ over $\cB$ in the $\left(\i,2\right)$-category $\i$-topoi. The objects are $\i$-topoi with a distinguished geometric morphism to $\cB$, i.e. the objects have the form $\cE \to \cB$ with $\cE$ an $\i$-topos. The morphisms from $e:\cE \to \cB$ to $f:\cF \to \cB$ consist of geometric morphisms $\varphi:\cE \to \cF$ together with a \emph{not necessarily invertible} natural transformation $$\alpha:\varphi^*f^* \Rightarrow e^*.$$

\chapter{Local Homeomorphisms and \'Etale Maps of $\i$-Topoi}\label{chap:etale}

Topoi and their higher categorical cousins $\i$-topoi are, in a precise sense, categorifications of topological spaces, which allow points to posses intrinsic symmetries, and in the case of $\i$-topoi, allow these symmetries themselves to have symmetries, and so on. In this section, we develop the theory of local homeomorphisms between such generalized spaces. In Chapter \ref{chap:etendues}, we will describe how to glue along such morphisms. This will lead naturally to a theory of generalized orbifolds and Deligne-Mumford stacks. We start with a brief essay on the role of topoi and higher topoi as generalized spaces, which the pragmatic reader is welcome to skip without harm.

\section{Topoi as Generalized Spaces}
The concept of a topos arose out of Grothendieck's work in algebraic geometry in the 1960s. The word ``topos'' itself is derived from the French word ``topologie.'' To quote Grothendieck and Verdier themselves:
\begin{quote}
``Hence, one can say that the notion of a topos arose naturally from the perspective of sheaves in topology, and constitutes a substantial broadening of the notion of a topological space, encompassing many concepts that were once not seen as part of topological intuition... As the term ``topos'' itself is
specifically intended to suggest, it seems reasonable and legitimate to the authors of this seminar to consider the aim of topology to be the study of topoi (not only topological spaces).''(Grothendieck and Verdier \cite{sga4}, p. 302).
\end{quote}
In a precise sense, topoi may be seen as a categorification of the concept of a topological space. For every $0 \le n \le \i$ there is a concept of an $n$-topos (see e.g. \cite{htt}), and each level is a categorification of the previous. In particular, for each pair of integers $m > n,$ the $\left(n+1,1\right)$-category of $n$-topoi embeds full and faithfully into the $\left(m+1,1\right)$-category of $m$-topoi by a functor $$\nu_n:\Tn \hookrightarrow \Tm.$$ Moreover, this functor has a left adjoint $\tau_{n-1},$ which should be thought of as ``decategorification.'' To understand $n$-topoi properly for small values of $n$, it helps to do a bit of ``negative thinking.,'' in the sense of \cite{negative}. In particular, it turns out that $0$-topoi are essentially the same thing as topological spaces.

More precisely, $0$-topoi are the same thing as \emph{locales}. Locales are the objects of study in the playfully named field of ``pointless topology,'' a phrase first coined by John Von Neumann \cite{von}. It was realized that all reasonable topological spaces, specifically \emph{sober} ones (which is quite a large class of spaces, including virtually all spaces considered in practice), are completely determined by their lattice of open sets; the underlying set of points can be recovered up to isomorphism as the set of prime elements of this lattice. The concept of a locale is the essence of space one arrives at after stripping away from the definition of a topological space the requirement for it to have an underlying set of points. Liberating a space from its need for points has many categorical benefits, as the category of locales is much better behaved that the category of topological spaces, particularly in the way subspaces and limits behave. Moreover, locales are merely a mild generalization of topological spaces, as every locale is a sublocale of a topological space \cite{atomless}.

At the first level of categorification, namely that of $1$-topoi, the situation becomes an even more non-trivial generalization of topological spaces than locales provide. Indeed, expanding the of the notion of ``space'' to include topoi, was indispensable to the inception of the infamous \emph{\'etale cohomology} of arithmetic schemes, used to prove the Weil conjectures. Given a scheme $X,$ the ``space'' whose cohomology is the \'etale cohomology of $X$ is actually a topos. From the mouth of Grothendieck:
\begin{quote}
``The crucial thing here, from the viewpoint of the Weil conjectures, is that the new notion [of space] is vast enough, that we can associate to each scheme a `generalized space' or `topos' (called the `\'etale topos' of the scheme in question). Certain `cohomology invariants' of this topos (`childish' in their simplicity!) seemed to have a good chance of offering `what it takes' to give the conjectures their full meaning, and (who knows!) perhaps to give the means of proving them.'' \cite{rec} p. 41
\end{quote}
Indeed, for \'etale cohomology, one needs to compute cohomology with respect to \'etale covers, and \'etale covers of a scheme $X$ may viewed, in a precise sense, as local homeomorphisms into the small \'etale topos $\Sh\left(X_{\et}\right).$ This would be impossible to achieve if the role of $\Sh\left(X_{\et}\right)$ was played by a topological space, or locale. 

Topoi however contain much more information than merely their cohomology, just like spaces. Indeed, the functor $$\Sh=\nu_0:\mathfrak{Top}_0 \hookrightarrow \mathfrak{Top}_1$$ from locales to topoi, restricts to a full and faithful embedding of the category of (sober) topological spaces into the $\left(2,1\right)$-category of topoi. Moreover, it was shown in \cite{ext} that every topos arises as the quotient of a locale by certain local symmetries (encoded by a groupoid), where the quotient is taken ``up to homotopy.'' Hence, one may think of a topos $\cE$ as a generalized locale, such that generalized points, i.e. morphisms $$\cL \to \cE$$ from a locale $\cL,$ can have automorphisms. Similarly, an $\i$-topos should be regarded as a categorification of a topological space (or locale) in which not only generalized points can have automorphisms, but also these automorphisms themselves can have automorphisms, as well as the automorphisms between these automorphisms and so on.

\begin{rmk}
The $\left(2,1\right)$-category of topoi is equivalent to ``\'etale complete localic stacks'' as shown in \cite{cont} and \cite{etalspme}.
\end{rmk}

\section{Local homeomorphisms, sheaves, and \'etale  maps}
Recall that for a given topological space $X,$ the \emph{\'etal\'e space construction} yields an equivalence
\begin{equation}\label{eq:etsp}
\xymatrix{\Sh\left(X\right) \ar@<-0.5ex>[r]_-{L} & \operatorname{Top}^{\mathit{loc. h.}}/X \ar@<-0.5ex>[l]_-{\Gamma}},
\end{equation}
between the category of sheaves over $X$ and the category of local homeomorphisms over $X$. If we choose to identify the space $X$ with the same space viewed as a topos, namely $\Sh\left(X\right)$, the equivalence (\ref{eq:etsp}) in particular tells us that local homeomorphisms into $\Sh\left(X\right)$ are in bijection with objects of $\Sh\left(X\right).$ Moreover, ones has that for a sheaf $$F \in \Sh\left(X\right),$$ there is a canonical equivalence $$\Sh\left(L\left(F\right)\right) \simeq \Sh\left(X\right)/F,$$
where  $\Sh\left(L\left(F\right)\right)$ is the topos of sheaves over the \'etal\'e space of $F$ and $\Sh\left(X\right)/F$ is the slice topos over $F.$ Furthermore, for any topos $\cE$ (or any higher topos), for each object $E \in \cE,$ $\cE/E$ is again a (higher) topos and there is a canonically induced geometric morphism $\pi:\cE/E \to \cE$,
\begin{equation*}\label{eq:ettt}
\xymatrix{\cE/E \ar@<-0.5ex>[r]_-{\pi_*} & \cE \ar@<-0.5ex>[l]_-{\pi^*}},
\end{equation*}
such that $\pi^*$ has a left adjoint
\begin{eqnarray*}
\pi_!:\cE/E &\to& \cE\\
f:F \to E &\mapsto& F
\end{eqnarray*} (see \cite{htt}, Proposition 6.3.5.1). In the case above, this geometric morphism is the one induced by the local homeomorphism $L\left(F\right) \to X.$

This motivates the following definitions:

\begin{dfn}\label{dfn:metale}
A geometric morphism $\cF \to \cE$ between $n$-topoi, with $0 \le n \le \i$ is \textbf{\'etale}, if it is equivalent to one of the form $$\cE/F \to \cE$$ for some object $F \in \cE.$ We denote the subcategory of $\Tn$ of all $n$-topoi whose morphisms are \'etale geometric morphisms by $\Tn^{\et}.$ \emph{If} $m \le n-1$ we will call such a morphism \textbf{$m$-\'etale} if, it is equivalent to an \'etale morphism of the form  $$\cE/E \to \cE,$$ with $E$ an $m$-truncated object.
We denote the subcategory of $\Tn$ of all $n$-topoi whose morphisms are $m$-\'etale geometric morphisms by $\Tn^{m-\et}.$
\end{dfn}

It is clear that there are intimate connections between the concepts of sheaves, \'etale maps, and local homeomorphisms.

\begin{rmk}\label{rmk:0truncov}
An \'etale morphism $$\cE \to \Sh_\i\left(X\right),$$ with $X$ a topological space, represents a sheaf of sets if and only if $\cE$ is equivalent to a space, if and only if the morphism is $0$-\'etale. A completely analogous statement holds for locales.
\end{rmk}

\begin{rmk}\label{rmk:noetale}
Every object $E$ of an $n$-topos is $\left(n-1\right)$-truncated. It follows that if $\cF \to \cE$ is an \'etale map of $n$-topoi, then it must be $\left(n-1\right)$-\'etale.
\end{rmk}

\begin{prop}(\cite{dag} Proposition 2.3.16)\label{prop:2.3.16}\\
Let $\cE$ be an $n$-localic $\i$-topos and let $E$ be an object of $\cE.$ Then $\cE/E$ is $n$-localic if and only if $E$ is $n$-truncated.
\end{prop}

The word ``\'etale'' has enjoyed interchangeable use with the phrase ``local homeomorphism'' for quite some time, motivated mathematically by the \'etal\'e-space construction. Johnstone points out in \cite{elephant2} that the same construction etymologically gave rise to the phrasing ``\'etale,'' due to a mistranslation of the French word ``\'etal\'e,'' and he hence prefers to obliterate its usage and refer only to local homeomorphisms. Despite this, the definition (and terminology) of an \'etale geometric morphism of $n$-topoi in Definition \ref{dfn:metale} is already standard for $n=1$ (notwithstanding Johnstone's protest) and $n=\infty$ \cite{htt}. Although, we do agree with Johnstone that it is unfortunate that the exact phrasing ``\'etale'' arose as a linguistic error, we do believe the that the concept of an \'etale morphism of topoi should stand on its own as an important concept (particularly due to Proposition \ref{lem:htt6.3.5.10}), independently from the concept of a local homeomorphism, and hence deserves its own name. This is largely due to the fact that we \emph{not} believe that an \'etale morphism is the correct embodiment of the concept of a local homeomorphism of $n$-topoi (or even $1$-topoi!) unless $n=\infty.$

For example, if $L$ is a $0$-topos (i.e. a locale), as a category, $L$ is a poset of ``open sets.'' An \'etale map of $0$-topoi is a map induced by slicing, so it can only correspond to an inclusion of an open sublocale. Although such maps are local homeomorphisms, there are many more local homeomorphisms than open inclusions! The problem here is that a local homeomorphism of locales is one induced by slicing their associated \emph{1}-topoi. This problem persists in higher categorical dimensions. Let $\cE$ be an $n$-topos and let $E$ be an $n$-truncated object of the associated $n$-localic $\left(n+1\right)$-topos $\nu_n\left(\cE\right),$ which is not $\left(n-1\right)$-truncated. (So in particular, $E$ cannot be identified with an element of $\cE$). The \'etale map
\begin{equation}\label{eq:emmm}
\nu_n\left(\cE\right)/E \to \nu_n\left(\cE\right),
\end{equation}
by virtue of Proposition \ref{prop:2.3.16}, is a geometric morphism between $n$-localic $\left(n+1\right)$-topoi. Since the functor
\begin{equation}\label{eq:inc}
\nu_n:\Tn \hookrightarrow \mathfrak{Top}_{n+1}
\end{equation}
is full and faithful (with essential image the $n$-localic $\left(n+1\right)$-topoi), it follows that the morphism (\ref{eq:emmm}) corresponds to a geometric morphism
$$\cF \to \cE$$
of $n$-topoi, which is not \'etale as a map of $n$-topoi (since it is not induced by slicing). This implies that the term ``local homeomorphism'' is unfitting to describe \'etale maps, since the point of the functor (\ref{eq:inc}) is to view $n$-topoi as $\left(n+1\right)$-topoi, and if we are to view higher topoi as generalized spaces, the concept of a local homeomorphism of higher topoi should not depend on how we choose to ``view'' a higher topos. We hence make the following non-standard definition:

\begin{dfn}\label{dfn:lochomeo}
A geometric morphism $$f:\cF \to \cE$$ between $n$-topoi is a \textbf{local homeomorphism} if $$\nu_n\left(f\right):\nu_n\left(\cF\right) \to \nu_n\left(\cE\right)$$ is \'etale.
\end{dfn}
We are of the opinion that such a map is the correct categorification of a local homeomorphism, since in particular this reproduces the correct notion when $n=0.$ Moreover, this definition is invariant under the embeddings $\nu_m.$ For $n=1,$ every \'etale geometric morphism is a local homeomorphism of $1$-topoi, but the converse is not true. However, we think that the broader class of geometric morphisms introduced do deserved to be called local homeomorphisms. We also propose the terminology \textbf{$n$-\'etale} map. There is no danger here, since in Definition \ref{dfn:metale} we only defined $m$-\'etale maps of $n$-topoi for $m \le n-1$. Moreover, with this convention, there is a bijection between equivalence classes of $n$-\'etale morphisms into $\cE$ and $n$-\'etale morphisms into the $\i$-topos $\nu^\i_n\left(\cE\right)$ in the sense of Definition \ref{dfn:metale}. Finally, notice that \'etale morphisms and local homeomorphisms of \emph{infinity} topoi are one and the same.

\begin{rmk}
When $\cE$ is a topological \'etendue (a special type of $1$-topos), the concept of a local homeomorphism in Definition \ref{dfn:lochomeo} agrees with that of \cite{etalspme}.
\end{rmk}

\begin{rmk}\label{rmk:open}
A (sober) topological space is determined by its lattice of open subsets. A similar statement holds for $\i$-topoi, once ``open subset'' is correctly interpreted. One can already see from the notion of Grothendieck topology that for topoi, more general notions of cover exist than ones generated by the inclusion of open subsets. For $X$ a (sober) topological space, open subsets are in bijection with $\left(-1\right)$-\'etale morphisms into $\Shi\left(X\right),$ so it follows that one can reconstruct $X$ from knowing all such $\left(-1\right)$-\'etale morphisms. However, if $\cE$ is an $\i$-topos, then one can only reconstruct the localic reflection of $\cE$ from its collection of $\left(-1\right)$-\'etale morphisms- that is, one can only find the best approximation of $\cE$ by a locale. One needs more general notions of ``open subset'' in order to reconstruct $\cE$ from its open subsets. If $\cE$ happens to be $1$-localic, then one can reconstruct $\cE$ out of all the $0$-\'etale maps into $\cE,$ but if not, this would only allow one to reconstruct its $1$-localic reflection. This pattern continues. For a general $\i$-topos $\cE$, one would need to consider all \'etale morphisms into $\cE$ as ``open subsets.''
\end{rmk}

\begin{dfn}\label{def:chi}
Given an $\i$-topos $\cE,$ there is a canonical map $$\chi_\cE:\cE \to \T^{\et}/\cE$$ which sends an object $E \in \cE$ to the canonical \'etale morphism $\cE/E \to \cE.$
\end{dfn}

The following lemma generalizes the equivalence (\ref{eq:etsp}):

\begin{prop}\label{lem:htt6.3.5.10} (\cite{htt} Remark 6.3.5.10)
For every $\i$-topos $\cE,$ the functor $\chi_\cE$ is an equivalence.
\end{prop}

For an $n$-topos $\cE,$ one may analogously define a functor $$\chi^n_\cE:\cE \to \Tn^{\et}/\cE.$$

\begin{cor}\label{cor:xin}
For every $n$-topos $\cE,$ the functor $\chi^n_\cE$ is an equivalence.
\end{cor}
\begin{proof}
$\cE$ is equivalent to the full subcategory of its associated $\i$-topos $\nu^\i_n\left(\cE\right)$ spanned by the $\left(n-1\right)$-truncated objects, and the full subcategory of $\T^{\et}/\cE$ spanned by the $\left(n-1\right)$-truncated objects is equivalent to $\Tn^{\et}/\cE,$ so the result follows from Proposition \ref{lem:htt6.3.5.10}.
\end{proof}

\begin{cor}\label{cor:etalspn}
For any $\i$-topos $\cE,$ the slice category $\T^{\left(n-1\right)-\et}/\cE$ is canonically equivalent to the $n$-localic reflection of $\cE$.
\end{cor}

\begin{rmk}
By an analogous proof, one has that for any $n$-topos $\cE,$
\begin{equation}\label{eq:genetsp}
\nu^{n+1}_n\left(\cE\right) \simeq \Tn^{\mathit{loc\mspace{3mu} h.}}/\cE,
\end{equation}
where $\Tn^{\mathit{loc\mspace{3mu} h.}}$ is the subcategory of $\Tn$ spanned by local homeomorphisms, and $\nu^{n+1}_n$ is as in Definition \ref{dfn:nloc}. Note that for $n=0$ we have $$\nu_0=\Sh:\mathfrak{Top}_0=\mathfrak{Loc} \to \mathfrak{Top}_1$$ is the inclusion of locales into topoi by taking sheaves, so (\ref{eq:genetsp}) collapses to the classical ``\'etal\'e-space'' equivalence. Therefore, Corollary \ref{cor:etalspn} may be viewed as a categorification of the \'etal\'e-space equivalence for topoi which are not spatial; that is to say: the underlying $\left(n+1,1\right)$-category of the embodiment of an $n$-topos $\cE$ as an $\left(n+1\right)$ topos is equivalent to the $\left(n+1,1\right)$-category of local homeomorphisms over $\cE$.

The equivalence (\ref{eq:genetsp}) also implies the equivalence of Corollary \ref{cor:xin}. For $n=\infty$, these are both the same statement. When $n \ne \infty,$ then  we have that $\cE \simeq \tau_{n-1}\left(\nu^{n+1}_n\cE\right)$ which in turn is equivalent to the full subcategory of $\nu^{n+1}_n\left(\cE\right) \simeq  \mathfrak{Top}^{n-\et}_n/\cE$ spanned by $\left(n-1\right)$-truncated objects, and the $\left(n-1\right)$-truncated objects in $\mathfrak{Top}^{n-\et}_n/\cE$ are precisely the \'etale ones, so this recovers Corollary \ref{cor:xin}. This also explains the ``miracle'' that $\Tn^{\et}/\cE$ (which a priori should be an $\left(n+1,1\right)$-category) is actually an $\left(n,1\right)$-category, since it consists of $\left(n-1\right)$-truncated objects.
\end{rmk}

\begin{ex}
Let $X$ be topological space or locale (i.e. a $0$-topos). The embodiment of $X$ as an $\i$-topos is $\Sh_\i\left(X\right).$ Let $$f:\cF \to\Sh_\i\left(X\right)$$ be an \'etale map. In light of Proposition \ref{lem:htt6.3.5.10}, we may identify $f$ with a sheaf $F$ of $\i$-groupoids on $X.$ The \'etale map $f$ is $n$-\'etale if and only if $F$ is actually a sheaf of $n$-groupoids. When $n=0,$ this is explained by Remark \ref{rmk:0truncov}.
\end{ex}

\begin{rmk}
Let $\cE$ be an $n$-topos for $n <\i,$ and let $F$ be a sheaf of $\i$-groupoids on the canonical site $\left(\cE,can\right).$ Then $F$ is representable by a local homeomorphism $$f:\cF \to \cE$$ with $\cF$ an $n$-topos if and only if $F$ is a sheaf of $n$-groupoids. When $n=0,$ this is explained by Remark \ref{rmk:0truncov}.
\end{rmk}

\begin{prop} (\cite{htt} Proposition 6.3.5.8) \label{prop:6358}
Let $E \in \cE$ be an object in an $\i$-topos and let $$g:\cF \to \cE$$ be a geometric morphism. Then the following is a pullback diagram in $\T$:
$$\xymatrix{\cF/g^*\left(E\right) \ar[r] \ar[d] & \cE/E \ar[d] \\
\cF \ar[r]^-{g} & \cE.}$$
\end{prop}

\begin{cor}\label{cor:pullstabet}
\'Etale geometric morphisms are stable under pullback in $\T.$
\end{cor}

\begin{lem}\label{lem:slicek}
Suppose that $\C$ is an $\infty$-category and $f:C \to D$ is in $\C_1$. Then the canonical map $$\pi:\left(\C/D\right)/f \to \C/C$$ is a trivial Kan fibration.
\end{lem}
\begin{proof}
Note that $\left(\C/C\right)_n$ is by definition the set of $\left(n+1\right)$-simplices $$\tau:\Delta^{n+1} \to \C$$ such that $\tau\left(n+1\right)=C.$ We may identify $\left(\left(\C/D\right)/f\right)_n$ with the set of $\left(n+2\right)$-simplices $$\sigma:\Delta^{n+2} \to \C$$ such that $d_{n+2}\left(\sigma\right) \in \left(\C/C\right)_n$ and $\sigma|_{\Delta^{\{n+1,n+2\}}}=f.$ The canonical map $\pi$ at level $n$ sends $\sigma$ to $d_{n+2}\left(\sigma\right).$ Suppose that we have a diagram
$$\xymatrix{\partial \Delta^n \ar[d] \ar[r]^-{\beta} & \left(\C/D\right)/f \ar[d]
^-{\pi}\\ \Delta^n \ar[r]_-{\alpha} &\C/C.}$$
We we want to show there exists a diagonal filler. Notice that $\alpha$ corresponds to an $\left(n+1\right)$-simplex $$\tilde \alpha: \Delta^{n+1} \to \C.$$ Furthermore, each face of $\alpha,$ $$d_i\left(\alpha\right):\Delta^{n-1} \to \C/C$$ has a lift induced by $\beta$, $$\beta_i:\Delta^{n-1} \to \left(\C/D\right)/f,$$ and each of these lifts are compatible on their overlap, by construction. Each $\beta_i$ corresponds to an $\left(n+1\right)$-simplex $$\tilde \beta_i: \Delta^{n+1} \to \C.$$ Observe that the collection of $\left(n+1\right)$-simplices $\left(\tilde \alpha,\tilde \beta_0,\ldots,\tilde \beta_n\right)$ assembles to a $\Lambda^{n+2}_{n+1}$ horn in $\C$. Let $H$ be any filler to a $\left(n+2\right)$-simplex. Then $H$ can be identified with an $n$-simplex of $\left(\C/D\right)/f$ which is a lift of $\alpha.$
\end{proof}

\begin{cor}\label{cor:compet}
The composition of two \'etale geometric morphisms is again \'etale.
\end{cor}

\begin{prop}(\cite{htt}, Corollary 6.3.5.9)\label{prop:6.3.5.9}
Suppose that a composite of geometric morphisms $$\cE \stackrel{f}{\to} \cF \stackrel{g}{\to} \G$$ is \'etale, and that $g$ \'etale. Then $f$ is \'etale.
\end{prop}

\section{The \'etale topology on $\i$-topoi.}
We will now show that the (necessarily large) $\icat$ $\T$ of $\i$-topoi caries a natural subcanonical Grothendieck topology. The role of this Grothendieck topology will be to have an elegant framework in which to speak about gluing together $\i$-topoi along local homeomorphisms.

\begin{dfn}
Let $\cE$ be an $\i$-topos. Suppose that $$\left(f_\alpha:\cE_\alpha \to \cE\right)$$ is a collection of \'etale geometric morphisms, which under the equivalence $$\T^{\et}/\cE \simeq \cE,$$ correspond to objects $\left(E_\alpha\right)$ of $\cE.$ Declare the collection of maps to be an \textbf{\'etale covering family of $\cE$} if the essentially unique morphism $$\coprod\limits_\alpha E_\alpha \to 1_\cE$$ is an effective epimorphism. We will call it an \textbf{$n$-\'etale covering family of $\cE$} if each of the morphisms $f_\alpha$ are $n$-\'etale.
\end{dfn}




\begin{prop}\label{prop:etexists}
The assignment to each $\i$-topos $\cE$ its \'etale covering families defines a Grothendieck pretopology on $\T.$ Similarly for $n$-\'etale covering families.
\end{prop}

\begin{proof}
$\mbox{ }$

\begin{itemize}
\item[i)] Every equivalence is clearly a singleton covering family.

\item[ii)] Suppose now that $g:\cF \to \cE$ and that $$\left(f_\alpha:\cE_\alpha \to \cE\right)$$ is an \'etale covering family, which each $f_\alpha$ corresponding to objects $E_\alpha$ of $\cE$. Then by Proposition \ref{prop:6358}, each induced \'etale morphism $$g^*\left(f_\alpha\right):\cF \times_{\cE} \cE_\alpha \to \cF$$ corresponds to the object $g^*\left(E_\alpha\right)$ of $\cF.$ By \cite{htt} 6.5.1.16, $g^*$ preserves effective epimorphisms. Hence
\begin{eqnarray*}
g^*\left(\coprod_\alpha E_\alpha \to 1_\cE\right) &\simeq& g^*\left(\coprod_\alpha E_\alpha\right) \to g^*\left(1_\cE\right)\\
&\simeq& \coprod_\alpha g^*\left(E_\alpha\right) \to 1_\cF
\end{eqnarray*}
is an effective epimorphism. It follows that the collection $$\left(g^*\left(f_\alpha\right):\cF \times_{\cE} \cE_\alpha \to \cF\right)$$ is an \'etale covering family. For the case of $n$-\'etale covering families, one only needs to observe that $g^*$ preserves $n$-truncated objects, since it is left exact (\cite{htt} Proposition 5.5.6.16).

\item[iii)] Suppose that $$\left(f_\alpha:\cE_\alpha \to \cE\right)$$ is an \'etale covering family and that for each $\alpha,$ we have an \'etale covering family $$\left(f^\beta_\alpha:\cE^\beta_{\alpha} \to \cE_\alpha\right).$$ We want to show that the composite $$\left(\cE^\beta_{\alpha} \to \cE\right)$$ is an \'etale covering family. Now, each $f_\alpha$ corresponds to an object $E_\alpha \in \cE$ such that $\cE/E_\alpha \simeq \cE_\alpha,$ and each $f^\beta_\alpha$ corresponds to an object of $\cE/E_\alpha,$ i.e. a map $g^\beta_\alpha:E^\beta_\alpha \to \cE_\alpha$ in $\cE.$ Since by Lemma \ref{lem:slicek}, $$\left(\cE/E_\alpha\right)/g^\beta_\alpha \simeq \cE/E^\beta_\alpha,$$ it suffices to show that the canonical map $$\coprod\limits_{\beta,\alpha} E^\beta_\alpha \to 1_\cE$$ is an effective epimorphism. Notice that the canonical map $$\coprod\limits_\beta g^\beta_\alpha \to id_{E_\alpha}$$ in $\cE/E_\alpha$ is an effective epimorphism, and since the forgetful functor $$\cE/E_\alpha \to \cE$$ preserves colimits, by Lemma 6.2.3.13 of \cite{htt}, it follows that, for all $\alpha,$ $$\coprod\limits_\beta g^\beta_\alpha :\coprod\limits_\beta E^\beta_\alpha \to E_\alpha$$ is an effective epimorphism. By Lemma 6.2.3.13 of \cite{htt}, it follows that $$\coprod\limits_{\alpha,\beta} g^\beta_\alpha :\coprod\limits_{\alpha,\beta} E^\beta_\alpha \to \coprod_\alpha E_\alpha$$ is an effective epimorphism. Since composites of effective epimorphisms are effective epimorphisms (Corollary 6.2.3.12 in \cite{htt}) we are done. For the $n$-\'etale topology, observe that if each $E_\alpha$ are $n$-truncated, and each $$g^\beta_\alpha:E^\beta_\alpha \to E_\alpha$$ are as well, then by \cite{htt}, Lemma 5.5.6.14, it follows that each $E^\beta_\alpha$ are $n$-truncated.
\end{itemize}
\end{proof}

\begin{dfn}
The Grothendieck topology on $\T$ generated by \'etale covering families is called the \textbf{\'etale topology}. We will denote the Grothendieck site by $\left(\T,\et\right).$ Similarly, the Grothendieck topology generated by $n$-\'etale covering families will be called the \textbf{$n$-\'etale topology}, and will be denoted by $\left(\T,n\mbox{-}\et\right).$
\end{dfn}

\begin{prop}\label{prop:topsubcan}
The \'etale topology is subcanonical.
\end{prop}

\begin{proof}
Let $\left(\cE_\alpha \to \cE\right)$ be an \'etale covering family of an $\i$-topos $\cE.$ It suffices to show that $\cE$ is the colimit of the \v{C}ech nerve $$C_\bullet:\Delta^{op} \to \T$$ of the induced map
\begin{equation}\label{eq:cov}
\coprod\limits_{\alpha} \cE_\alpha \to \cE,
\end{equation}
i.e. it suffices to show that (\ref{eq:cov}) is an effective epimorphism. Notice that by Corollary \ref{cor:pullstabet}  there is a canonical lift
$$\xymatrix@R=0.1in{\Delta^{op} \ar[r]^-{C_\bullet} \ar@{-->}[rdd]_-{\tilde C_\bullet}& \T\\
& \T^{\et} \ar[u]\\
& \T^{\et}/\cE. \ar[u]}$$
Under the equivalence $\T^{\et}/\cE \simeq \cE,$ $\tilde C_\bullet$ corresponds to the \v{C}ech nerve of
\begin{equation}\label{eq:episs}
\coprod\limits_{\alpha} E_\alpha \to 1_\cE,
\end{equation}
where each $E_\alpha$ is the element of $\cE$ corresponding to $$\cE_\alpha \to \cE.$$ By assumption, (\ref{eq:episs}) is an effective epimorphism, hence $$\colim \tilde C_\bullet \simeq id_\cE.$$
The forgetful functor $\T^{\et}/\cE \to \T^{\et}$ is colimit preserving, and so is the canonical functor $$\T^{\et} \to \T,$$ by \cite{htt} Theorem 6.3.5.13. It follows that $$\colim C_\bullet \simeq\cE$$ as desired.
\end{proof}

\begin{cor}
The $n$-\'etale topology on $\T$ is subcanonical for any $n$.
\end{cor}

\begin{rmk}
The \'etale topology on $\T$ naturally restricts to $\T^{\et}$. It is again subcanonical by virtue of \cite{htt}, Theorem 6.3.5.13. Moreover, for any $\i$-topos $\cE,$ the \'etale topology induces a Grothendieck topology on $\T^{\et}/\cE$, which under the equivalence $$\T^{\et}/\cE \simeq \cE$$ coincides with the epimorphism topology on $\cE.$
\end{rmk}

\begin{rmk}
One may restrict the $n$-\'etale Grothendieck pretopology on $\T$ to $m$-localic $\i$-topoi in the obvious way and get an induced $n$-\'etale Grothendieck topology on $\mathfrak{Top}_m.$ For $m=n-1,$ this topology can be described by declaring a collection of \'etale geometric morphisms of $m$-topoi $$\left(f_\alpha:\cE_\alpha \to \cE\right),$$ which under the equivalence $$\T^{\et}/\cE \simeq \cE,$$ correspond to objects $\left(E_\alpha\right)$ of $\cE,$ to be a covering family if $$\coprod\limits_\alpha E_\alpha \to 1_\cE$$ is an effective epimorphism. Call this Grothendieck topology the \textbf{\'etale topology} on $\mathfrak{Top}_m.$ Notice that the \'etale topology on the category $\mathfrak{Top}_0$ of locales is the standard topology whose covering families are given by open coverings.
\end{rmk}

\begin{lem}\label{lem:lalalucas}
Let $n < \i$ and let $\cE$ be an $n$-localic $\i$-topos and $F \in \cE$ any object. The $\cE/F$ admits an \'etale covering family $\left(\cF_\beta \to \cE/F\right)$ such that each composite $$\cF_\beta\ \to \cE/F \to \cE$$ is $\left(n-1\right)$-\'etale.
\end{lem}

\begin{proof}
Choose an $n$-site $\left(\C,J\right)$ for $\cE.$ Then, by \cite{htt} Proposition 5.1.5.3, $F$ is the colimit of $$\C/F \to \C \stackrel{y}{\longrightarrow}\cE,$$ where $y$ denotes the Yoneda functor into $\i$-sheaves. By Proposition 6.2.3.13 of \cite{htt}, it follows that the canonical map $$\coprod\limits_{y\left(C_\beta\right) \to F} y\left(C_\beta\right) \to F$$ is an effective epimorphism. Notice that each object $y\left(C_\beta\right)$ is $\left(n-1\right)$-truncated, since $\C$ is an $\left(n,1\right)$-category. The collection of maps $\left(\cE/y\left(C_\beta\right) \to \cE/F\right)$ now constitutes an \'etale covering family of $\cE/F$ with the desired property.
\end{proof}

\begin{prop}\label{prop:highlandertop}
For any $m \ge n-1,$ the $m$-\'etale topology on $\Tn$ is the same as the \'etale topology.
\end{prop}

\begin{proof}
It suffices to show that for any $n$-localic $\i$-topos $\cE$, any \'etale covering family $$\left(f_\alpha:\cE_\alpha \to \cE\right),$$ can be refined by an $\left(n-1\right)$-\'etale covering family. Each map $f_\alpha$ corresponds to an object $E_\alpha \in \cE.$ By Lemma \ref{lem:lalalucas}, for each $\alpha,$ there exists an \'etale covering family $\left(\cF^\beta_\alpha \to \cE_\alpha\right)$ such that each composite $\cF^\beta_\alpha \to \cE_\alpha \to \cE$ is $\left(n-1\right)$ \'etale. By Proposition \ref{prop:etexists}, $$\left(\cF^\beta_\alpha \to \cE\right)_{\alpha,\beta}$$ is now an $\left(n-1\right)$-\'etale covering family refining $$\left(f_\alpha:\cE_\alpha \to \cE\right)_\alpha.$$
\end{proof}

We conclude that the \'etale topology on $\Tn$ is the natural Grothendieck topology to use when viewing $n$-topoi as generalized spaces.


\chapter{Structured $\i$-Topoi}\label{chap:structured}
This section is meant to give some intuition behind the definitions of \emph{geometries} and \emph{geometric structures} as defined in \cite{dag}. These serve as a means to provide a general framework for describing the appropriate ambient $\i$-categories of $\i$-topoi with structure sheaves of various types of algebraic objects (e.g. commutative rings, differential graded commutative algebras, commutative ring spectra, etc.), and a suitable notion of such a structure sheaf having ``local stalks.'' This framework will be used in an essential way in the rest of this manuscript. However, if the reader is already familiar with the framework of geometries and geometric structure, or is uninterested in pedagogical text, they may simply skip to Section \ref{sec:etstr}.

\section{Structure Sheaves and Classifying Topoi}
In geometry, one studies objects which have more structure than merely that of a topological space. Typically this extra structure can be encoded succinctly by a structure sheaf. For example, in algebraic geometry, the classical approach to schemes is by modeling them as locally ringed spaces, i.e. topological spaces with a sheaf of rings (with local stalks). Although typically smooth manifolds are described by local charts, the category of smooth manifolds may also be realized as a full subcategory of locally ringed spaces, and this sheaf theoretic point of view becomes more essential when working with complex manifolds or supermanifolds. Still more exotic variants exist for the modern \emph{derived} geometer, where they can study various flavors of derived schemes which can be described by topological spaces with a sheaf of differential graded algebras, simplicial rings, or $\mathbf{E_\i}$-ring spectra, etc.

When using the \'etale topology on affine schemes, one should not think of a scheme as a locally ringed \emph{space}, at least not in the naive sense; the ``space'' underlying the affine scheme associated to a commutative ring $A$ whose cohomology yields \'etale cohomology is actually a topos--- the small \'etale topos $\Sh\left(A_{\acute{e}t}\right)$ of $A.$ Moreover, the small \'etale topos $\Sh\left(A_{\acute{e}t}\right)$ of $A$ carries a canonical ring object, i.e. a sheaf $\O_A$ of rings, which may be interpreted as a structure sheaf. Indeed, the appropriate category of such ringed topoi of the form $\left(\Sh\left(A_{\acute{e}t}\right),\O_A\right)$ is equivalent to the category of affine schemes, and the assignment $$A \mapsto \left(\Sh\left(A_{\acute{e}t}\right),\O_A\right)$$ should be thought of as a replacement of the classical spectrum functor $\mathbf{Spec}$ appropriate for working with the \'etale topology. A similar such spectrum functor exists for the \'etale topology for various flavors of \emph{derived} affine schemes, except in order to properly encode the structure sheaves, it is most convenient to replace the role of topoi with that of $\i$-topoi, see e.g. \cite{dag}.


A common theme in studying spaces with \emph{geometric structure} as spaces with a structure sheaf is that all the structure sheaves we use are sheaves with some sort of ``\emph{algebraic} structure''. To make this statement precise, one can speak of models for theories; quite a versatile sea of possible ``algebraic structures'' is provided by \emph{geometric theories} (and their higher categorical analogues). Roughly speaking, a geometric theory $\mathbb{T}$ is a certain first-order theory in logic (but possibly with infinitary disjunctions), and it can have models in any topos. E.g., there is a geometric theory of rings, and the category of models for this theory in the topos $\Set$ is the classical category of rings. This framework allows one to easily speak about structure sheaves with algebraic structure. For example, specifying a model for the geometric theory of rings in $\Sh\left(X\right),$ the topos of sheaves on a topological space, is the same as specifying a sheaf of rings on $X.$

Moreover, given a model $M$ of a geometric theory $\mathbb{T}$ in a topos $\cE,$ and a geometric morphism $$f:\cF \to \cE,$$ $f^*\left(M\right)$ is a model in $\cF,$ in other words the assignment to each topos the category of models for $\mathbb{T}$ is functorial. The functoriality of the assignment of models can more precisely be stated as follows:\\
\\
Given a geometric theory $\mathbb{T},$ it determines a functor
\begin{eqnarray*}
\mathfrak{Top}^{op} &\to& \mathbf{Cat}\\
\cE &\mapsto& \mathcal{M}\mbox{od}_\mathbb{T}\left(\cE\right)
\end{eqnarray*}
from the $2$-category of topoi to the $2$-category of (possibly large) categories, which sends a topos to the category of models for the theory $\mathbb{T}$ in $\cE.$ In fact, geometric theories are precisely those first order theories which enjoy this functoriality of models with respect to geometric morphisms. Moreover, for any geometric theory $\mathbb{T}$, it turns out that the functor $\mathcal{M}\mbox{od}_\mathbb{T}\left(\mspace{3mu}\cdot\mspace{3mu}\right)$ is always \emph{representable} by a topos $\mathcal{C}_\mathbb{T},$ called the classifying topos of $\mathbb{T}.$ Put more concretely, there is a unique topos $\mathcal{C}_\mathbb{T}$ such that for any other topos $\cE$, models for $\mathbb{T}$ in $\cE$ are the same as geometric morphisms $$\cE \to \mathcal{C}_\mathbb{T}.$$ In particular, for a topological space $X$, geometric morphisms $$\Sh\left(X\right) \to \mathcal{C}_\mathbb{T}$$ are the same as ``sheaves of $\mathbb{T}$-models'' on $X$. Moreover, for two geometric theories $\mathbb{S}$ and $\mathbb{T},$ morphisms of geometric theories $$\mathbb{S} \to \mathbb{T}$$ are the same as morphisms of topoi $$\mathcal{C}_\mathbb{T} \to \mathcal{C}_\mathbb{S},$$ which in turn are the same as $\mathbb{S}$-models in $\mathcal{C}_\mathbb{T}.$ Conversely, given any topos $\cF$, there is a unique geometric theory $\mathbb{T}_\cF$ with the property that maps into to $\cF$ classify $\mathbb{T}_\cF$-models. In a nutshell: there is a duality between the $2$-category of geometric theories and the $2$-category of topoi. (See e.g. \cite{elephant2}, Section D.1.)

Let us return to geometry with an important illustrative example:
\begin{ex}\label{ex:classrings}
Consider the opposite category $\left(\mathbf{fp\mbox{\textendash} rings}\right)^{op}$ of finitely presented commutative rings. In general, for any small category $\C$ with small limits and any topos $\cE$, there is a canonical equivalence of categories between the category of left exact functors $F:\C \to \cE$ and the category of geometric morphisms $\Geo\left(\cE,\Set^{\C^{op}}\right)$. Indeed, given such a left exact functor $F$, its left Kan extension along the Yoneda embedding $y:\C \to \Set^{\C^{op}},$ $$\Lan_y F:\Set^{\C^{op}} \to \cE$$ is left exact and has a right adjoint, and conversely, given a geometric morphism $f:\cE \to \Set^{\C^{op}},$ $f^* \circ y:\C \to \cE$ is left exact. Now, the category $\left(\mathbf{fp\mbox{\textendash} rings}\right)^{op}$ is the category freely generated under finite limits by its canonical commutative ring object $\mathbb{Z}[X]$, so that for any category $\sD$ with finite limits (for instance a topos), there is an equivalence of categories between left exact functors $$\left(\mathbf{fp\mbox{\textendash} rings}\right)^{op} \to \sD$$ and the category of commutative ring objects in $\sD$. Putting this all together, what this means is, the presheaf topos $\Set^{\left(\mathbf{fp\mbox{\textendash} rings}\right)}$ is the \emph{classifying topos for the theory of commutative rings}.\\
\newline
In particular, we know that for any topological space $X$, a geometric morphism $$\Sh\left(X\right) \to \Set^{\left(\mathbf{fp\mbox{\textendash} rings}\right)}$$ is the same thing as a sheaf of commutative rings on $X$. Suppose we want more, namely that we want that $X$ is \emph{locally ringed.}\\
\newline
Let us return again to the category $\left(\mathbf{fp\mbox{\textendash} rings}\right)^{op}.$ It comes equipped with the \emph{Zariski Grothendieck topology} $\mathcal{ZAR}$. A covering family of a ring $A$ in $\left(\mathbf{fp\mbox{\textendash} rings}\right)^{op}$ is determined by a collection elements $a_1,a_2,...,a_n$ of $A$ which generate the unit ideal. To such a collection, one may consider the associated localizations of rings $$A \to A\left[\frac{1}{a_i}\right].$$ The dual maps $$A\left[\frac{1}{a_i}\right] \to A$$ in $\left(\mathbf{fp\mbox{\textendash} rings}\right)^{op}$ constitute a cover.  Now, the Zariski Grothendieck topology $\mathcal{ZAR}$ induces a geometric morphism (in fact geometric embedding) $$\Sh_{\mathcal{ZAR}}\left(\left(\mathbf{fp\mbox{\textendash} rings}\right)^{op}\right) \hookrightarrow \Set^{\mathbf{fp\mbox{\textendash} rings}}$$ and composition with it produces from any geometric morphism $$\cE \to \Sh_{\mathcal{ZAR}}\left(\left(\mathbf{fp\mbox{\textendash} rings}\right)^{op}\right)$$ a commutative ring object in $\cE.$ In fact, the commutative ring objects in $\cE$ whose classifying maps $$\cE \to \Set^{\left(\mathbf{fp\mbox{\textendash} rings}\right)}$$ factor through $\Sh_{\mathcal{ZAR}}\left(\left(\mathbf{fp\mbox{\textendash} rings}\right)^{op}\right)$ are precisely the ones which are \emph{local ring objects} with respect to the internal logic of $\cE.$ In particular, local rings internal to the topos $Sh\left(X\right)$ are precisely sheaves of commutative rings with local stalks. In other words, the topos $\Sh_{\mathcal{ZAR}}\left(\left(\mathbf{fp\mbox{\textendash} rings}\right)^{op}\right)$ is the \emph{classifying topos for the theory of local rings}. (See e.g. Chapter VIII, Section 5.5 of \cite{sheaves}.)
\end{ex}


\section{Geometries and Geometric Structures}

Consider Example \ref{ex:classrings}, and consider the lax slice $2$-category $\mathfrak{Top}^{lax}/\Set^{\left(\mathbf{fp\mbox{\textendash} rings}\right)}.$ The objects may be identified with pairs $\left(\cE,\mathcal{O}_\cE\right),$ with $\mathcal{O}_\cE$ a commutative ring object in $\cE.$ An arrow $$\left(\cF,\mathcal{O}_\cF\right) \to \left(\cE,\mathcal{O}_\cE\right),$$ is a pair $\left(f,\varphi\right)$ with $$f:\cF \to \cE$$ a geometric morphism and $$\varphi:f^*\left(\mathcal{O}_\cE\right) \to \mathcal{O}_\cF$$ a homomorphism of ring objects. One easily sees that the full subcategory thereof consisting of geometric morphisms with domain of the form $\Sh\left(X\right)$ for $X$ a space is equivalent to the category of ringed-spaces\footnote{when restricting to sober topological spaces.}. However, the full subcategory of the lax slice $2$-category $\mathfrak{Top}^{lax}/\Sh_{\mathcal{ZAR}}\left(\left(\mathbf{fp\mbox{\textendash} rings}\right)^{op}\right)$ consisting of geometric morphisms with domain of the form $\Sh\left(X\right)$ is \emph{not} equivalent to the category of locally-ringed spaces; the morphisms need not induce a map of local rings at the level of stalks. In order to remedy this situation, one must turn to a certain factorization system on the category $\mathbf{CRings}$ of commutative rings.

\begin{dfn}
A map of commutative rings $$\varphi:A \to B$$ is a \textbf{localization} if there exists a set $S \subset A$ such that $\varphi$ is isomorphic to the morphism $$A \to A\left[S^{-1}\right],$$ where $A\left[S^{-1}\right]$ is the ring $A$ localized at the set $S.$
\end{dfn}

\begin{dfn}
A map of commutative rings $$\varphi:A \to B$$ is \textbf{conservative} if for all $a \in A,$ $a$ is invertible if and only if $\varphi\left(a\right)$ is.
\end{dfn}

\begin{rmk} \label{rmk:locmap}
A ring homomorphism $$\varphi:A \to B$$ between local rings is conservative if and only if it is a map of local rings (i.e. respects the unique maximal ideals).
\end{rmk}

The localizations and conservative morphisms pair together to form the left and right set of a factorization system on $\mathbf{CRings}$ respectively. Moreover, these classes of morphisms readily generalize to the category $\mathcal{M}_\mathbf{CRings}\left(\cE\right)$ of internal commutative rings in any topos $\cE,$ and again form a factorization system. In fact, this factorization system is even functorial in the sense that if $f:\cF \to \cE$ is any geometric morphism, then the induced functor $$\mathcal{M}\mbox{od}_\mathbf{CRings}\left(\cE\right) \to \mathcal{M}\mbox{od}_\mathbf{CRings}\left(\cF\right)$$ sends localizations to localizations, and conservative maps to conservative maps. With this in mind, by Remark \ref{rmk:locmap}, ones sees that for two locally ringed spaces $\left(X,\mathcal{O}_X\right)$ and $\left(Y,\mathcal{O}_Y\right),$ a map of ringed spaces $\left(f,\varphi\right)$ is a map of locally ringed spaces if and only if $\varphi$ is conservative.

This motivates the following definition in \cite{dag}:
\begin{dfn}
(\cite{dag} Definition 1.4.3) Let $\cB$ be an $\i$-topos. A \textbf{geometric structure} $\G$ on $\cB$ is an assignment to every $\i$-topos $\cE$ a factorization system $$\G\left(\cE\right)=( \G_L^{\cE}, \G_{R}^{\cE}),$$ on $\Geo(\cE,\cB)$, where $\G_L^{\cE}$ is the left set of the factorization system (in the sense of \cite{htt}, Section 5.2.8), and $\G_R^{\cE}$ the right set, which depends functorially on $\cE$ in the following sense:\\
\\
For every geometric morphism $f: \cE \rightarrow \cF$, the induced functor
$$\Geo(\cF,\cB) \to \Geo(\cE,\cB)$$ (given by composition) carries $\G_{L}^{\cF}$ to $\G_{L}^{\cE}$ and $\G_{R}^{\cF}$ to $\G_{R}^{\cE}$.
\end{dfn}

\begin{rmk}
To fully see the analogy, note that $\Geo(\cE,\cB)$ should be thought of as the $\i$-category of $\mathbb{T}_\cB$-models in $\cE,$ where $\mathbb{T}_\cB$ is the geometric $\i$-theory classified by $\cB.$ At this moment however, the concept of a geometric $\i$-theory has not been truly developed, so at present, this is only a philosophical remark.
\end{rmk}

\begin{dfn}
Let $\cB$ be an $\i$-topos equipped with a geometric structure $\G$. A \textbf{$\G$-structured $\i$-topos} is a geometric morphism $$\O_\cE:\cE \to \cB.$$ We will sometimes call this a $\G$-structure on $\cE$. The $\icat$ of $\G$-structured $\i$-topoi is the subcategory of $\T^{\mathit{lax}}/\cB,$ (as defined in Section \ref{sec:lax}) which has the same objects as the whole $\icat$, but whose morphisms $$\left(\varphi,\alpha\right):e \to f$$ must have the natural transformation $$\alpha:\varphi^*f^* \Rightarrow e^*,$$ which is an arrow in $\Geo\left(\cE,\cB\right),$ be in $\G^{\cE}_R$. Denote this $\icat$ by $\Str\left(\G\right).$ 
\end{dfn}

If $\left(\cE,\O_E\right)$ is an object of $\Str\left(\G\right),$ the morphism $$\O_\cE:\cE \to \cB$$ should be thought of as a structure sheaf on the $\i$-topos $\cE,$ with values in models for the geometric $\i$-theory classified by $\cB.$ When the geometric structure comes from a \emph{geometry} (as will be explained in the next section) this philosophical remark can be made precise.

\subsection{Geometries}\label{sec:geometry}
A wealth of examples of geometric structures arise from \emph{geometries}, as defined in \cite{dag}. We will only provide the basic idea here.

\begin{dfn}\label{dfn:essalg}
An \textbf{essentially algebraic $\i$-theory} is a small $\icat$ $\mathbb{T}$ with finite limits. A model for such a theory in an $\icat$ $\C$ with finite limits, is a finite limit preserving functor $$M:\mathbb{T} \to \C.$$ Such a theory $\mathbb{T}$ will be called \textbf{idempotent complete} if $\mathbb{T}$ is an idempotent complete $\icat$.
\end{dfn}

\begin{rmk}
The terminology ``essentially algebraic $\i$-theory'' is our own.
\end{rmk}

\begin{rmk}
The requirement that $\mathbb{T}$ be idempotent complete is superfluous if $\mathbb{T}$ is an $n$-category for any finite $n$ (Lemma 1.5.12 of \cite{dag}).
\end{rmk}

\begin{rmk}\label{rmk:classifyingT}
It follows immediately from Proposition 6.1.5.2 of \cite{htt} that the $\i$-category of models for $\mathbb{T}$ in an $\i$-topos $\cE,$ $\mathcal{M}\mbox{od}_\mathbb{T}\left(\cE\right)$ is canonically equivalent to the $\i$-category of geometric morphisms $$\Geo\left(\cE,\Pshi\left(\mathbb{T}\right)\right).$$
\end{rmk}

\begin{ex}
The category $\left(\mathbf{fp\mbox{\textendash} rings}\right)^{op},$ from Example \ref{ex:classrings}, is such an idempotent complete essentially algebraic ($\i$-)theory.
\end{ex}

A \textbf{geometry} $\g$ consists two sets of data.
\begin{itemize}
 \item[i)] an idempotent complete essentially algebraic $\i$-theory $\mathbb{T},$ and
\item[ii)] an \textbf{admissibility structure} on $\mathbb{T}.$
\end{itemize}

We refer the reader to \cite{dag} for the precise definition of an admissibility structure. For our purposes, it will suffice to remark that it consists of a Grothendieck pretopology $\mathcal{J}$ generated by a class of morphisms, called \emph{admissible morphisms}, which are closed under retracts and pullbacks, satisfy a left-cancellation property, and have the property that pullbacks along them exist. The role of the pretopology $\mathcal{J}$ is that, just as the $\i$-topos $\Pshi\left(\mathbb{T}\right)$ classifies models for $\mathbb{T}$ (Remark \ref{rmk:classifyingT}) the $\i$-topos $\Sh_\i\left(\mathbb{T},\mathcal{J}\right)$ classifies \emph{local} $\mathbb{T}$-models; its instructive to think of these as ``generalized local rings.'' The axioms which the admissible morphisms must satisfy are designed in such a way that they canonically give rise to a functorial factorization system on each $\icat$ of local $\mathbb{T}$-models in any $\i$-topos, and hence they induce a geometric structure $\G\left(\g\right)$ on the $\i$-topos $\Sh_\i\left(\mathbb{T},\mathcal{J}\right).$ For each $\i$-topos $\cE,$ it is instructive to view $\G\left(\g\right)_L^{\cE}$ as ``generalized localizations'' and $\G\left(\g\right)_{R}^{\cE}$ as ``generalized conservative maps'' between local $\mathbb{T}$-models in $\cE$.

The following table illustrates the Zariski and \'Etale geometry for ordinary commutative rings introduced by Lurie:
\begin{center}
\resizebox{6in}{!}{\begin{tabular}{llllll}
Geometry & algebras & local algebras & admissible morphisms  &``Localizations'' & ``Conservative maps''\\
\hline \\
Zariski & commutative rings & local rings & maps of the form $A \to A\left[a^{-1}\right]$& localizations & conservative maps\\
\'Etale & commutative rings & strict Henselian local rings & \'etale maps &ind-\'etale maps & Henselian maps\\
\end{tabular}}
\end{center}

There are analogues of the Zariski and \'Etale geometries for simplicial commutative rings and $\Ei$-ring spectra. We refer the reader to \cite{dag} and \cite{spectral} respectively.

\begin{ex}
The Grothendieck site associated to the Zariski geometry is $\left(\mathbf{fp\mbox{\textendash} rings}\right)^{op}$ equipped with the Zariski Grothendieck topology $\mathcal{ZAR}$ as in Example \ref{ex:classrings}. In particular, the classifying topos for the geometry is the classifying topos for local rings.
\end{ex}


Notice that if $\left(\cE,\O_\cE\right)$ is a $\G\left(\g\right)$-structured $\i$-topos, since $\mathbb{T}$ is an essentially algebraic $\i$-theory, $\O_\cE$ may be identified with a finite limit preserving functor $$\O_\cE:\mathbb{T} \to \cE.$$ Composing with the essentially unique map to the terminal $\i$-topos, one gets a finite limit preserving functor $$\mathbb{T} \stackrel{\O_\cE}{\longlongrightarrow} \cE \to \iGpd,$$ i.e. a model for $\mathbb{T}$ in $\iGpd.$ Since $\mathbb{T}$ has finite limits, the $\i$-category $\mathcal{M}\mbox{od}_\mathbb{T}\left(\iGpd\right)$ can be canonically identified with $\mathbf{Ind}\left(\mathbb{T}^{op}\right).$ Since morphisms of $\G\left(\g\right)$-structured $\i$-topoi induce morphisms of structure sheaves in the opposite direction, one gets a global section functor $$\Gamma_\g:\Str_{\G\left(\g\right)} \to \mathbf{Ind}\left(\mathbb{T}^{op}\right)^{op}=:\mathbf{Pro}\left(\g\right).$$ It is shown in \cite{dag} that for any geometry, the global sections functor $\Gamma_\g$ has a full and faithful left adjoint $$\mathbf{Spec}_\g:\mathbf{Pro}\left(\g\right) \to \Str_{\G\left(\g\right)}$$- a spectrum functor of $\g.$ For example, when $\g$ is the Zariski geometry, $\mathbf{Pro}\left(\g\right)$ is the opposite category of commutative rings and $\Specg$ is the classical spectrum functor, sending each ring $A$ to its associated affine scheme modeled as a locally ringed space (viewed as a locally ringed $\i$-topos).

\begin{dfn}\cite{dag}\label{dfn:affine}
An \textbf{affine $\g$-scheme} is a $\G\left(\g\right)$-structured $\i$-topos in the essential image of $\mathbf{Spec}_\g$. Denote the corresponding $\i$-category by $\Affg.$
\end{dfn}

In op. cit., Lurie goes on to develop the theory of general $\g$-schemes, by using affine $\g$-schemes as local models and gluing them together along local homeomorphisms of $\i$-topoi. When $\g$ is the \'etale geometry, this produces a theory of higher Deligne-Mumford stacks. In this manuscript, we generalize this theory, to allow for the starting local models to be of any form. This, for example, allows one to start with smooth manifolds as local models, and get a theory of higher smooth orbifolds (see Section \ref{sec:orbifolds}).

\section{\'Etale Morphisms of Structured $\i$-Topoi}\label{sec:etstr}
\begin{dfn}
Fix a geometric structure $\G$ on an $\i$-topos $\cB.$ Denote by $\Str^{\et}\left(\G\right)$ the subcategory of $\T/\cB,$ on the same objects, whose morphisms are \'etale geometric morphisms. Similarly, for any subcategory $\sL$ of $\Str\left(\G\right),$ denote by $\sL^{\et}$ the full subcategory of $\Str^{\et}\left(\G\right)$ whose objects are in $\sL$.
\end{dfn}

\begin{prop}\label{prop:etcolims}
The $\icat$ $\Str^{\et}\left(\G\right)$ is cocomplete and the inclusion $$\Str^{\et}\left(\G\right) \to \Str\left(\G\right)$$ preserves colimits. Moreover, a cocone $$\xymatrix{K^{\triangleright} \ar[r]^-{\mu} & \Str^{\et}\left(\G\right) \\ K \ar@{^{(}->}[u] \ar[ru]_-{F} & &}$$ is colimiting if and only if the induced cocone $$K^{\triangleright} \stackrel{\mu}{\longrightarrow} \Str^{\et}\left(\G\right) \to \T^{\et}$$ is,
where $\T^{\et}$ is as in Definition \ref{dfn:metale}.
\end{prop}

\begin{proof}
This is proven in \cite{dag} Proposition 2.3.5 for the case when $\G$ comes from a geometry. However, the same proof caries over verbatim for any geometric structure.
\end{proof}

\begin{rmk}
The morphisms from $e:\cE \to \cB$ to $f:\cF \to \cB$ in $\T/\cB$ consist of geometric morphisms $\varphi:\cE \to \cF$ together with a natural equivalence $$\alpha:\varphi^*f^* \Rightarrow e^*.$$ Since every equivalence is in $\G_{R},$ $\Str^{\et}\left(\G\right)$ may be seen as a (faithful but not full) subcategory of $\Str\left(\G\right).$
\end{rmk}

By Remark 2.3.4 of \cite{dag}, we have the following proposition:

\begin{prop}
Let $\left(\cE,\O_\cE\right)$ and $\left(\cF,\O_\cF\right)$ be $\G$-structured $\i$-topoi and let $E \in \cE_0$ be an object. Then the following is a pullback diagram of $\i$-groupoids:
$$\xymatrix{\Hom_{\Str\left(\G\right)}\left(\left(\cE/E,\O_\cE|_{E}\right),\left(\cF,\O_\cF\right)\right) \ar[r] \ar[d] & \Hom_{\Str\left(\G\right)}\left(\left(\cE,\O_\cE\right),\left(\cF,\O_\cF\right)\right) \ar[d]\\
\Hom_{\T}\left(\cE/E,\cF\right) \ar[r] & \Hom_{\T}\left(\cE,\cF\right),}$$
where $\O_\cE|_{E}$ denotes the composite $$\cE/E \longrightarrow \cE \stackrel{\O_\cE}{\longlongrightarrow} \cB.$$
\end{prop}

By Remark 2.3.20 of \cite{dag}, the following corollary follows:

\begin{cor}\label{cor:2.3.20}
Let $\left(\cE,\O_\cE\right)$ and $\left(\cF,\O_\cF\right)$ be $\G$-structured $\i$-topoi and let $E \in \cE_0$ be an object, and $$\left(f,\varphi\right):\left(\cF,\O_\cF\right) \to \left(\cE,\O_\cE\right)$$ a morphism in $\Str\left(\G\right)$ Then the following is a pullback diagram in $\Str\left(\G\right)$:
$$\xymatrix{\left(\cF/f^*\left(E\right), \O_{\cF}|_{f^*\left(E\right)}\right) \ar[d] \ar[r] & \left(\cE/E, \O_{\cE}|_{E}\right) \ar[d]\\
\left(\cF,\O_\cF\right) \ar[r] & \left(\cE,\O_\cE\right).}$$
\end{cor}

We now show that the \'etale topology on $\T$ naturally extends to $\Str\left(\G\right).$

\begin{dfn}
We say that a family of morphisms $$\left(\left(f_\alpha,\varphi_\alpha\right):\left(E_\alpha,\O_{E_\alpha}\right) \to \left(E,\O_\cE\right)\right)$$ in $\Str\left(\G\right)$ is an \textbf{\'etale covering family} if $\left(f_\alpha:\cE_\alpha \to \cE\right)$ is one in $\T.$ Similarly for $n$-\'etale covering families.
\end{dfn}

\begin{prop}
The assignment of each $\g$-structured $\i$-topos $\left(\cE,\O_\cE\right)$ its \'etale covering families define a Grothendieck pretopology on $\Str\left(\G\right).$ Similarly for $n$-\'etale covering families.
\end{prop}
\begin{proof}
In light of Corollary \ref{cor:2.3.20}, the proof of Proposition \ref{prop:etexists} carries over verbatim.
\end{proof}

\begin{rmk}
We shall again call the Grothendieck topology on $\Str\left(\G\right)$ generated by \'etale covering families and $n$-\'etale covering families respectively the \textbf{\'etale topology}, and the \textbf{$n$-\'etale topology}.
\end{rmk}

\begin{prop}\label{prop:subcanong}
For any geometric structure $\g,$ the \'etale topology on $\Str\left(\G\right)$ is subcanonical.
\end{prop}
\begin{proof}
In light of Corollary \ref{cor:pullstabet} and Corollary \ref{cor:2.3.20}, if $$\left(\left(f_\alpha,\varphi_\alpha\right)\left(E_\alpha,\O_{E_\alpha}\right) \to \left(E,\O_\cE\right)\right)$$ is an \'etale covering family, the \v{C}ech nerve $C_\bullet$ of $$\coprod\limits_{\alpha} \left(\cE_\alpha,\O_{\cE_\alpha}\right) \to \left(\cE,\O_\cE\right),$$ has a canonical lift
$$\xymatrix{\Delta^{op} \ar[r]^-{C_\bullet} \ar@{-->}[rd]_-{\tilde C_\bullet} & \Str\left(\G\right) \\
& \Str^{\et}\left(\G\right) \ar[u].}$$ In light of Proposition \ref{prop:etcolims} and Corollary \ref{cor:2.3.20}, it follows that the \'etale topology on $\Str\left(\G\right)$ is subcanonical if and only if the \'etale topology on $\T$ is. So we are done by Proposition \ref{prop:topsubcan}.
\end{proof}

\begin{cor}
 The $n$-\'etale topology on $\Str\left(\G\right)$ is subcanonical.
\end{cor}

\begin{cor}
The \'etale topology on $\Str\left(\G\right)$ naturally restricts to a subcanonical topology on $\Str^{\et}\left(\G\right).$
\end{cor}

\begin{proof}
This follows from Proposition \ref{prop:etcolims}.
\end{proof}

The following proposition follows by an analogous argument as the proof of Proposition \ref{prop:highlandertop}:

\begin{prop}
The restriction of each $m$-\'etale topology to the full subcategory of $\Str\left(\G\right)$ spanned by objects of the form $\left(\cE,\O_\cE\right),$ with $\cE$ $n$-localic, are all equivalent for $m \ge n-1.$
\end{prop}

\begin{prop}\label{prop:trivkan}
Let $\cE$ be an object of $\Str^{\et}\left(\G\right).$ Then the induced map $$\Str^{\et}\left(\G\right) /\left(\cE,\O_\cE\right) \to \T^{\et}/\cE$$ is a trivial Kan fibration.
\end{prop}

\begin{proof}
The proof is similar to that of Lemma $\ref{lem:slicek}$ so we omit it.
\end{proof}

\begin{cor}\label{cor:235}
For each $\g$-structured $\i$-topos $\left(\cE,\O_\cE\right)$ in $\Str^{\et}\left(\G\right),$\\ $\Str^{\et}\left(\G\right)/\left(\cE,\O_\cE\right)$ is equivalent to the underlying $\infty$-category of the $\i$-topos $\cE$.
\end{cor}
\begin{proof}
$\Str^{\et}\left(\G\right)/\left(\cE,\O_\cE\right)$ is equivalent to $\T^{\et}/\cE$ by Proposition \ref{prop:trivkan}, and by Proposition \ref{lem:htt6.3.5.10}, this is in turn equivalent to $\cE.$
\end{proof}

\begin{rmk}
This is proven in the case of geometric structures arising from geometries in \cite{dag} Proposition 2.3.5.
\end{rmk}

\begin{rmk}\label{rmk:canonstr}
For any $\g$-structured $\i$-topos $\left(\cE,\O_\cE\right),$ the \'etale topology on $\Str^{\et}\left(\G\right)$ induces a Grothendieck topology on $\Str^{\et}\left(\G\right)/\left(\cE,\O_\cE\right),$ which under the equivalence given by Corollary \ref{cor:235}, agrees with the epimorphism topology on $\cE.$
\end{rmk}

\chapter{\'Etendues: Gluing Local Models}\label{chap:etendues}
In this section, we will make precise what is means to glue structured $\i$-topoi along local homeomorphisms (i.e. \'etale maps) starting from a collection of local models. This parallels the way one builds manifolds out of Euclidean spaces, or schemes out of affine schemes. Since we are allowing our ``spaces'' to be $\i$-topoi however, in these two instances we get much richer theories than just the theory of smooth manifolds, or the theory of schemes, but rather get a theory of higher generalized orbifolds and a theory of higher Deligne-Mumford stacks respectively. This same framework extends to the setting of derived and spectral geometry as well. This will be spelled out in Chapter \ref{chap:examples}.

\section{\'Etendues}

\begin{notat}
If $\sL$ is a full subcategory of $\Str\left(\G\right)$ for some geometric structure $\G,$ we will denote by $\sL^{\et}$ the subcategory of $\sL$ whose morphisms are all the \'etale morphisms.
\end{notat}

\begin{dfn}
Let $\sL$ be a full subcategory of $\Str\left(\G\right)$ for some geometric structure $\G$. $\sL$ is said to be \textbf{$n$-localic} if every object $\left(\cL,\O_\cL\right)$ of $\sL$ has $\cL$ an $n$-localic $\i$-topos.
\end{dfn}

\begin{dfn}
Let $\sL$ be a full subcategory of $\Str\left(\G\right)$ for some geometric structure $\G$. Define $\sl\left(\sL\right)$ to be the full subcategory of $\Str\left(\G\right)$ consisting of those objects of the form $$\left(\cE,f^*\O_\cL\right),$$ where $$f:\cE \to \cL$$ is an \'etale geometric morphism, with $\cL \in \sL.$ 
\end{dfn}


\begin{rmk}\label{rmk:slclsd}
Any object of $\sl\left(\sL\right)$ is equivalent to one of the form $$\left(\cE/E,\O_\cE|_{E}\right),$$ where $\left(\cE,\O_\cE\right)$ is an object of $\sL$. Moreover, since the composition of two \'etale morphisms is again \'etale, it follows that $\sl\left(\sl\left(\sL\right)\right)=\sl\left(\sL\right).$ 
\end{rmk}



\begin{rmk}\label{rmk:sl}
In light of Remark \ref{rmk:open}, $\sl\left(\sL\right)$ should be regarded as the $\i$-category of all generalized open subsets of elements of $\sL.$ These will form the basic building blocks for scheme or manifold theory based on $\sL.$ For example, if $\sL$ consists of the single object $\mathbb{C}^n$ with is holomorphic structure sheaf, a general $n$-dimensional complex manifold is not locally holomorphic to $\mathbb{C}^n,$ but is rather locally holomorphic to open subspaces of $\mathbb{C}^n.$ This is why when starting with a collection of generalized spaces $\sL$ that we wish to glue together, we first expand the collection to include all generalized open subsets. This is the role of the $\i$-category $\sl\left(\sL\right).$
\end{rmk}

The following is an immediate consequence of Corollary \ref{cor:235}:

\begin{cor}\label{cor:2352}
For $\sL$ any subcategory of $\Str\left(\G\right),$ and $\left(\cE,\O_\cE\right)$ any object of $\sl\left(\sL\right),$ $\sl\left(\sL\right)^{\et}/\left(\cE,\O_\cE\right)$ is equivalent to the underlying $\icat$ of $\cE.$
\end{cor}

\begin{proof}
Any object of $\cF$ of $\Str\left(\G\right)$ which admits an \'etale map $\cF \to \cE$ over $\cB$ is in the subcategory $\sl\left(\sL\right)$ by Remark \ref{rmk:slclsd}. Hence $\Str^{\et}\left(\G\right)/\cE=\sl\left(\sL\right)^{\et}/\cE$ and the result follows from Corollary \ref{cor:235}.
\end{proof}

\begin{rmk}
From now on, when there is no risk of confusion, we will often abuse notation and denote an object $\left(\cE,\O_\cE\right)$ of $\sL$ by $\cE$ and similarly for morphisms. Furthermore, if $E \in \cE$, and $\left(\cE,\O_\cE\right)$ is an object of $\Str\left(\G\right),$ we will often write $\cE/E$ to mean $\left(\cE/E,\O_\cE|_{E}\right).$
\end{rmk}

\begin{rmk}
The initial $\i$-topos $\emptyset$ (which itself, as an $\icat$ is a contractible $\i$-groupoid) has $\Hom\left(\emptyset,\cB\right)$ a contractible $\i$-groupoid for each $\i$-topos $\cB.$ Hence, up to a contractible space of ambiguity, $\emptyset$ has a unique and canonical $\G$-structure. We will call $\emptyset$ with this $\G$-structure the initial $\G$-topos. Note that since for any $\left(\cE,\O_\cE\right)$ $$\xymatrix{\emptyset \ar[dr] \ar[rr] & & \cE \ar[dl]^-{\O_\cE}\\ & \cB&}$$ commutes up to a contractible choice of homotopies, it follows that the initial $\G$-topos is an initial object in $\Str\left(\G\right).$
\end{rmk}

\begin{dfn}
Let $\sL$ be a full subcategory of $\Str\left(\G\right)$ for some geometric structure $\G$. Define $\cosl\left(\sL\right)$ to be the full subcategory of $\Str\left(\G\right)$ consisting of the initial $\G$-topos, together with those objects $\cF$ which admit an \'etale covering family $\left(\cF_\alpha \to \cF\right)$ such that each $\cF_\alpha$ is in $\sL.$
\end{dfn}

\begin{rmk}\label{rmk:locequiv}
As \'etale morphisms are to be regarded as local homeomorphisms, or ``the inclusions of generalized open subsets,'' the $\i$-category $\cosl\left(\sL\right)$ can be thought of as the collection of all structured $\i$-topoi which are ``locally equivalent'' to an object  in $\sL.$
\end{rmk}

\begin{ex}\label{ex:affines}
If $\G$ comes from a geometry $\g,$ and $\sL$ is the full subcategory of $\Str\left(\G\right)$ spanned by affine $\g$-schemes in the sense of \cite{dag}, then $\cosl\left(\sL\right)$ is the $\icat$ of $\g$-schemes.
\end{ex}



\begin{dfn}\label{dfn:letendue}
For a fixed $\sL,$ the $\icat$ $\cosl\left(\sl\left(\sL\right)\right)=: \overline \sL$ is called the \textbf{\'etale closure} of $\sL.$ Objects $\left(\cE,\O_\cE\right)$ are called \textbf{$\sL$-\'etendues}.
\end{dfn}

\begin{rmk}
Combining Remark \ref{rmk:sl} and \ref{rmk:locequiv}, the $\i$-category $\lbar$ of $\sL$-\'etendues is the $\i$-category consisting of those structured $\i$-topoi which can be glued together along local homeomorphisms out of generalized open subsets of elements of $\sL.$ I.e, the objects of $\lbar$ are the ``manifolds'' or ``schemes'' that one can build starting with objects in $\sL.$
\end{rmk}


\begin{prop}\label{prop:coslcosl}
For any full subcategory $\sL$ of $\Str\left(\G\right),$ $\cosl\left(\cosl\left(\sL\right)\right)=\cosl\left(\sL\right)$
\end{prop}

\begin{proof}
Suppose that $\cE$ is such that it admits an \'etale covering family $$\left(\cE_\alpha \to \cE\right)_\alpha$$ with each $\cE_\alpha$ admitting an \'etale covering family $$\left(\cE^\beta_\alpha \to \cE_\alpha\right)_\beta$$ with each $\cE^\beta_\alpha \in \sL.$ Then by Proposition \ref{prop:etexists}, the composite $$\left(\cE^\beta_\alpha \to \cE\right)_{\alpha,\beta}$$ is an \'etale covering family, so we are done.
\end{proof}

\begin{cor}
$\cosl\left(\overline \sL\right)=\overline \sL$
\end{cor}

\begin{prop}\label{prop:2.3.10}(\cite{dag} 2.3.10)
If $\G$ arises from a geometry $\g$ and $\cF \to \cE$ is an \'etale morphism in $\Str\left(\G\right)$ with $\cE$ is a $\g$-scheme, then $\cF$ is a $\g$-scheme.
\end{prop}

\begin{cor}
If $\g$ is a geometry and $\sL$ is the $\icat$ of \emph{affine $\g$-schemes} viewed as a subcategory of $\G\left(\g\right)$-structured $\i$-topoi, then $\overline \sL$ is the $\icat$ of $\g$-schemes.
\end{cor}

\begin{proof}
Notice that one always has the inclusion $$\cosl\left(\sL\right) \subseteq \cosl\left(\sl\left(\sL\right)\right)=\overline \sL.$$ From Proposition \ref{prop:2.3.10}, $$\sl\left(\sL\right) \subseteq \cosl\left(\sL\right).$$ It follows then that $$\overline \sL =\cosl\left(\sl\left(\sL\right)\right) \subseteq \cosl\left(\cosl\left(\sL\right)\right)=\cosl\left(\sL\right).$$
\end{proof}


\begin{prop}
$\sl\left(\cosl\left(\sL\right)\right) \subseteq \cosl\left(\sl\left(\sL\right)\right).$
\end{prop}
\begin{proof}
Suppose that $\cL \in \cosl\left(\sL\right)$ and $T \in \cL_0$. Let $\left(T_\alpha\right)_{\alpha \in A}$ be a set of objects such that $$\coprod\limits_\alpha T_\alpha \to 1$$ is an effective epimorphism and such that each $\cL/T_\alpha$ is in $\sL$. Then, since effective epimorphisms are stable under pullback, the projection map $$T \times \coprod\limits_\alpha T_\alpha \to T$$ is an effective epimorphism. Each projection
\begin{equation*}\label{eq:projection}
pr^1_\alpha:T \times T_\alpha \to T
\end{equation*}
is an object $V_\alpha$ in the slice $\i$-topos $\cL/T$ and it follows that the canonical map $$\coprod\limits_\alpha V_\alpha \to 1_{\cL/T}=id_T$$ is an effective epimorphism. Note that we also have the projections
\begin{equation*}
pr^2_\alpha:T \times T_\alpha \to T_\alpha,
\end{equation*}
which are objects $W_\alpha$ in the slice topos $\left(\cL/T_\alpha\right).$ By Lemma \ref{lem:slicek}, for each $\alpha$, one has $$\left(\cL/T\right)/V_\alpha \simeq \cL/\left(T \times T_\alpha\right)\simeq \left(\cL/T_\alpha\right)/W_\alpha.$$
So each $\left(\cL/T\right)/V_\alpha \in \sl\left(\sL\right),$ showing that $\cL/T$ is in $\cosl\left(\sl\left(\sL\right)\right).$
\end{proof}

\begin{cor}\label{cor:stabslic}
$\sl\left(\overline \sL\right)=\overline \sL.$
\end{cor}

\begin{cor}\label{cor:etclscls}
The \'etale closure  of $\overline \sL$ is $\overline \sL$ itself.
\end{cor}

\begin{rmk}\label{rmk:2352}
By Corollary \ref{cor:2352}, it follows that for each object $\left(\cE,\O_\cE\right)$ in $\overline{\sL},$ $\overline{\sL}^{\et}/\cE$ is equivalent to the underlying $\infty$-category of $\cE$.
\end{rmk}

\begin{prop}\label{prop:etalestable}
For any subcategory $\sD$ of $\Str\left(\G\right),$ the following are equivalent:
\begin{itemize}
\item[1)] $\cosl\left(\sD\right)=\overline \sD$
\item[2)] $\sl\left(\cosl\left(\sD\right)\right) = \cosl\left(\sD\right).$
\item[3)] $\sl\left(\sD\right) \subseteq \cosl\left(\sD\right)$
\end{itemize}
\end{prop}

\begin{proof}
By Corollary  \ref{cor:etclscls}, $1) \implies 2).$ $2 \implies 3)$ is obvious. Suppose $3)$ holds. Then by Proposition \ref{prop:coslcosl}
$$\cosl\left(\sD\right)=\cosl\left(\cosl\left(\sL\right)\right) \supseteq \cosl\left(\sl\left(\sD\right)\right)=\overline \sD.$$ So $3) \implies 1).$
\end{proof}

\begin{dfn}
For any subcategory $\sD$ of $\Str\left(\G\right)$ satisfying any of the equivalent conditions of Proposition \ref{prop:etalestable}, $\sD$ is said to be an \textbf{\'etale blossom}. If $\sL$ is another subcategory such that $$\lbar=\dbar,$$ $\sD$ is said to be an \textbf{\'etale blossom for $\lbar.$}
\end{dfn}

\begin{rmk}
For any $\sL,$ $\sl\left(\sL\right)$ is an \'etale blossom for $\overline \sL.$
\end{rmk}

The following proposition is a special case of Proposition 1.2.13.8 of \cite{htt}:

\begin{prop}\label{prop:htt1.2.13.8}
Let $C$ be an object in an $\infty$-category $\C$ and suppose that $$p:K \to \C/C$$ is a functor such that the composite $p_0$ of $$K \stackrel{p}{\longrightarrow} \C/C \to \C$$ has a colimit. Then
\begin{itemize}
\item[i)] A colimit for $p$ exists and this colimit is preserved by the projection $$\C/C \to \C.$$
\item[ii)] A cocone $$\xymatrix{K^{\triangleright} \ar[r]^-{\mu} & \C/C \\ K \ar@{^{(}->}[u] \ar[ru]_-{p} & &}$$ is colimiting if and only if the composite $$\xymatrix{K^{\triangleright} \ar[r]^-{\mu} & \C/C \ar[r] & \C \\ K \ar@{^{(}->}[u] \ar[rru]_-{p_0} & &}$$ is.
\end{itemize}
\end{prop}

\begin{prop}\label{prop:colim1}
Let $F:K \to \C$ be a functor between $\infty$-categories and $$\xymatrix{K^{\triangleright} \ar[r]^-{\rho} & \C\\ K \ar@{^{(}->}[u] \ar[ru]_-{F} & &}$$
a colimiting cocone for $F$ with vertex $\rho\left(\infty\right)=C.$ Then $\rho$ induces a canonical lift of $F$ to $$\xymatrix{K \ar[r]^-{\tilde \rho} \ar[rd]_-{F} & \C/C \ar[d]^-{\pi}\\ & \C}$$ and the colimit of $\tilde \rho$ exists and is a terminal object.
\end{prop}

\begin{proof}
Note by \cite{Joyal}, for every simplicial set $X,$ the functor $$\left( \mspace{3mu} \cdot \mspace{3mu}\right)\star X:\Sset \to X/\Sset$$ has a right adjoint $sl$ which associates a map $p:X \to Y$ to a simplicial set $Y/p.$ In the case that $X=\Delta^0,$ $X/\Sset=\Sset_{*}$ is the category of pointed simplicial sets and if $v:\Delta^0 \to \C$ is an object of an $\infty$-category
$\C$, the simplicial set $sl\left(\C,v\right)$ is the $\infty$-slice category $\C/v.$ In this case, the left adjoint $\left( \mspace{3mu} \cdot \mspace{3mu}\right)\star \Delta^0$ is best denoted $\left( \mspace{3mu} \cdot \mspace{3mu}\right)^{\triangleright}$
as it associates a simplicial set to its right cone. Denote the co-unit of this adjunction by $\varepsilon.$ The cocone
$$\rho:K^{\triangleright} \to \C$$ by definition corresponds to a map $$\rho:K \star \Delta^{0} \to \C$$ such that
$\rho|_{\Delta^0}=C,$ i.e. a map of pointed simplicial sets $$\left(K^{\triangleright},\infty\right) \to \left(\C,C\right).$$
By adjunction, this corresponds to a map $\tilde \rho:K \to sl\left(\C,C\right)=\C/C$ such that $\pi \circ p =F,$
where $\pi:\C/C \to \C$ is the canonical projection. Consider the terminal object $id_C \in \C/C.$ Since it is terminal, the canonical projection
$$\lambda:\left(\C/C\right)/id_C \to \C/C$$
is a trivial Kan fibration, so we can choose a section $\sigma.$ Composition with $\sigma$ induces a map $$K \stackrel{\tilde \rho}{\longrightarrow} \C/C \stackrel{\sigma}{\longrightarrow} \left(\C/C\right)/id_C=sl\left(\left(\C/C\right),id_C\right)$$ which, by adjunction corresponds to a map of pointed simplicial sets
$$\widehat{\rho}:\left(K^{\triangleright},\infty\right) \to \left(\C/C,id_C\right),$$
such that the following diagram commutes: $$\xymatrix{K^{\triangleright} \ar[r]^-{\widehat{\rho}} & \C/C, \\ K \ar@{^{(}->}[u] \ar[ru]_-{\tilde \rho} & &}$$ i.e. it corresponds to a cocone for $\tilde \rho$ with vertex $id_C.$ Notice that $\lambda$ is the functor $sl$ applied to the map of pointed simplicial sets
$$\pi:\left(\C/C,id_C\right) \to \left(\C,C\right).$$
 Since $\sigma$ is a section of $\lambda,$ the naturality square for the co-unit $\varepsilon$ shows that $\pi \circ \widehat{\rho}=\rho,$ which is colimiting. Now, Proposition \ref{prop:htt1.2.13.8} implies that $\widehat{\rho}$ is also colimiting.
\end{proof}

\begin{lem}\label{lem:holmol}
Let $\sL$ be any full subcategory of $\Str^{\et}\left(\G\right).$ Then $\cosl\left(\sL\right)^{\et}$ is cocomplete and the inclusion $$\cosl\left(\sL\right)^{\et} \hookrightarrow \Str^{\et}\left(\G\right)$$ preserves colimits. Moreover, a cocone $$\xymatrix{K^{\triangleright} \ar[r]^-{\mu} & \cosl\left(\sL\right)^{\et} \\ K \ar@{^{(}->}[u] \ar[ru]_-{F} & &}$$ is colimiting if and only if the induced cocone $$K^{\triangleright} \longrightarrow \Str^{\et}\left(\G\right)$$ is.
\end{lem}
\begin{proof}
By definition, $\cosl\left(\sL\right)^{\et}$ contains the initial $\G$-topos. Notice that the canonical map $\emptyset \to \cE$ is always \'etale since it is induced by slicing over the initial object of $\cE,$ hence the initial $\G$-topos is an initial object in $\cosl\left(\sL\right)^{\et}.$
Now, suppose that $F:J \to \cosl\left(\sL\right)^{\et}$ is any functor from a non-empty small $\infty$-category $J$. Denote by $\cE$ the colimit of the functor $$J\stackrel{F}{\longrightarrow} \cosl\left(\sL\right)^{\et} \to \Str^{\et}\left(\G\right)$$ (which exists by Proposition \ref{prop:etcolims}). Let $$\rho:J^{\triangleright} \to \Str^{\et}\left(\G\right)$$ be a colimiting cocone on $\cE$. Then by Proposition \ref{prop:colim1}, there exist a lift $$\widetilde{F}:J \to \Str^{\et}\left(\G\right)/\cE$$ of $F$ over the canonical projection $$\Str^{\et}\left(\G\right)/\cE \to \Str^{\et}\left(\G\right),$$ such that $$\colim \widetilde{F} \simeq id_\cE.$$ Under the equivalence $\Str^{\et}\left(\G\right)/\cE \simeq \cE,$ the colimit of $\widetilde{F}$ is the terminal object $1$ and the components of the cocone for $\widetilde{F},$ $$\rho\left(i\right):F(i) \to \cE$$ correspond to objects $F_i \in \cE,$ such that each $\cE/F_i\simeq F(i)$ is in $\cosl\left(\sL\right)^{\et}.$ By \cite{htt}, Lemma 6.2.3.13, it follows that the induced map $\coprod\limits_{i \in  J_0} F_i \to 1$ is an effective epimorphism. Hence $\cE \in \cosl\left(\cosl\left(\sL\right)^{\et}\right)=\cosl\left(\sL\right)^{\et}$ and is a colimit for $F$.
\end{proof}

\begin{cor}
For any full subcategory $\sL$ of $\Str\left(\G\right),$ $\overline{\sL}^{\et}$ is cocomplete.
\end{cor}


\begin{dfn}
A full subcategory $\C$ of an $\i$-category $\sD$ is said to \textbf{generate $\sD$ under colimits} if the smallest subcategory of $\sD$ containing $\C$ and closed under taking small colimits in $\sD$ is $\sD$ itself.
\end{dfn}

\begin{prop}\label{prop:cobloss}
If $\sD$ is a full subcategory of $\lbar$ such that $\sD^{\et}$ generates $\lbare$ under colimits, then $\sD$ is an \'etale blossom for $\lbar$. Similarly, if $\sD$ is a full subcategory of $\sl\left(\sL\right)$ such that $\sD^{\et}$ generates $\sl\left(\sL\right)^{\et}$ under colimits, then $\sD$ is an \'etale blossom for $\lbar$.
\end{prop}

\begin{proof}
Suppose $\sD^{\et}$ generates $\lbare$ under colimits. The inclusion $$\cosl\left(\sD\right)^{\et} \hookrightarrow \cosl\left(\lbar\right)^{\et}=\lbar^{\et}$$ preserves and reflects colimits and $\cosl\left(\sD\right)^{\et}$ is cocomplete. Since $\sD^{\et}$ generates $\lbare$ under colimits, this implies $\cosl\left(\sD\right)^{\et}$ contains $\lbare.$ The other case is analogous.
\end{proof}

\begin{prop}\label{prop:gens}
Denote by $\left(\sl\left(\sL\right)^{\et}\right)^{\coprod}$ the full subcategory of $\overline{\sL}^{\et}$ generated by $\sl\left(\sL\right)^{\et}$ under arbitrary coproducts. Then every object of $\overline{\sL}^{\et}$ is the colimit of a diagram of the form $$\Delta^{op} \stackrel{G_\bullet}{\longlongrightarrow} \left(\sl\left(\sL\right)^{\et}\right)^{\coprod} \hookrightarrow \overline{\sL}^{\et}.$$ 
In particular, $\overline{\sL}^{\et}$ is generated under colimits by $\sl\left(\sL\right)^{\et}.$
\end{prop}

\begin{proof}
Let $\cE$ be an object in $\overline{\sL}.$ Then there exists a set of objects $\left(E_\alpha \in \cE\right)_{\alpha \in A}$ such that $$\rho:\coprod E_\alpha \to 1_\cE$$ is an effective epimorphism,
and such that each $\cE/E_\alpha \in \sl\left(\sL\right).$ Since $1_{\cE}$ is the colimit of the \v{C}ech nerve of $\rho,$ and the composite $$\cE \simeq \overline \sL^{\et}/\cE \to \overline \sL^{\et}$$
is colimit preserving, it follows that $\cE$ is the colimit of the \v{C}ech nerve $C_\bullet$ of $$\coprod{\cE/E_\alpha} \to \cE.$$ Since $\cE \simeq \overline \sL^{\et}/\cE,$ and colimits are universal in $\cE,$ we have for each $n,$
\begin{eqnarray*}
C_n &\simeq& \left(\coprod \cE/E_\alpha\right) \times_{\cE} \cdots \times_{\cE} \left(\coprod \cE/E_\alpha\right)\\
 &\simeq& \cE/\left(\coprod E_\alpha \times \cdots \times\coprod E_\alpha\right)\\
 &\simeq& \cE/\left(\coprod\limits_{\tiny \left(\alpha_1,\cdots,\alpha_n\right) \in A^n} \left(E_{\alpha_1} \times \cdots \times E_{\alpha_n}\right)
\right)\\
&\simeq& \coprod\limits_{\tiny \left(\alpha_1,\cdots,\alpha_n\right) \in A^n} \cE/\left(E_{\alpha_1} \times \cdots \times E_{\alpha_n}\right).\\
\end{eqnarray*}
Consider the map $pr_1:E_{\alpha_1} \times \cdots \times E_{\alpha_n} \to E_{\alpha_1}$. Under the equivalence $\cE \simeq \overline \sL^{\et}/\cE,$ it corresponds to an \'etale map $$\cE/\left(E_{\alpha_1} \times \cdots \times E_{\alpha_n}\right) \to \cE/E_{\alpha_1}.$$ Since each $\cE/E_{\alpha_i}$ is in $\sl\left(\sL\right)^{\et},$ so is each $\cE/\left(E_{\alpha_1} \times \cdots \times E_{\alpha_n}\right)$ by Remark \ref{rmk:slclsd}. Hence, each $C_n \simeq \coprod \cE/\left(E_{\alpha_1} \times \cdots \times E_{\alpha_n}\right) \in \left(\sl\left(\sL\right)^{\et}\right)^{\coprod},$ and $\cE \simeq \colim C_\bullet.$
\end{proof}


\begin{thm}\label{thm:cosluniv}
Let $\sD$ be an \'etale blossom. Then $\cosl\left(\sD\right)^{\et}$ is the smallest subcategory of $\Str\left(\G\right)^{\et}$ containing $\sl\left(\sD\right)$ such that its inclusion into $\Str\left(\G\right)^{\et}$ preserves and reflects colimits.
\end{thm}
\begin{proof}
Suppose $i:\K^{\et} \hookrightarrow \Str\left(\G\right)^{\et}$ preserves and reflects colimits and $\K^{\et}$ contains $\sl\left(\sD\right).$ It suffices to show that each object $\cE$ of $\Str\left(\G\right)^{\et}$ which lies in $\cosl\left(\sD\right)$ is also in $\K^{\et}.$ Since $i$ reflects colimits and $\Str\left(\G\right)^{\et}$ is cocomplete, $\K^{\et}$ is cocomplete. Hence, $\K^{\et}$ contains $\left(\sl\left(\sD\right)^{\et}\right)^{\coprod}.$ Let $j:\cosl\left(\sD\right)=\overline\sD \hookrightarrow \Str\left(\G\right)^{\et}$ denote the inclusion. Then Proposition \ref{prop:gens} together  with Lemma \ref{lem:holmol} implies that for any object $\cE$ of $\overline \sD=\cosl\left(\sD\right),$ $j\left(\cE\right)$ can be written as a colimit of a diagram of the form $$\Delta^{op} \to \left(\sl\left(\sD\right)^{\et}\right)^{\coprod} \to \Str\left(\G\right)^{\et}.$$ Since $\K^{\et}$ contains $\left(\sl\left(\sD\right)^{\et}\right)^{\coprod}$ and $i$ reflects colimits, it follows that $\cE \in \K^{\et}.$
\end{proof}


\begin{cor}
Suppose that $\sD \subseteq \sl\left(\sL\right)$ is any full subcategory. Then $$\cosl\left(\sD\right)=\overline \sL$$ if and only if $$\cosl\left(\sD\right) \supseteq \sl\left(\sL\right).$$
\end{cor}
\begin{proof}
Suppose that $\cosl\left(\sD\right) \supseteq \sl\left(\sL\right).$ Then since $$\sl\left(\sL\right) \supseteq \sD,$$ by Remark \ref{rmk:slclsd} $$\sl\left(\sL\right)=\sl\left(\sl\left(\sL\right)\right) \supseteq \sl\left(\sD\right),$$ and hence $\sD$ is an \'etale blossom. Hence, $\cosl\left(\sD\right)=\overline \sD.$ On one hand, since $$\sl\left(\sL\right) \subseteq \sD,$$ one has that $$\overline \sL=\cosl\left(\sl\left(\sD\right)\right)\subseteq \cosl\left(\sD\right).$$ On the other hand, $\overline \sL^{\et} \hookrightarrow \Str\left(\G\right)^{\et}$ preserves and reflects colimits and contains $\sl\left(\sL\right),$ which in turn contains $\sl\left(\sD\right).$ Hence, by Theorem \ref{thm:cosluniv}, $$\cosl\left(\sD\right) \subseteq \overline \sL.$$ The converse is trivial.
\end{proof}

We now will define a technical, yet highly important concept for this manuscript, namely that of a \emph{strong \'etale blossom}. A general subcategory $\sL$ of $\Str_\G$ may not be rich enough so that $\sL$-\'etendues are faithfully represented by the $\i$-sheaves they induce over $\sL,$ even if $\sL$ is an \'etale blossom. Heuristically, the main problem is that for $\cL$ in $\sL,$ $\sL$ may ``fail to contain enough of the open sets'' of $\cL.$ Strong \'etale blossoms are defined in such a way as to avoid such a pathology. It will turn out that strong \'etale blossoms are in abundant enough supply that for any subcategory $\sL$ of $\Str_\G,$ there exists a strong \'etale blossom $\sD \supset \sL,$ such that a structured $\i$-topos $\left(\cE,\O_\cE\right)$ in $\Str_\G$ is an $\sL$-\'etendue if and only if it is a $\sD$-\'etendue. The principal difference is only that the canonical functor $$\lbar \to \Pshi\left(\sL\right)$$ may not be fully faithful, whereas the functor $$\lbar \to \Pshi\left(\sD\right)$$ will be. However, this will not be proven until the end of Section \ref{sec:functorofpoints}.

Suppose that $\sD$ is a locally small full subcategory of $\Str\left(\G\right)$ and let $\cD$ be an object of $\sD.$ Let $$i:\sD^{\et} \hookrightarrow \overline{\sD}^{\et}$$ denote the full and faithful inclusion. Denote by $\widetilde{i/\cD}$ the induced full and faithful inclusion $$\sD^{\et}/\cD \stackrel{i/\cD}{\longlonghookrightarrow} \overline{\sD}^{\et}/\cD \stackrel{\sim}{\longrightarrow} \cD.$$ Suppose that $$r_\cD:\C_\cD\hookrightarrow \sD^{\et}/\cD$$ is the inclusion of a small full subcategory. Then there exists a left Kan extension
$$L_\cD:=\Lan_{y} \left(\widetilde{i/\cD} \circ r_\cD\right):\Psh_\i\left(\C_\cD\right) \to \cD,$$ where $y$ denotes the Yoneda embedding of $\C_\cD.$
This functor has a right adjoint $R_{\cD};$ If $D$ is an object of $\cD$ and $f:\cE \to \cD$ is an \'etale morphism in $\sD$ belonging to $\C_\cD,$ then it corresponds to an object $E$ in $\cD$ and  $$R_{\cD}\left(D\right)\left(f\right)=\Hom_\cD\left(E,D\right).$$

\begin{dfn}\label{dfn:strongblossom}
A locally small full subcategory $\sD$ of $\Str\left(\G\right)$ as above is said to be a \textbf{strong \'etale blossom} (for $\overline{\sD}$) if for each object $\left(\cD,\O_\cD\right)$ of $\sD,$ the slice category $\sD^{\et}/\cD$ admits a small full subcategory $\C_\cD$ as above such that the induced adjunction
$$\xymatrix@C=1.5cm{\cD \ar@{^{(}->}[r]<-0.9ex>_-{R_\cD} & \Psh_\i\left(\C_\cD\right) \ar@<-0.5ex>[l]_-{L_{\cD}}},$$
exhibits $\cD$ as an accessible left exact localization of $\Psh_\i\left(\C_\cD\right),$ i.e. $R_{\cD}$ is full and faithful and $L_\cD$ is left exact and accessible.
\end{dfn}

\begin{rmk}\label{rmk:supblossom}
If $\sD \hookrightarrow \sD' \hookrightarrow \dbar$ are fully faithful inclusions of locally small full subcategory of $\Str\left(\G\right)$ and $\sD$ is a strong \'etale blossom for $\dbar,$ then so is $\sD'.$
\end{rmk}

\begin{prop}\label{prop:largesite}
Let $\cD$ be an object of a strong \'etale blossom $\sD.$ Then the functor $i_\cD$ defined as the composite
\begin{equation}\label{eq:iD}
\cD \stackrel{y}{\longhookrightarrow} \Pshi\left(\cD\right) \stackrel{\sim}{\longrightarrow} \Pshi\left(\dbare/\cD\right) \stackrel{\left(i/\cD\right)^*}{\longlongrightarrow} \Psh_\i\left(\sD^{\et}/\cD\right),
\end{equation}
admits a left exact left adjoint $a_\cD.$
\end{prop}

\begin{proof}
Choose a site $\C_\cD$ for $\cD,$ with a full and faithful inclusion $$r_\cD:\C_\cD \hookrightarrow \sD^{\et}/\cD$$ as in Definition \ref{dfn:strongblossom}. Consider the functor $$r_\cD^*:\Pshi\left(\sD^{\et}/\cD\right) \to \Pshi\left(\C_\cD\right)$$ given by restriction along $r_\cD$. Denote the composite
$$\C_\cD \stackrel{r_\cD}{\longlonghookrightarrow} \sD^{\et}/\cD \stackrel{i/\cD}{\longlonghookrightarrow} \dbare/\cD \stackrel{\theta}{\longrightarrow} \cD,$$ with $\theta$ the canonical equivalence, by $\psi_\cD.$ Consider the left exact localization
$$\xymatrix@C=1.5cm{\cD \ar@{^{(}->}[r]<-0.9ex>_-{R_\cD} & \Psh_\i\left(\C_\cD\right) \ar@<-0.5ex>[l]_-{L_{\cD}}}.$$
Notice that $L_\cD=\Lan_y\left(\psi_\cD\right).$ Let $f$ be an object of $\C_\cD$ and $D$ an object of $\cD.$ On one hand, one has
\begin{eqnarray*}
R_\cD\left(D\right)\left(f\right)  &\simeq& \Hom\left(y\left(f\right),R_\cD\left(D\right)\right)\\
&\simeq& \Hom\left(\psi_\cD\left(f\right),D\right).
\end{eqnarray*}
On the other hand, if $\varphi$ is in $\sD^{\et}/\cD,$ one has
\begin{eqnarray*}
r_\cD^*y\left(\varphi\right)\left(f\right) &\simeq& \Hom\left(r_\cD\left(f\right),\varphi\right)\\
&\simeq& \Hom\left(\psi_\cD\left(f\right),\theta\circ i/\cD\left(\varphi\right)\right).
\end{eqnarray*} One concludes that $$r_\cD^*y\left(\varphi\right)\simeq R_\cD\left(\theta\circ i/\cD\left(\varphi\right)\right),$$ and hence
\begin{equation}\label{eq:rds}
L_\cD r_\cD^*y\left(\varphi\right) \simeq \theta\circ i/\cD\left(\varphi\right).
\end{equation}
Notice that the restriction functor $$\widehat{r_\cD}^*:\LPshi\left(\sD^{\et}/cD\right) \to \LPshi\left(\C_\cD\right)$$ between presheaves of large $\i$-groupoids has a right adjoint $\widehat{r_\cD}_*$ given by the formula
$$
\widehat{r_\cD}_*F\left(\varphi\right) \simeq \Hom\left(\widehat{r_\cD}^*y\left(\varphi\right),F\right),
$$
which takes values in small $\i$-groupoids if $F$ does, and hence we have an induced adjunction
$$\xymatrix@C=1.5cm{\Psh_\i\left(\C_\cD\right)\ar@{^{(}->}[r]<-0.9ex>_-{{r_\cD}_*} & \Psh_\i\left(\sD^{\et}/\cD\right). \ar@<-0.5ex>[l]_-{r_\cD^*}}$$
The functor $r_\cD^*$ preserves all limits since they are compute point-wise, and a standard argument shows that since $r_\cD$ is full and faithful so is ${r_\cD}_*.$  Combining the two localizations exhibits $\cD$ as a left exact localization of $\Psh_\i\left(\sD^{\et}/\cD\right)$ :
$$\xymatrix@C=1.5cm@R=1.5cm{\cD \ar@{^{(}->}[r]<-0.9ex>_-{i_\cD} & \Psh_\i\left(\sD^{\et}/\cD\right) \ar@<-0.5ex>[l]_-{a_\cD}.}$$
We know wish to identify $i_\cD={r_\cD}_* \circ R_\cD,$ with the functor (\ref{eq:iD}). Let $D$ be an object of $\cD$ and $\varphi$ an object of $\sD^{\et}/\cD$. Notice that
\begin{eqnarray*}
{r_\cD}_*R_\cD\left(D\right)\left(\varphi\right) &\simeq& \Hom\left(r_\cD^*y\left(\varphi\right),R_\cD\left(D\right)\right)\\
&\simeq& \Hom\left(L_\cD r_\cD^*y\left(\varphi\right),D\right)\\
&\simeq& \Hom\left(\theta \circ i/\cD\left(\varphi\right),D\right)\\
&\simeq& \left(\theta^*\left(i/\cD\right)^*y\left(D\right)\right)\left(\varphi\right).
\end{eqnarray*}
with the second to last equivalence following from (\ref{eq:rds}).
\end{proof}


It is not clear from the definition that a strong \'etale blossom is in fact an \'etale blossom. However, the following proposition shows that this is indeed the case:

\begin{prop}\label{prop:stretex}
If $\sD$ is a strong \'etale blossom, then $\sD$ is an \'etale blossom for $\overline{\sD}.$
\end{prop}

\begin{proof}
By Proposition \ref{prop:cobloss}, it suffices to show that each object in $\sl\left(\sD\right)^{\et}$ can be obtained as the colimit of a diagram in $\sD^{\et}.$ Since $\sD$ is a strong \'etale blossom, choose for each $\cD$ in $\sD$ a small subcategory $$r_\cD:\C_\cD \hookrightarrow \sD^{\et}/\cD$$ satisfying the conditions of Definition \ref{dfn:strongblossom}, and let
$$\xymatrix@C=1.5cm@R=2.5cm{\cD \ar@{^{(}->}[r]<-0.9ex>_-{R_\cD} & \Psh_\i\left(\C_\cD\right) \ar@<-0.5ex>[l]_-{L_{\cD}}}$$
denote the corresponding localization. Suppose that $\cE \in \sl\left(\sD\right)^{\et}_0.$ Then $\cE \simeq \cD/D$ for some $D \in \cD,$ for some $\cD \in \sD.$ Now, by \cite{htt} Proposition 5.1.5.3, it follows that $R_\cD\left(D\right)$ is the colimit of $$\C_\cD/R_\cD\left(D\right) \to \C_\cD \stackrel{y}{\longhookrightarrow} \Psh_\i\left(\C_\cD\right),$$ where $y$ denotes the Yoneda embedding. So $D$ itself is the colimit of the composite $$\C_\cD/R_\cD\left(D\right) \to \C_\cD \stackrel{y}{\longhookrightarrow} \Psh_\i\left(\C_\cD\right) \stackrel{L_\cD}{\longlongrightarrow} \cD.$$ Under the equivalence $\cD \simeq \sl\left(\sD\right)^{\et}/\cD,$ this implies that $$\cD/D \to \cD$$ is the colimit of the functor
\begin{eqnarray*}
\C_\cD/R_\cD\left(D\right) &\to& \sl\left(\sD\right)^{\et}/\cD\\
\left(C \to D \right) &\mapsto& \left(\cD/L_\cD y\left(C\right)\to \cD\right).\\
\end{eqnarray*}
Each $C$ in $\C_\cD$ corresponds to an \'etale map $$r_\cD\left(C\right):\cD_C \to \cD,$$ with $\cD_C$ in $\sD.$ Since $L_\cD=\Lan_y\left(i/\cD \circ r_\cD\right)$ (with notation as in Definition \ref{dfn:strongblossom}), this implies that for each $C$ $$\cD/L_\cD y\left(C\right)\simeq \left(\sD^{\et}/\cD\right)/r_\cD\left(C\right).$$ By Lemma \ref{lem:slicek}, this in turn implies $$\cD/L_{\cD}y\left(C\right) \simeq \sD^{\et}/\cD_C \simeq \cD_C \in \sD.$$
Since the canonical functor $$\sl\left(\sD\right)^{\et}/\cD \to \sl\left(\sD\right)^{\et}$$ preserves colimits, and each $\cD/L_\cD y\left(C\right)$ is in $\sD,$ this shows that $\cE$ is a colimit of a diagram in $\sD^{\et}.$
\end{proof}

\begin{cor}\label{cor:stetgen}
If $\sD$ is a strong \'etale blossom, then $\sD^{\et}$ generates $\dbare$ under colimits.
\end{cor}
\begin{proof}
By the proof of Proposition \ref{prop:stretex}, $\sD^{\et}$ generates $\sl\left(\sD\right)^{\et}$ under colimits, which in turn, by Proposition \ref{prop:gens}, generates $\dbare$.
\end{proof}

Since the definition of a strong \'etale blossom, on its surface, may seem complicated, one may be forgiven for asking whether or not strong \'etale blossoms exist, and if they are easy to construct. The following proposition (and its proof) shows that they indeed are in abundant supply:

\begin{prop}\label{prop:existbloss}
If $\sL$ is any full subcategory of $\Str\left(\G\right),$ then there exists a strong \'etale blossom $\sD$ for $\lbar.$ If $\sL$ is small, $\sD$ may be chosen small. If $\sL$ is $n$-localic, $\sD$ may be chosen $n$-localic. If $\sL$ is both $n$-localic and small, $\sD$ may be chosen to be also.
\end{prop}

\begin{proof}
For each $\i$-topos $\cL$ in $\sL,$ one may choose a small $\icat$ $\C_\cL$ such that $\cL$ is obtained as a left exact localization $$\xymatrix@1{\cL\mspace{4mu} \ar@{^{(}->}[r]<-0.9ex>_-{i_\cL} & \Psh_\i\left(\C_\cL\right) \ar@<-0.5ex>[l]_-{a_\cL}},$$ and such that the induced functor $$\C_\cL \stackrel{y}{\longhookrightarrow} \Psh_\i\left(\C_\cL\right) \stackrel{a_\cL}{\longlongrightarrow} \cL$$ is full and  faithful. If $\sL$ is $n$-localic, we may arrange each of the $\C_\cL$ to be $\left(n,1\right)$-categories.  Let $\sD$ the subcategory of $\sl\left(\sL\right)$ spanned by objects of the form $$\cL/a_\cL y\left(C\right)$$ for some $C \in \C_\cL.$ Notice that if $\sL$ is small, so is $\sD.$ Notice further that if $\C_\cL$ is an $\left(n,1\right)$-category, each object $C$ is $\left(n-1\right)$-truncated, so $\cL/a_\cL y\left(C\right)$ is $n$-localic by Proposition \ref{prop:2.3.16}. We claim that $\sD$ is a strong \'etale blossom. Indeed, let $$\cD=\cL/a_\cL y\left(C\right)$$ be an arbitrary object of $\sD.$ The $\icat$ $\sD^{\et}/\cD$ can be identified with the full subcategory of $\cD=\cL/a_\cL y\left(C\right)$ spanned by those morphisms $$f:e \to a_\cL y\left(C\right)$$ such that $\cD/f \simeq \cL/e$ is in $\sD.$ Under this  identification and the equivalence in Remark \ref{rmk:5.3.5.4}, the functor
\begin{equation}\label{eq:rd}
\C_\cL/C \stackrel{y}{\longhookrightarrow} \Psh_\i\left(\C_\cL/C\right) \stackrel{\sim}{\longrightarrow} \Psh_\i\left(\C_\cL\right)/y\left(C\right) \hookrightarrow \cL/a_\cL y\left(C\right)=\cD
\end{equation}
identifies $\C_\cL/C$ as a full subcategory $$r_\cD:\C_\cL/C \hookrightarrow \sD^{\et}/\cD$$ in such a way that, in the notation of Definition \ref{dfn:strongblossom}, $i/\cD \circ r_\cD$ is equivalent to the above functor (\ref{eq:rd}). By the proof of Proposition \ref{prop:locslice} and the uniqueness of left adjoints, the left Kan extension $\Lan_{y}\left(i/\cD \circ r_\cD\right)$ is a left exact (accessible) left adjoint since its right adjoint can be identified with the canonical fully faithful functor $$ \cD=\cL/a_\cL y\left(C\right)\hookrightarrow \Psh_\i\left(\C_\cL\right)/y\left(C\right).$$
\end{proof}

This proposition combined with Remark \ref{rmk:supblossom} immediately implies the following:

\begin{cor}\label{cor:largestblossom}
For any full subcategory $\sL$ of $\Str\left(\G\right),$ $\lbar$ is a strong \'etale blossom.
\end{cor}

\begin{prop}\label{prop:locblossom}
Suppose that $\sD$ is a strong \'etale blossom and $\cE$ is an object of $\dbar.$ Then there exists small full subcategory $\sD_\cE \hookrightarrow \sD$ such that $\sD_\cE$ is a small strong \'etale blossom, and $\cE$ is in $\overline{\sD_\cE}.$
\end{prop}

\begin{proof}
Since $\cE$ is in $\dbar=\cosl\left(\sD\right),$ there exists a set of objects $$\left(E_\alpha \in \cE_0\right)_{\alpha \in A_\cE},$$ such that $\left(\cE/E_\alpha \to \cE\right)_{\alpha \in A_\cE}$ is an \'etale covering family, and each $\cE/E_\alpha \simeq \cD_\alpha$ for some $\cD_\alpha$ in $\sD.$ Since $\sD$ is a strong \'etale blossom, for each $\alpha,$ one can find a small full subcategory $$r_{\cD_\alpha}:\C_{\cD_\alpha} \hookrightarrow \sD^{\et}/\cD_\alpha$$ such that, in the notation of Definition \ref{dfn:strongblossom},
$$\xymatrix@1{\cD_\alpha \mspace{4mu} \ar@{^{(}->}[r]<-0.9ex>_-{R_{\cD_\alpha}} & \Psh_\i\left(\C_{\cD_\alpha}\right) \ar@<-0.5ex>[l]_-{L_{\cD_\alpha}}},$$
is a left exact accessible localization. Let $\sD_\cE$ be the full subcategory of $\dbar$ on those objects $\cF$ of the form $\cD_\alpha/L_{\cD_\alpha}\left(y\left(C\right)\right),$ for some $\alpha,$ with $C$ in $\C_{\cD_\alpha}$. By the proof of Proposition \ref{prop:existbloss}, $\sD_\cE$ is a small strong \'etale blossom, and since each $\cD_\alpha$ is in $\sD_\cE,$ it follows that $\cE$ is in $\overline{\sD_\cE}.$ Finally, observe that if $C$ is in $\C_{\cD_\alpha},$ then $r_{\cD_\alpha}\left(C\right)$ is an \'etale map of the form $$\cD_\alpha/D \to \cD_\alpha$$ with $\cD_\alpha/D$ an object of $\sD,$ and $L_{\cD_\alpha}\left(y\left(C\right)\right)$ is equivalent to the object $D.$ It follows that $\sD_\cE$ is in fact a full subcategory of $\sD.$
\end{proof}

\begin{ex}\label{ex:affineblossom}
Let $\g$ is a geometry in the sense of \cite{dag} and let $\sL=\Affg$ be the $\icat$ of affine $\g$-schemes viewed as a subcategory of $\G\left(\g\right)$-structured $\i$-topoi. Each object $\cL$ has its underlying $\i$-topos of the form $\Sh_\i\left(\Pro\left(\g\right)^{ad}/X\right)$ (in the sense of Notation 2.2.6 of \cite{dag}) where $X$ is an object of $\Pro\left(\g\right)$, so letting $\C_\cL=\Pro\left(\g\right)^{ad}/X$ gives a canonical choice of site. Moreover, for every object $f:Y \to X$ of $\Pro\left(\g\right)^{ad}/X,$ $$\cL/y\left(f\right) \simeq \operatorname{Spec}_\g\left(Y\right).$$ It follows from the proof of Proposition \ref{prop:existbloss} that $\sL$ is a strong \'etale blossom. Moreover, if $\cE$ is a $\g$-scheme, then one can find an \'etale covering family $$\left(\operatorname{Spec}_\g\left(X_\alpha\right) \to \cE\right)_\alpha,$$ and the $\icat$ $\Affg^\cE$ of affine schemes of the form $$\operatorname{Spec}_\g\left(Y\right)$$ where $Y$ is any object of $\Pro\left(\g\right)$ admitting an admissible morphism
$$Y \to X_\alpha,$$ for some $X_\alpha,$ is a small strong \'etale blossom such that $\cE$ is in $\overline{\Affg^\cE}.$
\end{ex}

\section{The functor of points approach}\label{sec:functorofpoints}
In this subsection, we will explain how to represent \'etendues by sheaves over their basic building blocks. Many ideas for the proofs in this section come from Section 2.4 of \cite{dag}, however we develop these results in a more general setting.

Given a locally small subcategory $\sL$ of $\Str\left(\G\right),$ one may expect that it should be possible to represent $\sL$-\'etendues by the sheaves they induce over $\sL.$ This will not work in full generality however, as one may have to add to $\sL$ ``more open subsets,'' i.e. one can embed the $\icat$ of $\lbar$ into an $\icat$ of sheaves over a strong \'etale blossom $\sD$ for $\lbar.$ In fact, we will need something a bit stronger than local smallness for this to work:

\begin{dfn}
A subcategory $\sL$ of $\Str_\G$ is said to be \textbf{strongly locally small} if $\sl\left(\sL\right)$ is locally small.
\end{dfn}

\begin{rmk}
One can rephrase strong local smallness by a sheaf theoretic condition: A subcategory $\sL$ of $\Str_\G$ is strongly locally small if and only if for all $\cL$ and $\cL'$ in $\sL,$ the functor
$$\cL^{op} \times \cL' \stackrel{\sim}{\longrightarrow} \left(\lbar^{\et}/\cL\right)^{op} \times \lbar^{\et}/\cL'  \to \left(\lbar\right)^{op} \times \lbar \stackrel{\Hom_{\lbar}\left(\mspace{3mu} \cdot \mspace{3mu},\mspace{3mu} \cdot \mspace{3mu}\right)}{\longlonglongrightarrow} \LiGpd,$$ i.e. the functor sending $\left(L,L'\right)$ to $\Hom_{\lbar}\left(\cL/L,\cL'/L\right)$ factors through the inclusion $$\iGpd \hookrightarrow \LiGpd$$ of essentially small $\i$-groupoids into large $\i$-groupoids.
\end{rmk}

\begin{ex}
It follows from Proposition 2.3.12 of \cite{dag} that if $\sL$ is affine $\g$-schemes as in Example \ref{ex:affineblossom}, that $\sL$ is strongly locally small.
\end{ex}

\begin{ex}
If $\sL$ is small and locally small, then it does \emph{not} follow that it is strongly locally small. We would like to thank Anton Fetisov for explaining to us how to construct counterexamples. For simplicity, we will provide a counterexample for the case of trivial geometric structure. For example, suppose that $M$ is a locale with no points (e.g. the locale of regular open subsets of $\RR$), and let $\cB$ be the classifying topos for abelian groups. Consider the product (in the $2$-category of topoi) $$\cE:=\Sh\left(M\right) \times \cB.$$ Notice that there exists no geometric morphism $$\Set \to \cE$$ as such a morphism would induce a geometric morphism $\Set=\Sh\left(*\right) \to \Sh\left(M\right),$ which would correspond to a point of the locale $M$. By \cite{ext}, there exists an open surjection $\Sh\left(L\right) \to \cE$ with $L$ a locale. Let $\cF$ denote the topos $$\cF:=\Sh\left(L \coprod *\right)\simeq \Sh\left(L\right) \coprod \Set.$$ Notice that the there cannot exist a geometric morphism $\cF \to \cE,$ as this would induce a geometric morphism $$\Set \to \cE$$ which we just proved cannot exist. Since $\cF$ is localic, the category of geometric morphisms from $\cE$ to $\cF$ is equivalent to the poset of continuous maps from the localic reflection of $\cE$ to $L \coprod *,$ and hence is essentially small. It follows that if $\sL$ is the full subcategory of topoi on $\cF$ and $\cE$ (which we may identify with a full subcategory of $\i$-topoi), then $\sL$ is locally small. Observe that the \'etale map of locales $L \to L \coprod *$ corresponds to an \'etale geometric morphism $$\Sh\left(L\right) \to \cF,$$ hence $\Sh\left(L\right)$ (or rather $\Shi\left(L\right)$) is in $\sl\left(\sL\right).$ However, the category of geometric morphism from $\Sh\left(L\right)$ to $\cE$ is equivalent to the product 
$$\Geo\left(\Sh\left(L\right),\Sh\left(M\right)\right) \times Ab\left(\Sh\left(L\right)\right),$$
where $Ab\left(\Sh\left(L\right)\right)$ is the category of sheaves of abelian groups on $L.$ Consider the composite $$\Sh\left(L\right) \to \Sh\left(M\right) \times \cB \to \Sh\left(M\right)$$ where the map $\Sh\left(L\right) \to \Sh\left(M\right) \times \cB=\cE$ is the open surjection. The existence of this map implies that $\Geo\left(\Sh\left(L\right),\Sh\left(M\right)\right)$ is non-empty, and hence one concludes that $\Geo\left(\Sh\left(L\right),\cE\right)$ is not essentially small, as the category of abelian sheaves can easily be shown to not be essentially small. It follows that $\sL$ is not strongly locally small.
\end{ex}

\begin{prop}\label{prop:etlocsmall}
For $\G$ any geometric structure, $\Str_\G^{\et}$ is locally small.
\end{prop}
\begin{proof}
We will first prove that $\T^{\et}$ is locally small. We would like to thank Zhen Lin for pointing out a proof of this for $1$-topoi, whose proof readily generalizes as follows: Let $\cE$ and $\cF$ be two $\i$-topoi. We may choose a regular cardinal $\kappa$ such that both $\cE$ and $\cF$ are $\kappa$-accessible. If $f:\cE \to \cF$ is an \'etale geometric morphism, since $f^*$ is a left adjoint, it preserves all colimits, and since $\cE$ is generated under colimits by its $\kappa$-compact objects, it follows that $f^*$ is uniquely determined (up to equivalence) by what it does on $\kappa$-compact objects. Notice that since $f$ is \'etale, there is a further adjunction $$f_! \dashv f^*,$$ and so it follows from Proposition 5.4.7.7 of \cite{htt} that $f_!$ (and also $f^*$) is $\kappa$-accessible. Combining this observation with Proposition 5.5.1.4 of \cite{htt}, we conclude that $f^*$ must carry $\kappa$-compact objects to $\kappa$-compact objects.  Since the subcategory of $\kappa$-compact objects in either $\i$-topos is essentially small, one concludes that $\pi_0\left(\Hom_{\T^{\et}}\left(\cE,\cF\right)\right)$ is small. Since the $\i$-category $\Geo^{\et}\left(\cE,\cF\right)$ of \'etale geometric morphisms from $\cE$ to $\cF$ is a full subcategory of $\Geo\left(\cE,\cF\right)$ which is locally essentially small by Remark \ref{rmk:loclocsmall}, we conclude that the former $\i$-category is also locally essentially small, so the result now follows.

Now consider the case where $\G$ is not trivial and let $\cB$ be the base $\i$-topos for $\G.$ Suppose that $\left(\cE,\O_\cE\right)$ and $\left(\cF,\O_\cF\right)$ are objects of $\Str_\G.$ Up to equivalence, the projection $$\Hom_{\Str_\G^{\et}}\left(\left(\cE,\O_\cE\right),\left(\cF,\O_\cF\right)\right) \to \Hom_{\T^{\et}}\left(\cE,\cF\right)$$ can be identified with the left fibration classifying the functor
\begin{eqnarray*}
\Hom_{\T^{\et}}\left(\cE,\cF\right)^{op} &\to& \LiGpd\\
f:\cE \to \cF &\mapsto& \Hom_{\Hom_{\T}\left(\cE,\cB\right)}\left(\O_\cF \circ f,\O_\cE\right).
\end{eqnarray*}
It suffices to show the above functor factors through the inclusion $$\iGpd \hookrightarrow \LiGpd.$$ Since $\Hom_{\T}\left(\cE,\cB\right)$ is locally small by Remark \ref{rmk:loclocsmall}, we are done.
\end{proof}

\begin{prop}\label{prop:stronglocsmallblossom}
If $\sL$ is strongly locally small, then there exists a locally small strong \'etale blossom $\sD$ for $\lbar,$ which can be chosen to be small if $\sL$ is small, and can be chosen to be $n$-localic if $\sL$ is.
\end{prop}

\begin{proof}
The construction of the strong \'etale blossom in Proposition \ref{prop:existbloss} is a subcategory of $\sl\left(\sL\right).$
\end{proof}

Let $\cE$ and $\cD$ be two objects in $\Str_\G.$  Consider the functor
$$\cD^{op} \times \cE \stackrel{\sim}{\longrightarrow} \left(\Str_\G^{\et}/\cD\right)^{op} \times \Str_\G^{\et}/\cE  \to \left(\Str_\G\right)^{op} \times \Str_\G \stackrel{\Hom_{\Str_\G}\left(\mspace{3mu} \cdot \mspace{3mu},\mspace{3mu} \cdot \mspace{3mu}\right)}{\longlonglongrightarrow} \LiGpd.$$
By adjunction, it determines a functor $$\tilde\tau^\cE_\cD:\cE \to \LPshi\left(\cD\right),$$ such that for all $E$ in $\cE$ and $D$ in $\cD,$ $$\tilde \tau^\cE_\cD\left(E\right)\left(D\right) \simeq \Hom_{\dbar}\left(\cD/D,\cE/E\right).$$ Observe that for all $E$ in $\cE,$ $\tau^\cE_\cD\left(E\right):\cD^{op} \to \LiGpd$ preserves small limits, hence is a sheaf in the sense of Definition \ref{dfn:sheavesoverE}. Denote the induced functor by $$\tau^\cE_\cD:\cE \to \LShi\left(\cD\right).$$ Analogously, consider the functor
$$\cD^{op} \times \cE \stackrel{\sim}{\longrightarrow} \left(\Str_\G^{\et}/\cD\right)^{op} \times \Str_\G^{\et}/\cE  \to \left(\Str_\G^{\et}\right)^{op} \times \dbare \stackrel{\Hom_{\Str_\G^{\et}}\left(\mspace{3mu} \cdot \mspace{3mu},\mspace{3mu} \cdot \mspace{3mu}\right)}{\longlonglongrightarrow} \LiGpd,$$ which induces a functor
$$\tau^{\cE,{\et}}_\cD:\cE \to \LShi\left(\cD\right).$$

\begin{lem}\label{lem:taucolim}
Let $\cE$ and $\cD$ be two objects in $\Str_\G.$ Then both functors $\tau^\cE_\cD$ and $\tau^{\cE,{\et}}_\cD$  preserve small colimits.
\end{lem}

\begin{proof}
This proof is based off of the proof of Proposition 2.3.12 in \cite{dag}. We will prove both cases at once, and denote by $\tau$ a functor which could represent either $\tau^\cE_\cD$ or $\tau^{\cE,{\et}}_\cD,$ and distinguish when necessary. Let
$$J^{\triangleright} \stackrel{\rho}{\longrightarrow} \cE$$ be a colimit for a functor $$F:J \to \cE$$ with $J$ a small $\icat$. We will informally write this colimit as $E=\colim E_\alpha,$ and the components of the cocone as $$\rho_\alpha:E_\alpha \to E.$$ We wish to show that $$\colim \tau\left(E_\alpha\right) \simeq \tau\left(\colim E_\alpha\right).$$ Since colimits in $\LShi\left(\cD\right)$ are universal, this holds if and only if for each map from a representable $\varphi:y\left(D\right) \to \tau\left(E\right),$ one has $$\colim y\left(D\right) \times_{\tau\left(E\right)} \tau\left(E_\alpha\right) \simeq y\left(D\right).$$ By the Yoneda lemma, such a map $\varphi$ corresponds to a morphism $\tilde \varphi:\cD/D \to \cE/E$ (which is \'etale if $\tau=\tau^{\cE,{\et}}_\cD$). Denote the inverse image functor of the underlying geometric morphism by $$\varphi^*:\cE/E \to \cD/D.$$ Suppose that $f:D' \to D$ is a morphism in $\cD.$ Since we have a pullback diagram
$$\xymatrix{y\left(D\right) \times_{\tau\left(E\right)} \tau\left(E_\alpha\right) \ar[r] \ar[d]_-{pr_\alpha} & \tau\left(E_\alpha\right) \ar[d]^-{\tau\left(\rho_\alpha\right)} \\
y\left(D\right) \ar[r]^-{\varphi} & \tau\left(E\right),}$$
there is a canonical equivalence $$\Hom_{\LShi\left(\cD\right)/y\left(D\right)}\left(y\left(f\right),pr_\alpha\right) \simeq \Hom_{\LShi\left(\cD\right)/\tau\left(E\right)}\left(\varphi \circ y\left(f\right),\tau\left(\rho_\alpha\right)\right).$$
By the Yoneda lemma, we have
$$\Hom_{\LShi\left(\cD\right)/\tau\left(E\right)}\left(\varphi \circ y\left(f\right),\tau\left(\rho_\alpha\right)\right) \simeq \Hom_{\Str_\G/\left(\cE/E\right)}\left(\left(\cD/D\right)/f \to\cD/D \stackrel{\tilde \varphi}{\to} \cE/E,\cE/E_\alpha \to \cE/E\right),$$ which in turn is equivalent to
$$\Hom_{\Str_\G/\left(\cD/D\right)}\left(\left(\cD/D\right)/f \to \cD/D,\left(\cD/D\right)/\varphi^*\rho_\alpha \to  \cD/D\right)$$ by Corollary \ref{cor:2.3.20}. Since \'etale maps are stable under pullback, each of the morphisms are \'etale, therefore any morphism between them is \'etale (Proposition \ref{prop:6.3.5.9}) and hence this $\i$-groupoid is equivalent to
$$\Hom_{\Str^{\et}_\G/\left(\cD/D\right)}\left(\left(\cD/D\right)/f \to \cD/D,\left(\cD/D\right)/\varphi^*\rho_\alpha \to  \cD/D\right).$$ Finally, under the equivalence $\Str_\G^{\et}/\left(\cD/D\right)\simeq \cD/D,$ one has
$$\Hom_{\Str^{\et}_\G/\left(\cD/D\right)}\left(\left(\cD/D\right)/f \to \cD/D,\left(\cD/D\right)\varphi^*\rho_\alpha \to  \cD/D\right) \simeq \Hom_{\cD/D}\left(f,\varphi^*\left(\rho_\alpha\right)\right).$$ Let $$\pi_D:\cD/D \to \cD$$ be the canonical projection. By Remark \ref{rmk:5.3.5.4}, $$\LPshi\left(\cD\right)/y\left(D\right)\simeq \LPshi\left(\cD/D\right)$$ so from the equivalence 
$$\Hom_{\LShi\left(\cD\right)/y\left(D\right)}\left(y\left(f\right),pr_\alpha\right) \simeq \Hom_{\cD/D}\left(f,\varphi^*\left(\rho_\alpha\right)\right)$$
it follows that one has $$pr_\alpha \simeq y_{\cD/D}\left(\varphi^*\rho_\alpha\right)$$ and hence $$y\left(D\right) \times_{\tau\left(E\right)} \tau\left(E_\alpha\right) \simeq y\left(\pi_D \varphi^* \rho_\alpha\right).$$ By Remark 6.3.5.17 of \cite{htt}, the Yoneda embedding $$y:\cD/D \hookrightarrow \LShi\left(\cD\right)$$ preserves small colimits, and $\varphi^*$ and $\pi_D$ do as well as they are left adjoint. By Proposition \ref{prop:colim1}, $$\colim \rho_\alpha \simeq id_E,$$ so that $$\colim \pi_D \varphi^* \rho_\alpha \simeq E.$$ It follows that

\begin{eqnarray*}
\colim \left(y\left(D\right) \times_{\tau\left(E\right)} \tau\left(E_\alpha\right)\right) &\simeq& \colim y\left(\pi_D \varphi^*\rho_\alpha\right)\\
&\simeq& y\left(\pi_D \varphi^*\colim \rho_\alpha\right)\\
&\simeq& y\left(\pi_D id_D\right)\\
&\simeq& y\left(D\right)
\end{eqnarray*}
as desired.
\end{proof}

\begin{prop}\label{prop:locallsmall}
If $\sD$ is a locally small strong \'etale blossom, then $\dbar$ is locally small.
\end{prop}

\begin{proof}
First, suppose that $\cD$ is in $\dbar$ and let $Z_\cD$ be full subcategory of $\dbar^{\et}$ on those $\cE$ such that $\Hom\left(\cE,\cD\right)$ is essentially small. Notice that $Z_\cD$ is closed under colimits, and since $\sD$ is locally small, it also contains $\sD^{\et}.$ By Corollary \ref{cor:stetgen}, $\sD^{\et}$ generates $\dbare$ under colimits, so it follows that $Z_\cD=\dbare,$ and hence $\Hom\left(\cE,\cD\right)$ is essentially small for all $\cE$ in $\dbar.$


Now suppose that $\cE$ and $\cF$ are arbitrary objects of $\dbar.$ Again by Corollary \ref{cor:stetgen}, $\cF$ be written as an iterative colimit in $\dbare$ using objects in $\sD.$ This observation combined with Proposition \ref{prop:colim1} implies that $id_\cF$ may be written as an iterative colimit in $\dbare/\cF$ using objects in $\sD^{\et}/\cF.$ Under the equivalence $\dbare/\cF\simeq \cF,$ this implies that $1_\cF$ may be written as an iterative colimit in $\cF$ using objects $F$ such that
$\cF/F$ is in $\sD.$ Consider the functor $$\tau^\cF_\cE:\cF \to \LShi\left(\cE\right),$$ which preserves colimits by Lemma \ref{lem:taucolim}. It follows that the $\i$-sheaf $\tau^\cF_\cE\left(1_\cF\right)$ may be written as an iterative colimit of $\i$-sheaves of the form $\tau^\cF_\cE\left(F\right),$ with $\cF/F$ in $\sD.$ Notice that for all $E$ in $\cE,$ $$\tau^\cF_\cE\left(F\right)\left(E\right)=\Hom\left(\cE/E,\cF/F\right),$$ hence is a small $\i$-groupoid. So, $\tau^\cF_\cE\left(F\right)$ may be written as an iterative colimit of small $\i$-sheaves, hence is again a small sheaf by Proposition 6.3.5.17 of \cite{htt}. Therefore $$\tau^\cF_\cE\left(1_\cF\right)\left(1_\cE\right)=\Hom\left(\cE,\cF\right)$$ is an essentially small $\i$-groupoid as desired.


\end{proof}

For the rest of section, unless otherwise noted, $\sD$ will be a \emph{locally small strong \'etale blossom.} By Proposition \ref{prop:locallsmall}, both functors $\tau^\cE_\cD$ and $\tau^{\cE,{\et}}_\cD$ restrict to functors (which we will write by the same symbols)
$$\tau^{\cE}_\cD:\cE \to \Shi\left(\cD\right),$$ and
$$\tau^{\cE,{\et}}_\cD:\cE \to \Shi\left(\cD\right).$$

\begin{dfn}
Recall from Proposition \ref{prop:largesite} that given an object $\cD$ of a strong \'etale blossom $\sD,$ the underlying $\i$-topos of $\cD$ can be written canonically as a left exact localization $$\xymatrix@C=1.5cm@R=1.5cm{\cD \ar@{^{(}->}[r]<-0.9ex>_-{i_\cD} & \Psh_\i\left(\sD^{\et}/\cD\right) \ar@<-0.5ex>[l]_-{a_\cD}.}$$
A functor $F:\sD^{\et}/\cD \to \iGpd$ is an \textbf{$\i$-sheaf} if it is in the essential image of $i_\cD.$ Equivalently, $F$ is an $\i$-sheaf if and only if the co-unit $F \to i_\cD a_\cD\left(F\right)$ is an equivalence. Denote by $\Sh_\i\left(\sD^{\et}/\cD\right)$ the full subcategory of $\Psh_\i\left(\sD^{\et}/\cD\right)$ on the sheaves.
\end{dfn}

\begin{rmk}\label{rmk:immm}
It follows immediately that $\Sh_\i\left(\sD^{\et}/\cD\right)$ is canonically equivalent to $\cD$.
\end{rmk}

\begin{dfn}\label{dfn:localargesheaf}
Let $\cD$ be an object of a strong \'etale blossom $\sD$. Denote by $S_\cD$ the collection of morphisms $f$ in $\Psh_\i\left(\sD^{\et}/\cD\right)$ such that $a_\cD\left(f\right)$ is an equivalence, and denote by $\widehat{S}_\cD$ its saturation in the $\i$-category $\LPshi\left(\sD^{\et}/\cD\right).$ A presheaf of large $\i$-groupoids $F:\sD^{\et}/\cD \to \LiGpd$ is an \textbf{$\i$-sheaf} if it is $\widehat{S}_\cD$-local. Denote by $\LShi\left(\sD^{\et}/\cD\right)$ the full subcategory of $\LShi\left(\sD^{\et}/\cD\right)$ on the sheaves, and denote the corresponding localization by
$$\xymatrix{\LShi\left(\sD^{\et}/\cD\right) \ar@{^{(}->}[r]<-0.9ex>_-{\widehat{i}_\cD} & \LPshi\left(\sD^{\et}/\cD\right). \ar@<-0.5ex>[l]_-{\widehat{a}_\cD}}$$
\end{dfn}

\begin{prop}\label{prop:pop}
For $\cD$ an object of a strong \'etale blossom $\sD$, the is a canonical equivalence $$\LShi\left(\sD^{\et}/\cD\right) \simeq \LShi\left(\cD\right),$$ where $\LShi\left(\cD\right)$ is the $\i$-category of functors $F:\cD^{op} \to \LiGpd$ which preserve small limits.
\end{prop}

\begin{proof}
Choose a small subcategory $$r_\cD:\C_\cD \hookrightarrow \sD^{\et}/\cD$$ satisfying the conditions of Definition \ref{dfn:strongblossom}, and let
$$\xymatrix@C=1.5cm@R=2.5cm{\cD \ar@{^{(}->}[r]<-0.9ex>_-{R_\cD} & \Psh_\i\left(\C_\cD\right) \ar@<-0.5ex>[l]_-{L_{\cD}}}$$
denote the corresponding localization. Let $S_\C$ denote the collection of morphisms $f$ in $\Psh_\i\left(\C_\cD\right)$ such that $L_\cD\left(f\right)$ is an equivalence. Denote by $\widehat{S}_\C$ the saturation of this class of morphisms in $\LPshi\left(\C_\cD\right).$ Denote by $$\xymatrix@C=1.5cm{\LPshi\left(\C_\cD\right)\ar@{^{(}->}[r]<-0.9ex>_-{\widehat{\left(r_\cD\right)}_*} & \LPshi\left(\sD^{\et}/\cD\right) \ar@<-0.5ex>[l]_-{\widehat{\left(r_\cD\right)}^*}}$$ the geometric morphism of $\i$-topoi in the universe of large sets induced by $r_\cD.$ Since $\widehat{\left(r_\cD\right)}^*$ preserves colimits, one can identify $\widehat{S}_\cD$ with the collection of morphisms $g$ in $\LPshi\left(\sD^{\et}/\cD\right)$ such that $\widehat{\left(r_\cD\right)}^*\left(g\right)$ is in $\widehat{S}_\C$. It follows that $$\LShi\left(\sD^{\et}/\cD\right) \simeq \widehat{S}_\C^{-1}\LPshi\left(\C_\cD\right).$$ By the proof of Remark 6.3.5.17 of \cite{htt}, $$\widehat{S}_\C^{-1}\LPshi\left(\C_\cD\right) \simeq \LShi\left(\cD\right),$$ so we are done.
\end{proof}

\begin{dfn}\label{dfn:sheaveDet}
For $\cD$ an object of a strong \'etale blossom $\sD,$ consider the canonical projection $$\pi_\cD:\sD^{\et}/\cD \to \sD^{\et}.$$ A functor $$F:\left(\sD^{\et}\right)^{op} \to \iGpd$$ is an \textbf{$\i$-sheaf} if for all objects $\cD$ of $\sD,$ $\pi_\cD^*\left(F\right)$ is an $\i$-sheaf on $\cD$. Denote the full subcategory of $\Psh_\i\left(\sD^{\et}\right)$ on the sheaves by $\Sh_\i\left(\sD^{\et}\right).$ Denote the analogously defined subcategory of $\LPshi\left(\sD^{\et}\right)$ of sheaves by $\LShi\left(\sD^{\et}\right).$
\end{dfn}

\begin{dfn}\label{dfn:sheaveD}
Let $\sD$ be a strong \'etale blossom and denote by $j:\sD^{\et} \to \sD$ the canonical functor. A functor $$F:\sD^{op} \to \iGpd$$ is an \textbf{$\i$-sheaf} if $j^*F$ is an $\i$-sheaf over $\sD^{\et}.$ Denote the full subcategory of $\Pshi\left(\sD\right)$ on the sheaves by $\Sh_\i\left(\sD\right)$ and the analogously defined subcategory of $\LPshi\left(\sD\right)$ of sheaves by $\LShi\left(\sD\right).$
\end{dfn}

\begin{rmk}\label{rmk:affshvs}
If $\sD$ is the strong \'etale blossom of affine $\g$-schemes as in Example \ref{ex:affineblossom}, by unwinding the definitions, one sees that the equivalence $$\Specg:\Pro\left(\g\right) \stackrel{\sim}{\longlongrightarrow} \Affg$$ induces a canonical equivalence $$\Shi\left(\Affg\right) \stackrel{\sim}{\longlongrightarrow} \Shi\left(\Pro\left(\g\right)\right),$$ where $\Shi\left(\Pro\left(\g\right)\right)$ is as in Definition 2.4.3 of \cite{dag} (where it is denoted $\operatorname{Shv}\left(\Pro\left(\g\right)\right)$).
\end{rmk}

\begin{prop}\label{prop:loclimsheaf}
Consider the strong \'etale blossom $\dbar.$ For a presheaf $$F:\left(\dbare\right)^{op} \to \iGpd,$$ the following conditions are equivalent:
\begin{itemize}
\item[i)] $F$ is an $\i$-sheaf.
\item[ii)] For all objects $\cE$ of $\dbar,$ $$\left(\dbare/\cE\right)^{op} \stackrel{\pi^{op}_\cE}{\longlonglongrightarrow} \left(\dbare\right)^{op} \stackrel{F}{\longrightarrow} \iGpd$$ preserves small limits.
\item[iii)] For all objects $\cE$ of $\dbar,$ the above functor $\pi_\cE^*\left(F\right)$ is representable.
\end{itemize}
\end{prop}
\begin{proof}
A presheaf $F$ is an $\i$-sheaf if and only if $\pi_\cE^*\left(F\right)$ is an $\i$-sheaf for all $\cE$ in $\dbar.$ An $\i$-sheaf on $\dbare/\cE,$ by definition, is a presheaf in the essential image of  $$\cE \stackrel{y}{\longhookrightarrow} \Pshi\left(\cE\right) \stackrel{\theta^*}{\longrightarrow} \Pshi\left(\dbare/\cE\right),$$ where $$\theta:\dbare/\cE \stackrel{\sim}{\longrightarrow} \cE$$ is the equivalence. Since $\theta^*$ sends representables to representables, one concludes that $i) \Leftrightarrow iii).$ Notice that $ii) \Leftrightarrow iii)$ follows from Remark \ref{rmk:limitpresrep}.
\end{proof}

\begin{prop}\label{prop:subcanonical}
Denote by $$i_\sD:\sD \hookrightarrow \overline{\sD}$$ and $$i^{\et}_\sD:\sD^{\et} \hookrightarrow \dbare$$ the canonical inclusions. Then, the functors $$\dbar \stackrel{y}{\longhookrightarrow} \Pshi\left(\dbar\right) \stackrel{i^*_\sD}{\longlongrightarrow} \Pshi\left(\sD\right)$$
and
$$\dbare \stackrel{y^{\et}}{\longlonghookrightarrow} \Pshi\left(\dbare\right) \stackrel{\left(i^{\et}_\sD\right)^*}{\longlongrightarrow} \Pshi\left(\sD^{\et}\right)$$
factor through the inclusions $$i:\Shi\left(\sD\right) \hookrightarrow \Pshi\left(\sD\right)$$
and $$i^{\et}:\Shi\left(\sD^{\et}\right) \hookrightarrow \Pshi\left(\sD^{\et}\right)$$ respectively, where $y$ and $y^{\et}$ are the appropriate Yoneda embeddings.
\end{prop}
\begin{proof}
Suppose that $\cE$ is in $\dbar.$ We want to show that $y\left(\cE\right)$ and $y^{\et}\left(\cE\right)$ are sheaves. Therefore, using the notation of Definitions \ref{dfn:sheaveDet} and \ref{dfn:sheaveD}  we need to show for all $\cD$ in $\sD,$ $j^*\pi_\cD^*y\left(\cE\right)$ and $\pi_\cD^*y^{\et}\left(\cE\right)$ are sheaves on $\sD^{\et}/\cD.$ Suppose that $$f:\cD/D \to \cD$$ is an \'etale map, with $\cD/D$ in $\sD.$
Notice that
$$\left(j^*\pi_\cD^*i_\cD^*y\left(\cE\right)\right)\left(f\right) \simeq \Hom_{\dbar}\left(\cD/D,\cE\right)$$
and
$$\left(\pi_\cD^* \left(i^{\et}_\cD\right)^* y^{\et}\left(\cE\right)\right)\left(f\right) \simeq \Hom_{\dbare}\left(\cD/D,\cE\right).$$
Denote by $\theta:\dbare/\cD \stackrel{\sim}{\longrightarrow} \cD,$ and denote by $$q:\Shi\left(\cD\right) \hookrightarrow \Pshi\left(\cD\right)$$ the canonical inclusion. Notice that $$\pi_D\left(f\right)\simeq D \simeq \theta\circ i/\cD\left(f\right),$$ so that one has canonical equivalences
$$j^*\pi_\cD^*i_\cD ^*y\left(\cE\right) \simeq \left(\left(i/\cD\right)^*\theta^*q\left(\tau^\cE_\cD\left(1_\cE\right)\right)\right)$$
and
$$\pi_\cD^*\left(i^{\et}_\cD\right)^*y^{\et}\left(\cE\right) \simeq \left(\left(i/\cD\right)^*\theta^*q\left(\tau^{\cE,{\et}}_\cD\left(1_\cE\right)\right)\right).$$
By Proposition 6.3.5.17 of \cite{htt}, there exists an equivalence $$\cD \stackrel{\sim}{\longrightarrow} \Shi\left(\cD\right)$$ such that the following diagram commutes:
$$\xymatrix{ \Shi\left(\cD\right) \ar@{_{(}->}[rd]^-{q} & \\
\cD \ar@{^{(}->}[r]^-{y} \ar[u]^-{\rotatebox[origin=C]{90}{$\sim$}} & \Pshi\left(\cD\right).}$$
Therefore, there exists $D_\cE$ and $D'_\cE$ in $\cD$ such that for all $D$ in $\cD$,
$$\Hom\left(D,D_\cE\right)\simeq \tau^{\cE}_\cD\left(1_\cE\right)\left(D\right)$$
and
$$\Hom\left(D,D'_\cE\right)\simeq \tau^{\cE,{\et}}_\cD\left(1_\cE\right)\left(D\right).$$ This implies that both $j^*\pi_\cD^*y\left(\cE\right)$ and $\pi_\cD^*y^{\et}\left(\cE\right)$ are in the image of $$i_\cD=\left(i/\cD\right)^* \circ \theta^* \circ y.$$
\end{proof}

\begin{cor}
Every representable presheaf on $\sD$ and $\sD^{\et}$ is an $\i$-sheaf.
\end{cor}

\begin{rmk}\label{rmk:moresheaves}
Notice that nowhere in the proof of Proposition \ref{prop:subcanonical} is it needed that $\cE$ is actually in $\dbar$; any object in $\Str_\G$ would do, so there are induced functors $$\Str_\G \to \Shi\left(\sD\right)$$ and $$\Str^{\et}_\G \to \Shi\left(\sD^{\et}\right).$$
\end{rmk}

\begin{dfn}\label{dfn:lowerstar}
Let $f:\cD \to \cD'$ be a morphism in $\sD.$ Denote by $\widetilde{f}$ the functor
\begin{equation}
\widetilde{f}:\sD^{\et}/\cD' \to \sD^{\et}/\cD
\end{equation}
induced by pulling back along $f,$ i.e. an \'etale map $\cE' \to \cD'$ is sent to the map $\cD \times_{\cD'} \cE' \to \cD,$ which is \'etale since \'etale maps are stable under pullback.
\end{dfn}

\begin{prop}\label{prop:lowerstar}
Suppose that $f:\cD \to \cD'$ in $\sD$ has an underlying geometric morphism of $\i$-topoi
$$\xymatrix@C=1.5cm@R=2.5cm{\cD \ar@<-0.5ex>[r]_-{f_*} & \cD' \ar@<-0.5ex>[l]_-{f^*}}.$$
Then the following square
$$\xymatrix{\cD \ar@{^{(}->}[r]^-{i_\cD} \ar[d]_-{f_*} & \Pshi\left(\sD^{\et}/\cD\right) \ar[d]^-{\widetilde{f}^*} \\
\cD' \ar@{^{(}->}[r]^-{i_\cD'} & \Pshi\left(\sD^{\et}/\cD'\right)}$$ commutes up to canonical homotopy.
\end{prop}

\begin{proof}

Let $D$ be an object of $\cD.$ On one hand, if $E$ is an object of $\cD'$ such that $\cD'/E$ is in $\sD,$ one has that
\begin{eqnarray*}
i_\cD f_*\left(D\right)\left(\cD'/E \to \cD'\right) &\simeq& \Hom\left(E,f_*D\right)\\
&\simeq& \Hom\left(f^*E,D\right).
\end{eqnarray*}
On the other hand, one has the following string of equivalences:
\begin{eqnarray*}
\widetilde{f}^*\left(i_\cD\left(D\right)\right)\left(\cD'/E \to \cD'\right) &\simeq& i_\cD\left(D\right)\left(\cD/f^*E \to \cD\right)\\
&\simeq& \Hom\left(f^*E,D\right).
\end{eqnarray*}
\end{proof}

\begin{dfn}\label{dfn:upperstar}
Let $\varphi:\cE \to \cD$ be an \'etale map in $\sD.$ Denote by $\tp$ the functor
\begin{eqnarray*}
\sD^{\et}/\cE &\to& \sD^{\et}/\cD\\
\lambda &\mapsto& \varphi \circ \lambda.
\end{eqnarray*}
\end{dfn}

\begin{prop}\label{prop:upperstar}
Let $\varphi:\cE \to \cD$ be an \'etale map in $\sD.$
Then, the following diagram commutes up to canonical homotopy:
$$\xymatrix{\cD \ar[r]^-{\varphi^*} \ar@{^{(}->}[d]_-{i_\cD} & \cE \ar@{^{(}->}[d]^-{i_\cE}\\
\Pshi\left(\sD^{\et}/\cD\right) \ar[r]^-{\tp^*}  & \Pshi\left(\sD^{\et}/\cE\right).}$$
\end{prop}

\begin{proof}
Let $E \in \cD$ be the object corresponding to $\varphi:\cE \to \cD$ under the equivalence $\overline{\cD}^{\et}/\cD \simeq \cD.$ Let $D$ be an object of $\cD$ and let $x:E' \to E$ be an object of $\cD/E\simeq \cE$ such that $\cE/x$ is in $\sD.$ On one hand, one has the following:
\begin{eqnarray*}
i_\cE\varphi^*\left(D\right)\left(\cE/x \to \cE\right) &\simeq& \Hom\left(x,\varphi^*D\right)\\
&\simeq& \Hom\left(\varphi_!x,D\right)\\
&\simeq& \Hom\left(E',D\right),
\end{eqnarray*}
where $\varphi_!$ is the left adjoint to $\varphi^*$ (See Chapter \ref{chap:etale}).
On the other hand:
\begin{eqnarray*}
\tp^*\left(i_\cD\left(D\right)\right)\left(\cE/x \to \cE\right) &\simeq& \Hom\left(\tp\left(\cE/x \to \cE\right),\cD/D \to \cD\right)\\
&\simeq& \Hom\left(\left(\cD/E\right)/x \to \cD/E \to \cD, \cD/D \to \cD\right)\\
&\simeq& \Hom\left(\cD/E' \to \cD,\cD/D \to \cD\right)\\
&\simeq& \Hom\left(E',D\right),
\end{eqnarray*}
with the last equivalence following from the equivalence $\overline{\cD}^{\et}/\cD \simeq \cD.$
\end{proof}

\begin{rmk}
If $\varphi:\cE \to \cD$ is an \'etale map in $\sD,$ then it follows by the universal property for pullbacks that $\tp$ is left adjoint to $\wt$, and $\tp^*$ is left adjoint to $\wt^*.$
\end{rmk}

\begin{prop}\label{prop:upperstar2}
Let $\varphi:\cE \to \cD$ be an \'etale map in $\sD.$
Then, the following diagram commutes up to canonical homotopy:
$$\xymatrix{\Pshi\left(\sD^{\et}/\cD\right) \ar[d]_-{a_\cD} \ar[r]^-{\tp^*} & \Pshi\left(\sD^{\et}/\cE\right) \ar[d]^-{a_\cE}\\
 \cD \ar[r]^-{\varphi^*}  & \cE.}$$
\end{prop}

\begin{proof}
As both composites $\varphi^* \circ a_\cD$ and $a_\cE \circ \tp^*$ are colimit preserving, it suffices to show that they agree on representables. Let $E \in \cD$ be the object corresponding to $\varphi:\cE \to \cD$ under the equivalence $\overline{\cD}^{\et}/\cD \simeq \cD.$ Let $D$ be an object of $\cD$ such that $\cD/D$ is in $\sD,$ and let $x:E' \to E$ be an object of $\cD/E\simeq \cE$ such that $\cE/x$ is in $\sD.$ Since $$a_\cD y\left(\cD/D \to \cD\right)\simeq D,$$ by the proof of Proposition \ref{prop:upperstar}, one has that $$\left(i_\cE\varphi^*a_Dy\left(\cD/D \to \cD\right)\right)\left(\cE/x \to \cE\right) \simeq  \Hom\left(E',D\right).$$ Also by the proof of Proposition \ref{prop:upperstar}, since $$y\left(\cD/D \to \cD\right)=i_\cD\left(D\right),$$ it follows that $$\left(\tp^*y\left(\cD/D \to \cD\right)\right)\left(\cE/x \to \cE\right) \simeq \Hom\left(E',D\right).$$ Hence one has that $$\tp^*y\left(\cD/D \to \cD\right) \simeq i_\cE\varphi^*a_Dy\left(\cD/D \to \cD\right),$$ so that
\begin{eqnarray*}
a_\cE \tp^*y\left(\cD/D \to \cD\right) &\simeq& a_\cE i_\cE\varphi^*a_Dy\left(\cD/D \to \cD\right)\\
&\simeq& \varphi^*a_Dy\left(\cD/D \to \cD\right).
\end{eqnarray*}
\end{proof}

\begin{prop}\label{prop:sheafification}
The inclusion $$i:\Pshi\left(\sD\right) \hookrightarrow \Shi\left(\sD\right)$$ has a left exact left adjoint $a.$ Moreover, if $\sD$ is small, $\Shi\left(\sD\right)$ is an $\i$-topos.
\end{prop}

\begin{proof}
The proof is similar in spirit to Proposition 2.4.4 of \cite{dag} and Lemma 6.3.5.21 of \cite{htt}. Let $\Fun^{\et}\left(\Delta^1,\sD\right)$ denote the full subcategory of the functor $\icat$ $\Fun\left(\Delta^1,\sD\right)$ on the \'etale morphisms. We will denote morphisms in this category vertically, that is, if  $\varphi:\cE \to \cD$ and $\varphi:\cE' \to \cD'$ are \'etale morphisms, a homotopy commutative square
$$\xymatrix{\cE \ar[r]^-{\varphi} \ar[d] & \cD \ar[d]\\
\cD' \ar[r]^-{\varphi'} & \cD'}$$ will be regarded as a morphism $\varphi \to \varphi'$ in $\Fun^{\et}\left(\Delta^1,\sD\right)$.
Consider the functor given by evaluation at $\{1\} \in \Delta^1,$ $$q:\Fun^{\et}\left(\Delta^1,\sD\right) \to \sD.$$ Since \'etale maps are stable under pullback, this is a Cartesian fibration classifying the functor $$\sD^{op} \to \iCat$$ which on objects sends each object $\cD$ of $\sD$ to the $\icat$ $\sD^{\et}/\cD,$ and on arrows sends a map $f:\cD \to \cD'$ to the functor $\widetilde{f}$ of Definition \ref{dfn:lowerstar} induced by pulling back along $f.$ The corresponding functor $$q^{op}:\Fun^{\et}\left(\Delta^1,\sD\right)^{op} \to \sD^{op}$$ is a coCartesian fibration. Consider the projection functor $$r:\sD^{op} \times \iGpd \to \sD^{op},$$ which may be regarded as the Cartesian fibration classifying the constant functor into $\iCat$ on $\sD^{op}$ with value $\iGpd.$ Consider the simplicial set $\Kk$ together with a map $$p:\Kk \to \sD^{op},$$ satisfying the universal property that for every $j:X \to \sD^{op},$
\begin{equation}\label{eq:sections}
\Hom_{\Set^{\Delta^{op}}/\sD^{op}}\left(j,p\right) \cong \Hom_{\Set^{\Delta^{op}}}\left(X \times_{\sD^{op}} \Fun^{\et}\left(\Delta^1,\sD\right)^{op}, \iGpd\right).
\end{equation}
By \cite{htt}, Corollary 3.2.2.13, $p$ is a Cartesian fibration. It classifies the functor $$\left(\sD^{op}\right)^{op}=\sD \to \iCat$$ which on objects sends $\cD$ to $\Pshi\left(\sD^{\et}/\cD\right),$ and to arrows $f:\cD \to \cD',$ associates the functor $$\widetilde{f}^*:\Pshi\left(\sD^{\et}/\cD\right) \to \Pshi\left(\sD^{\et}/\cD'\right),$$ where $\widetilde{f}$ is as in Definition \ref{dfn:lowerstar}. To get a handle on what $\Kk$ looks like as an $\icat$, it is conceptually easier to describe $\Kk^{op}$. The objects of $\Kk^{op}$ are pairs $\left(\cD,F\right)$ with $\cD$ an object of $\sD$ and $$F:\left(\sD^{\et}/\cD\right)^{op} \to \iGpd,$$ and  morphisms $\left(\cD,F\right) \to \left(\cD',F'\right)$ correspond to pairs $\left(f,\alpha\right),$ with $f:\cD \to \cD'$ in $\sD$ and $$\alpha:F' \to \widetilde{f}^*\left(F\right)$$ in $\Pshi\left(\sD^{\et}/\cD'\right).$ 

Let $\M$ denote the $\icat$ of sections of $p.$ From equation (\ref{eq:sections}), it follows that it is canonically equivalent to the $\icat$ $$\Fun\left(\Fun^{\et}\left(\Delta^1,\sD\right)^{op},\iGpd\right)=\Pshi\left(\Fun^{\et}\left(\Delta^1,\sD\right)\right).$$ Notice that if $\cD$ is an object of $\sD,$ then $\sD^{\et}/\cD$ can be canonically identified with the fiber of $q$ over $$\Delta^0 \stackrel{\cD}{\longrightarrow} \sD,$$ and hence one has a canonical functor $$\xi_\cD:\sD^{\et}/\cD \to \Fun^{\et}\left(\Delta^1,\sD\right)$$ for each $\cD.$ Concretely, one has that an object $f:\cE \to \cD$ of $\sD^{\et}/\cD$ is sent to itself, whereas a morphism, regarded as a triangle
$$\xymatrix@R=0.7cm@C=0.2cm{\cE \ar[dr]_-{\varphi} \ar[rr]^-{\psi} & & \cE' \ar[dl]^-{\varphi'}\\
& \cD & }$$ is sent to the square
$$\xymatrix{\cE \ar[r]^-{\varphi}  \ar[d]_-{\psi} & \cD \ar[d]^-{id_\cD}\\
\cE' \ar[r]^-{\varphi'} & \cD,}$$
regarded as a morphism in $\Fun^{\et}\left(\Delta^1,\sD\right).$ Notice that there is a natural bijection between sections of $p$ and sections of $p^{op}.$ Unwinding definitions, one sees that given an object $H$ in $\M,$ that is a presheaf $H$ on $\Fun^{\et}\left(\Delta^1,\sD\right),$ the section $$\sigma_H:\sD \to \Kk^{op}$$ of $p$ corresponding to $H$ sends each object $\cD$ to the pair $\left(\cD,\xi_\cD^*\left(H\right)\right),$ and each morphism $f:\cD \to \cD'$ (of $\sD$) to the morphism $$\left(f,\alpha\left(f\right)\right):\left(\cD,\xi_\cD^*\left(H\right)\right) \to \left(\cD',\xi_\cD'^*\left(H\right)\right),$$ where $\alpha\left(f\right)$ can be described as follows:\\
Given an object $\varphi:\cE' \to \cD'$ of $\sD^{\et}/\cD',$ the homotopy commutative square
$$\xymatrix{\cE' \times_{\cD'} \cD \ar[d]_-{pr_1} \ar[r]^-{\widetilde{f}\left(\varphi'\right)} & \cD \ar[d]^-{f}\\
\cE' \ar[r]^-{\varphi'} & \cD',}$$ may be regarded as an arrow $$\left(pr_1,f\right):\widetilde{f}\left(\varphi'\right) \to \varphi'$$ in $\Fun^{\et}\left(\Delta^1,\sD\right).$ The morphism $$\alpha\left(f\right):\xi_\cD'^*\left(H\right) \to \widetilde{f}^*\xi_\cD^*\left(H\right)$$ is the natural transformation whose component along $\varphi'$ is given by $H\left(pr_1,f\right).$

Denote by $$\nabla:\sD \to \Fun^{\et}\left(\Delta^1,\sD\right)$$ the full and faithful inclusion which associates to each object $\cD$ its identity morphism. Notice that the functor $\nabla$ is in fact left adjoint to the functor $$ev_0:\Fun^{\et}\left(\Delta^1,\sD\right) \to \sD$$ given by evaluation at the vertex $\{0\} \in \Delta^1.$ We claim that $\nabla^*$ is left adjoint to $$ev_0^*:\Pshi\left(\sD\right) \to \Pshi\left(\Fun^{\et}\left(\Delta^1,\sD \right)\right).$$ Consider the corresponding functors
$$\widehat{\nabla}^*:\LPshi\left(\Fun^{\et}\left(\Delta^1,\sD\right)\right) \to \LPshi\left(\sD\right)$$
$$\widehat{ev_0}^*:\LPshi\left(\sD\right) \to \LPshi\left(\Fun^{\et}\left(\Delta^1,\sD \right)\right),$$
between the $\qcats$ of large presheaves. Since these restrict to $\nabla^*$ and $ev_0^*$ respectively on the subcategories of small presheaves, it suffices to show that $\widehat{\nabla}^*$ is left adjoint to $\widehat{ev_0}^*$. Since $\sD$ is $\V$-small, where $\V$ is the universe of large sets, there exists a global right Kan extension functor $$\widehat{\nabla}_*:\LPshi\left(\sD\right) \to \LPshi\left(\Fun^{\et}\left(\Delta^1,\sD \right)\right)$$ which is right adjoint to $\widehat{\nabla}^*$. Let $\varphi:\cE \to \cD$ be an \'etale morphism, and let $F$ be a large presheaf on $\sD.$ Then
\begin{eqnarray*}
\widehat{\nabla}_*\left(F\right)\left(\varphi\right) &\simeq& \Hom\left(y\left(\varphi\right),\widehat{\nabla}_*\left(F\right)\right)\\
&\simeq& \Hom\left(\widehat{\nabla}^*y\left(\varphi\right),F\right).
\end{eqnarray*}
Notice that if $\cE'$ is in $\sD,$ one has
\begin{eqnarray*}
\widehat{\nabla}^*y\left(\varphi\right)\left(\cE'\right) &\simeq& y\left(\varphi\right)\left(\nabla\left(\cE\right)\right)\\
&=& \Hom\left(\nabla\left(\cE'\right),\varphi\right)\\
&\simeq& \Hom\left(\cE',ev_0\left(\varphi\right)\right)\\
&=&\Hom\left(\cE',\cE\right).
\end{eqnarray*}
It follows that $\widehat{\nabla}^*y\left(\varphi\right)=y\left(\cE\right).$ Notice that
\begin{eqnarray*}
\widehat{ev_0}^*\left(F\right)\left(\varphi\right) &=& F\left(\cE\right)\\
&\simeq& \Hom\left(y\left(\cE\right),F\right).
\end{eqnarray*}
Therefore, there is a natural equivalence $\widehat{ev_0}^*\simeq \widehat{\nabla}_*$ and since $\widehat{ev_0}^*$ is right adjoint to $\widehat{\nabla}^*$, it follows that $ev_0^*$ is right adjoint to $\nabla^*.$ Since $\nabla$ is full and faithful, it follows that $\widehat{ev_0}^*$ and hence $ev_0^*$ are as well, and therefore
$$\xymatrix@C=1.5cm{\Pshi\left(\sD\right)\ar@{^{(}->}[r]<-0.9ex>_-{ev_0^*} & \Pshi\left(\Fun^{\et}\left(\Delta^1,\sD \right)\right)=\M \ar@<-0.5ex>[l]_-{\nabla^*}}$$
is a localization. Denote by $\M'$ the essential image of $ev_0^*.$

In particular, $ev_0^*$ preserves small limits, however, since colimits are computed point-wise, it follows that it also preserves small colimits. Notice further that $\nabla$ also has a \emph{left} adjoint given by $q=ev_1.$ By a similar argument for the adjunction $\nabla^* \vdash ev_0^*$, one can deduce that $ev_1^*$ is left adjoint to $\nabla^*$ so that likewise $\nabla^*$ preserves all small limits and colimits.

Denote by $\Kk_0$ the full subcategory of $\Kk$ on those $\left(\cD,F\right)$ such that $F$ is an $\i$-sheaf. Suppose that $f:\cD \to \cD'$ in $\sD$ has an underlying geometric morphism of $\i$-topoi
$$\xymatrix@C=1.5cm@R=2.5cm{\cD \ar@<-0.5ex>[r]_-{f_*} & \cD' \ar@<-0.5ex>[l]_-{f^*}}.$$
It follows immediately from Proposition \ref{prop:lowerstar} that $p$ restricts to a Cartesian fibration $$p_0:\Kk_0 \to \sD^{op}.$$ Denote by $\M_0$ the full subcategory of $\M$ on those sections factoring through the inclusion $$\Kk_0 \hookrightarrow \Kk.$$ This is the full subcategory on those presheaves $H$ such that for each $\cD,$ $\xi_\cD^*\left(H\right)$ is an $\i$-sheaf. Fix an object $\cD$ of $\sD$ and consider the following diagram of functors:
$$\xymatrix{\sD^{\et}/\cD \ar[d]_-{\pi_\cD} \ar[r]^-{\xi_\cD} & \Fun^{\et}\left(\Delta^1,\sD\right)\\
\sD^{\et} \ar[r]^-{j} & \sD \ar@{^{(}->}[u]_-{\nabla}.}$$
This diagram does \emph{not} commute but, denoting the composite $j \circ \pi_\cD=:\chi_\cD,$ there is a canonical non-invertible natural transformation $$\omega_\cD: \nabla \circ \chi_\cD \Rightarrow \xi_\cD,$$ whose component along an \'etale map $\varphi:\cE \to \cD$ is given by the morphism
$$\xymatrix{\cE \ar[d]_-{id_\cE} \ar[r]^-{id_\cE} & \cE \ar[d]^-{\varphi}\\
\cE \ar[r]^-{\varphi} & \cD}$$
in $\Fun^{\et}\left(\Delta^1,\sD\right).$
Now suppose that $H$ is in $\M'.$ Then, we have an equivalence $H \simeq ev_0^*\nabla^*H.$ The component of the natural transformation $$\left(ev_0^*\nabla^*H\right) \circ \nabla \circ \chi_\cD \Rightarrow \left(ev_0^*\nabla^*H\right) \circ \xi_\cD$$ is then $$ev_0^*\nabla^*H\left(\omega_\cD\left(\varphi\right)\right) = H\left(\nabla ev_0 \omega_\cD\left(\varphi\right)\right)=\left(H \circ \nabla\right)\left(id_\cE\right),$$ hence an equivalence. It follows that if $H$ is in $\M',$ then the following diagram commutes up to natural homotopy:
$$\xymatrix{\left(\sD^{\et}/\cD\right)^{op} \ar[r]^-{\xi^{op}_\cD} \ar[rd]_-{\chi^{op}_\cD} & \Fun^{\et}\left(\Delta^1,\sD\right)^{op} \ar[r]^-{H} & \iGpd.\\
& \sD^{op}  \ar[ur]_-{\nabla^*H} &}$$
Hence, for $H$ in $\M' \simeq \Pshi\left(\sD\right),$ $\xi_\cD^*H$ is an $\i$-sheaf if and only if $\chi_\cD^*\left(\nabla^*H\right)$ is an $\i$-sheaf. Let $\M'_0$ denote the $\icat$ $\M' \cap \M_0.$ It follows that $\nabla^*$ restricts to an equivalence $$\nabla^*|_{\M'_0}:\M'_0 \stackrel{\sim}{\longlongrightarrow}\Shi\left(\sD\right).$$
By \cite{dag}, Lemma 2.4.9, since both $p$ and $p_0$ are Cartesian fibrations, one has the following assertion:
\begin{itemize}
\item The inclusion $\M_0 \hookrightarrow \M$ admits a left adjoint $L.$ Moreover, a morphism $f:X \to X'$ in $\M$ exhibits $X'$ as a $\M_0$-localization of $X$ if and only if for each $\cD,$ $$\xi_\cD^*f:\xi_\cD^*X \to \xi_\cD^*X'$$ exhibits $\xi_\cD^*X'$ as the sheafification of $\xi_\cD^*X$ in $\Pshi\left(\sD^{\et}/\cD\right).$
\end{itemize}
To show that the inclusion $$i:\Shi\left(\sD\right) \hookrightarrow \Pshi\left(\sD\right)$$ admits a left adjoint, it suffices to show that the functor $L$ carries $\M'$ to $\M'.$ To prove this, we first make the following claim:\\
\textbf{Claim:} If $H$ is in $\M',$ then for each \'etale map $\varphi:\cE \to \cD$ in $\sD,$
\begin{equation}\label{eq:claimchi}
\tp^*\xi_\cD^*H \simeq \xi_\cE^*H,
\end{equation}
where $\tp$ is the functor in Definition \ref{dfn:upperstar}.
To see this, let $\lambda:\cE' \to \cE$ be \'etale, then, on one hand, one has the following string of equivalences:
\begin{eqnarray*}
\tp^*\xi_\cD^*H\left(\lambda\right) &\simeq& H\left(\xi_\cD\tp\left(\lambda\right)\right)\\
&\simeq& H\left(\varphi \circ \lambda\right)\\
&\simeq& ev_0^*\nabla^*H\left(\varphi \circ \lambda\right)\\
&\simeq& H\left(\nabla ev_0\left(\varphi \circ \lambda\right)\right)\\
&\simeq& H\left(id_\cE'\right).
\end{eqnarray*}
On the other hand:
\begin{eqnarray*}
\xi_\cE^*H\left(\lambda\right) &=& H\left(\lambda\right)\\
&\simeq& ev_0^*\nabla^*H\left(\lambda\right)\\
&\simeq& H\left(id_\cE'\right).
\end{eqnarray*}
We will now show that if $H$ is in $\M',$ then $LH$ is too. Let $\varphi:\cE \to \cD$ be \'etale in $\sD$. Then, by \cite{dag}, Lemma 2.4.9, one has $$LH\left(\varphi\right) \simeq \left(i_Da_\cD\xi_\cD^*H\right)\left(\varphi\right).$$ Notice that
$$\left(i_\cD a_\cD \xi_\cD^*H\right)\left(\varphi\right) \simeq \left(\tp^*i_\cD a_\cD \xi_\cD^* H\right)\left(id_\cE\right).$$
Moreover, by Proposition \ref{prop:upperstar}, one has $$\tp^*i_Da_\cD\xi_\cD^*H \simeq i_\cE\varphi^*a_\cD\xi_\cD^*H,$$ and by Proposition \ref{prop:upperstar2}, one has $$i_\cE\varphi^*a_\cD\xi_\cD^*H \simeq i_\cE a_\cE \tp^* \xi_\cD^*H.$$ Furthermore by (\ref{eq:claimchi}), it follows that $$i_\cE a_\cE \tp^* \xi_\cD^*H \simeq i_\cE a_\cE\xi_\cE^*H.$$ Putting all this together, one has:
\begin{eqnarray*}
LH\left(\varphi\right) &\simeq& \left(i_Da_\cD\xi_\cD^*H\right)\left(\varphi\right)\\
&\simeq& \left(i_\cE a_\cE\xi_\cE^*H\right)\left(id_\cE\right)\\
&\simeq& LH\left(id_\cE\right)\\
&\simeq& \left(ev_0^*\nabla^*LH\right)\left(\varphi\right),
\end{eqnarray*}
so that $LH$ is in $\M'.$ It follows that the composite $$a:=\nabla^* \circ L \circ ev_0^*$$ is left adjoint to $i$. Since both $\nabla^*$ and $ev_0^*$ preserve all small limits, to deduce that $a$ is left exact, it suffices to show that $L$ is. To this end, let $\lim H_i$ be a finite limit, and let $\varphi:\cE \to \cD$ be an \'etale map. Notice that since limits in presheaf $\infty$-category are computed point wise, for all $\cD,$ $\xi_\cD^*$ preserves all limits. Hence
\begin{eqnarray*}
L\left(\lim H_i\right)\left(\varphi\right)&\simeq& \left(a_\cD\xi_\cD^*\lim H_i\right)\left(\varphi\right)\\
&\simeq& \lim a_\cD\xi_\cD^*H_i\left(\varphi\right)\\
&\simeq& \lim LH_i\left(\varphi\right),
\end{eqnarray*}
since $a_\cD$ is left exact.

Suppose now that $\sD$ is small. Since $\Shi\left(\sD\right)$ is a left exact localization of $\Pshi\left(\sD\right),$ to show that it is an $\i$-topos, we only need to show that this localization is accessible, that is, we need to show that there exists a regular cardinal $\kappa$ for which $\Pshi\left(\sD\right)$ is $\kappa$-accessible, and such that the composite $$\Pshi\left(\sD\right) \stackrel{a}{\longrightarrow} \Shi\left(\sD\right) \stackrel{i}{\longrightarrow} \Pshi\left(\sD\right)$$ preserves $\kappa$-filtered colimits. Since $\sD$ is small,  $\Pshi\left(\sD\right)$ is $\kappa_0$-accessible for some regular cardinal $\kappa_0.$ Moreover, for each $\cD$ in $\sD,$ there exists regular cardinals $\kappa_\cD$ such that each $\Pshi\left(\sD^{\et}/\cD\right)$ is $\kappa_\cD$-accessible and the composite $$\Pshi\left(\sD^{\et}/\cD\right) \stackrel{a_\cD}{\longlongrightarrow} \Shi\left(\sD^{\et}/\cD\right) \stackrel{i_\cD}{\longlongrightarrow} \Pshi\left(\sD^{\et}/\cD\right)$$ preserves $\kappa_\cD$-filtered colimits. Choose a large enough regular cardinal $\kappa$ greater than $\kappa_0$ and each $\kappa_\cD,$ and let $$\colim F_\alpha$$ be a $\kappa$-filtered colimit in $\Pshi\left(\sD\right).$ It suffices to show that $ia$ preserves this colimit. Since colimits in $\Pshi\left(\sD\right)$ are computed point wise, it suffices to show that for each $\cD$, one has $$\left(ia\colim F_\alpha\right)\left(\cD\right) \simeq \colim F_\alpha\left(\cD\right).$$ It follows from the construction of $a$ that
$$\left(ia\colim F_\alpha\right)\left(\cD\right) \simeq \left(a_\cD\xi_\cD^*ev_0^*\colim F_\alpha\right)\left(id_\cD\right).$$ Notice that $$\left(a_\cD\xi_\cD^*ev_0^*\colim F_\alpha\right)\left(id_\cD\right) \simeq \left(a_\cD\colim \xi_\cD^*ev_0^*F_\alpha\right)\left(id_\cD\right),$$ and since the colimit is in particular $\kappa_\cD$-filtered, one has further that $$\left(a_\cD\colim \xi_\cD^*ev_0^*F_\alpha\right)\left(id_\cD\right) \simeq \left(\colim \xi_\cD^*ev_0^*F_\alpha\right)\left(id_\cD\right).$$ Notice finally that
\begin{eqnarray*}
\left(\colim \xi_\cD^*ev_0^*F_\alpha\right)\left(id_\cD\right) &\simeq& \colim F_\alpha\left(ev_0\xi_\cD id_\cD\right)\\
&\simeq& \colim F_\alpha\left(\cD\right).
\end{eqnarray*}

\end{proof}

\begin{cor}\label{cor:colimps}
If $F:J \to \Shi\left(\sD\right)$ is a functor with $J$ small, then it has a colimit, and a cocone $$\rho:J^{\triangleright} \to \Shi\left(\sD\right)$$ for $F$ is colimiting if and only if for all $\cD$ in $\sD$, the composite
$$J^{\triangleright} \stackrel{\rho}{\longrightarrow} \Shi\left(\sD\right) \stackrel{\chi_\cD^*}{\longlongrightarrow} \Shi\left(\sD^{\et}/\cD\right)\simeq \cD$$ is, where $\chi_\cD:\sD^{\et}/\cD \to \sD$ is the canonical projection functor $j \circ \pi_\cD$.
\end{cor}

\begin{proof}
We will freely use notation introduced in the proof of Proposition \ref{prop:sheafification}. Since $\LPshi\left(\sD\right)$ has all (large) colimits, and $\Pshi\left(\sD\right)$ is stable under small colimits in $\LPshi\left(\sD\right),$ it follows that $\Pshi\left(\sD\right)$ has all small colimits. Hence, $\Shi\left(\sD\right)$ does as well, since $a$ preserves small colimits. Since $p_0$ is a Cartesian fibration, by Proposition 5.1.2.2 of \cite{htt}, it follows that if $\mu:J^{\triangleright} \to \M_0,$ is a cocone, then it is colimiting if and only if for each $\cD$ in $\sD,$ the composite $$J^{\triangleright} \stackrel{\mu}{\longrightarrow} \M_0 \stackrel{\xi_\cD^*}{\longlongrightarrow} \Shi\left(\sD^{\et}/\cD\right)$$ is. Notice that if $F$ is an $\i$-sheaf in $\Shi\left(\sD\right),$ then $ev_0^*\left(iF\right)$ is in $\M_0.$ To see this, note that for all $\cD$ in $\sD,$ by the proof of Proposition \ref{prop:sheafification} it follows that $$\xi_\cD^*ev_0^*\left(iF\right)\simeq \chi_\cD^*\nabla^*ev_0^*\left(iF\right) \simeq \chi_\cD^*  F.$$ It follows that $ev_0^*$ restricts to a functor $$\widetilde{ev_0^*}:\Shi\left(\sD\right) \to \M_0.$$ Consider the following string of equivalences:
\begin{eqnarray*}
\colim F_\alpha &\simeq& a\left(\colim i F_\alpha\right)\\
&\simeq& \nabla^* \circ L \circ ev_0^*\left(\colim i F_\alpha\right)\\
&\simeq& \nabla^*\left(\colim L ev_0^* i F_\alpha\right)\\
&\simeq& \nabla^*\left( \colim \widetilde{ev_0^*}F_\alpha\right),
\end{eqnarray*}
from whence it follows that $$\widetilde{ev_0^*}\left(\colim F_\alpha\right) \simeq ev_0^*\nabla^*\left( \colim \widetilde{ev_0^*}F_\alpha\right) \simeq \colim \widetilde{ev_0^*}F_\alpha.$$ Putting this all together, it follows that $\rho$ is colimiting if and only if the composite
$$J^{\triangleright} \stackrel{\rho}{\longrightarrow} \Shi\left(\sD\right) \stackrel{\widetilde{ev_0^*}}{\longlonglongrightarrow} \M_0$$ is, which in turn is if and only if each composite $$J^{\triangleright} \stackrel{\rho}{\longrightarrow} \Shi\left(\sD\right) \stackrel{\widetilde{ev_0^*}}{\longlonglongrightarrow} \M_0 \stackrel{\xi_\cD^*}{\longlonglongrightarrow} \Shi\left(\sD^{\et}/\cD\right)$$ is. However, from the proof of Proposition \ref{prop:sheafification}, it follows that this final composite is equivalent to the composite $$J^{\triangleright} \stackrel{\rho}{\longrightarrow} \Shi\left(\sD\right) \stackrel{\chi_\cD^*}{\longlongrightarrow} \Shi\left(\sD^{\et}/\cD\right).$$
\end{proof}

\begin{cor}
The inclusion $$\widehat{i}:\LPshi\left(\sD\right) \hookrightarrow \LShi\left(\sD\right)$$ has a left exact left adjoint $\widehat{a},$ and $\LShi\left(\sD\right)$ is an $\i$-topos in the Grothendieck universe $\mathcal{V}$ of large sets. If $F:J \to \LShi\left(\sD\right)$ is a functor with $J$ a $\mathcal{V}$-small $\icat$, then it has a colimit, and a cocone $$\rho:J^{\triangleright} \to \LShi\left(\sD\right)$$ for $F$ is colimiting if and only if for all $\cD$ in $\sD$, the composite
$$J^{\triangleright} \stackrel{\rho}{\longrightarrow} \LShi\left(\sD\right) \stackrel{\chi_\cD^*}{\longlongrightarrow} \LShi\left(\sD^{\et}/\cD\right)$$ is.
\end{cor}

\begin{proof}
The proof is completely analogous to that of Proposition \ref{prop:sheafification} and Corollary \ref{cor:colimps}, with the role of $\iGpd$ replaced with that of $\LiGpd.$
\end{proof}

\begin{cor}\label{cor:colimpset}
The inclusion $$i^{\et}:\Pshi\left(\sD^{\et}\right) \hookrightarrow \Shi\left(\sD^{\et}\right)$$ has a left exact left adjoint $a^{\et},$ and if $\sD$ is small, $\Shi\left(\sD^{\et}\right)$ is an $\i$-topos. If $F:J \to \Shi\left(\sD^{\et}\right)$ is a functor with $J$ small, then it has a colimit, and a cocone $$\rho:J^{\triangleright} \to \Shi\left(\sD^{\et}\right)$$ for $F$ is colimiting if and only if for all $\cD$ in $\sD$, the composite
$$J^{\triangleright} \stackrel{\rho}{\longrightarrow} \Shi\left(\sD^{\et}\right) \stackrel{\pi_\cD^*}{\longlongrightarrow} \Shi\left(\sD^{\et}/\cD\right)$$ is. Finally, the obvious adaptation of these two statement for large $\i$-sheaves also holds.
\end{cor}

\begin{proof}
The proof of the first statement is analogous to that of Proposition \ref{prop:sheafification}, with the role of the Cartesian fibration $$q:\Fun^{\et}\left(\Delta^1,\sD\right) \to \sD$$ replaced by the Cartesian fibration $$ev_1:\Fun\left(\Delta^1,\sD^{\et}\right) \to \sD^{\et}$$ given by evaluation on the vertex $\{1\} \in \Delta^1,$ and the role of $\chi_\cD$ replaced by that of $\pi_\cD.$ The rest of the proof is entirely the same. The second statement follows almost verbatim from the proof of Corollary \ref{cor:colimps}. The third statement is trivial.
\end{proof}

\begin{prop}\label{prop:bothfuncolims}
With the notation of Proposition \ref{prop:subcanonical}, both functors
$$\dbare \stackrel{j}{\longrightarrow} \dbar \stackrel{i^*_\sD \circ y}{\longlonglongrightarrow} \Shi\left(\sD\right)$$
and
$$\dbare \stackrel{\left(i^{\et}_\sD\right)^* \circ y^{\et}}{\longlonglongrightarrow} \Shi\left(\sD^{\et}\right)$$
preserve small colimits.
\end{prop}
\begin{proof}
Let $$J^\triangleright \stackrel{\rho}{\longrightarrow} \dbare$$ be a small colimiting cocone, with vertex $\cE$. By Corollary \ref{cor:colimps} and Corollary \ref{cor:colimpset}, we need to show that for all $\cD$ in $\sD,$ the composites
$$J^\triangleright \stackrel{\rho}{\to} \dbare \stackrel{j}{\longrightarrow} \dbar \stackrel{i^*_\sD \circ y}{\longlonglongrightarrow} \Shi\left(\sD\right) \stackrel{\left(j \circ \pi_\cD\right)^*}{\longlonglongrightarrow} \Shi\left(\sD^{\et}/\cD\right)$$
and
$$J^\triangleright \stackrel{\rho}{\to} \dbare \stackrel{\left(i^{\et}_\sD\right)^* \circ y^{\et}}{\longlonglongrightarrow} \Shi\left(\sD^{\et}\right) \stackrel{\pi_\cD^*}{\longlongrightarrow} \Shi\left(\sD^{\et}/\cD\right)$$
are colimiting cocones. Notice that, by the proof of Proposition \ref{prop:colim1}, there exists a colimiting cocone
$$J^{\triangleright} \stackrel{\widehat{\rho}}{\longrightarrow} \dbare/\cE$$ with vertex $id_\cE$ such that the composition $$J^\triangleright \stackrel{\widehat{\rho}}{\longrightarrow} \dbare/\cE \to \dbare$$ is equivalent to $\rho.$ It therefore suffices to show that
$$J^{\triangleright} \stackrel{\widehat{\rho}}{\longlongrightarrow} \dbare/\cE \to \dbare \stackrel{j}{\longrightarrow} \dbar \stackrel{i^*_\sD \circ y}{\longlongrightarrow} \Shi\left(\sD\right) \stackrel{\left(j \circ \pi_\cD\right)^*}{\longlongrightarrow} \Shi\left(\sD^{\et}/\cD\right)$$
and
$$J^{\triangleright} \stackrel{\widehat{\rho}}{\longlongrightarrow} \dbare/\cE \to \dbare  \stackrel{\left(i^{\et}_\sD\right)^* \circ y^{\et}}{\longlonglongrightarrow} \Shi\left(\sD^{\et}\right) \stackrel{\pi_\cD^*}{\longlongrightarrow} \Shi\left(\sD^{\et}/\cD\right)$$
are colimiting cocones. However, observe that the composites
$$\resizebox{5in}{!}{$\cE \stackrel{\sim}{\longrightarrow} \dbare/\cE \to \dbare \stackrel{j}{\longrightarrow} \dbar \stackrel{i^*_\sD \circ y}{\longlongrightarrow} \Shi\left(\sD\right) \stackrel{\left(j \circ \pi_\cD\right)^*}{\longlongrightarrow} \Shi\left(\sD^{\et}/\cD\right) \stackrel{\sim}{\longrightarrow} \Shi\left(\cD\right)$}$$
and
$$\resizebox{5in}{!}{$\cE \stackrel{\sim}{\longrightarrow} \dbare/\cE \to \dbare  \stackrel{\left(i^{\et}_\sD\right)^* \circ y^{\et}}{\longlonglongrightarrow} \Shi\left(\sD^{\et}\right) \stackrel{\pi_\cD^*}{\longrightarrow} \Shi\left(\sD^{\et}/\cD\right) \stackrel{\sim}{\longrightarrow} \Shi\left(\cD\right)$}$$
are equivalent to $\tau^\cE_\cD$ and $\tau^{\cE,{\et}}_\cD$ respectively, which preserve small colimits by Lemma \ref{lem:taucolim}.
\end{proof}

\begin{thm}\label{thm:fullyfaithful}
With the notation of Proposition \ref{prop:subcanonical}, both functors
\begin{equation}\label{eq:ff1}
\dbar \stackrel{i^*_\sD \circ y}{\longlonglongrightarrow} \Shi\left(\sD\right)
\end{equation}
and
\begin{equation}\label{eq:ff2}
\dbare \stackrel{\left(i^{\et}_\sD\right)^* \circ y^{\et}}{\longlonglongrightarrow} \Shi\left(\sD^{\et}\right)
\end{equation}
are full and faithful.
\end{thm}

\begin{proof}
We will prove that (\ref{eq:ff1}) is full and faithful. The proof for (\ref{eq:ff2}) is completely analogous. Let $\cE$ be an object of $\dbar.$ We wish show that for all $\cF$ in $\dbar,$ the canonical map
$$\Hom\left(\cF,\cE\right) \to \Hom\left(i^*_\sD y\left(\cF\right),i^*_\sD y\left(\cE\right)\right)$$
is an equivalence.
Denote by $\overline{\sD^{\et}}'$ the full subcategory of $\dbare$ on those objects $\cF$ for which this holds. Notice that if $\cD$ is in $\sD,$ then $$i^*_\sD y\left(\cD\right) \simeq y\left(\cD\right),$$ where on the left hand side $y$ is the Yoneda embedding for $\dbar,$ and on the right hand side it is the Yoneda embedding for $\sD$. By the Yoneda lemma if follows that
\begin{eqnarray*}
\Hom\left(i^*_\sD y\left(\cD\right),i^*_\sD y\left(\cE\right)\right) &\simeq& \Hom\left(y\left(\cD\right),i^*_\sD y\left(\cE\right)\right)\\
&\simeq& i^*_\sD y\left(\cE\right)\left(\cD\right)\\
&\simeq& \Hom\left(\cD,\cE\right),
\end{eqnarray*}
so $\sD^{\et}$ is contained in $\overline{\sD^{\et}}'.$ Suppose that $\cF \simeq  \colim \cF_\alpha$ is a colimit in $\dbare,$ with each $\cF_\alpha$ in $\overline{\sD^{\et}}'.$ Notice that by Lemma \ref{lem:holmol} and Proposition \ref{prop:etcolims}, this is also a colimit in $\dbar.$  Furthermore, by Proposition \ref{prop:bothfuncolims}, $$i^*_\sD y\left(\colim \cF_\alpha\right) \simeq \colim i^*_\sD y\left(\cF_\alpha\right),$$ so one has that
\begin{eqnarray*}
\Hom\left(i^*_\sD y\left(\colim \cF_\alpha\right),i^*_\sD y\left(\cE\right)\right) &\simeq& \lim \Hom\left(i^*_\sD  y\left(\cF_\alpha\right),i^*_\sD y\left(\cE\right)\right)\\
&\simeq& \lim \Hom\left(\cF_\alpha,\cE\right)\\
&\simeq& \Hom\left(\cF,\cE\right).
\end{eqnarray*}
Hence $\overline{\sD^{\et}}'$ is stable under colimits and contains the strong \'etale blossom $\sD,$ so we are done by Corollary \ref{cor:stetgen}.
\end{proof}

\section{A classification of the functor of points.}

By Theorem \ref{thm:fullyfaithful}, for a small and locally small strong \'etale blossom $\sD,$ one can equivalently view the $\icat$ $\dbar$ of $\sD$-\'etendues via their functor of points:

\begin{dfn}
An $\i$-sheaf $\X$ on $\sD$ is called a \textbf{Deligne-Mumford stack} if it is in the essential image of $i^*_\sD \circ \overline{y},$ where $$\overline{y}:\dbar \hookrightarrow \Shi\left(\dbar\right),$$ is the Yoneda embedding and $$i^*_\sD:\Shi\left(\dbar\right) \to \Shi\left(\sD\right)$$ is the canonical restriction functor. Denote the corresponding $\icat$ by $\dm\left(\sD\right).$ A Deligne-Mumford stack is called a Deligne-Mumford $n$-stack, if it is $n$-truncated.
\end{dfn}

\begin{rmk}
It follows immediately that $\dm\left(\sD\right) \simeq \dbar.$
\end{rmk}

We now wish to give a categorical representability criteria for Deligne-Mumford stacks, that is, we wish to characterize categorically precisely which $\i$-sheaves on $\sD$ are Deligne-Mumford. We will show that the $\i$-category of Deligne-Mumford stacks on $\sD$ is determined entirely by the canonical functor $$j:\sD^{\et} \to \sD$$ and the Yoneda embedding $$y:\sD \hookrightarrow \Shi\left(\sD\right).$$
 First, we will need an important concept:

\begin{dfn}\label{dfn:stronglygen}
A subcategory $\C \stackrel{i}{\longhookrightarrow} \mathscr{E}$ of a  $\icat$ $\mathscr{E}$ is said to be \textbf{strongly generating} if the left Kan extension $$\Lan_i \left(i\right) \simeq id_\mathscr{E}.$$ A cocomplete $\icat$ $\mathscr{E}$ is said to be of \textbf{small generation} if it admits a small strongly generating subcategory.
\end{dfn}

\begin{rmk}
If $\mathscr{E}$ in Definition \ref{dfn:stronglygen} is cocomplete, such a Kan extension is computed point wise by assigning an object $E$ of $\mathscr{E}$ the colimit of the canonical projection $$\C/E \to \mathscr{E},$$ where  $\C/E$ is the full subcategory of the slice category $\mathscr{E}/E$ on arrows with domain in $\C.$ This can be expressed more informally by the formula
$$\Lan_i \left(i\right)\left(E\right) \simeq \underset{C \to E}\colim \mspace{2mu} C.$$
\end{rmk}

\begin{prop}(\cite{htt} Remark 5.2.9.4)\label{prop:5.2.9.4}
A small subcategory $\C \stackrel{i}{\longhookrightarrow} \mathscr{E}$ of a cocomplete locally small $\icat$ $\mathscr{E}$ is strongly generating if and only if the composite
$$\mathscr{E} \stackrel{y}{\longhookrightarrow} \Pshi\left(\mathscr{E}\right) \stackrel{i^*}{\longrightarrow} \Pshi\left(\C\right)$$ is full and faithful.
\end{prop}

\begin{proof}
Let $$L=\Lan_y\left(i\right):\Pshi\left(\C\right) \to \mathscr{E}.$$ This is a colimit preserving functor whose right adjoint can be identified with $i^* \circ y.$ Hence, $i^* \circ y$ is full and faithful if and only if the co-unit of the adjunction $$L \dashv i^* \circ y$$ is an equivalence. Notice that for $E$ an object of $\mathscr{E},$ one has that
\begin{eqnarray*}
\left(\Lan_y i \right) \left( i^*y\left(E\right)\right) &\simeq& \underset{\resizebox{1cm}{!}{$y(C) \to i^*y(E)$}}\colim \mspace{2mu} C\\
&\simeq& \underset{C \to E}\colim \mspace{2mu} C\\
&\simeq& \Lan_i \left(i\right)\left(E\right),
\end{eqnarray*}
and the component of the co-unit along $E$ can be identified with the canonical map $$\Lan_i \left(i\right)\left(E\right) \to E.$$ It follows that $i^* \circ y$ is full and faithful if and only if $$\Lan_i \left(i\right) \simeq id_\mathscr{E}.$$
\end{proof}

\begin{cor}\label{cor:5.2.9.4}
If a small subcategory $\C \stackrel{i}{\longhookrightarrow} \mathscr{E}$ of a cocomplete locally small $\icat$ $\mathscr{E}$ is strongly generating, then $\mathscr{E}$ is a reflective subcategory of $\Pshi\left(\C\right).$
\end{cor}

The following corollary follows from Theorem \ref{thm:fullyfaithful}:

\begin{cor}\label{cor:dense}
If $\sD$ is a small strong \'etale blossom, then $\sD^{\et}$ strongly generates $\dbare.$
\end{cor}

\begin{lem}\label{lem:sheavesthesame}
For any locally small strong \'etale blossom $\sD,$ there are canonical equivalences
\begin{equation}\label{eq:smallnotet}
\Shi\left(\sD\right) \simeq \Shi\left(\dbar\right),
\end{equation}
\begin{equation}\label{eq:largenotet}
\LShi\left(\sD\right) \simeq \LShi\left(\dbar\right),
\end{equation}
\begin{equation}\label{eq:smallet}
\Shi\left(\sD^{\et}\right) \simeq \Shi\left(\dbare\right),
\end{equation}
and
\begin{equation}\label{eq:larget}
\LShi\left(\sD^{\et}\right) \simeq \LShi\left(\dbare\right).
\end{equation}
\end{lem}

\begin{proof}
We will prove the equivalences (\ref{eq:smallet}) and (\ref{eq:larget}). The proof of the other equivalences are analogous. To ease notation, denote by $$l:\sD^{\et} \hookrightarrow \dbare$$ the canonical functor. If $\V$ is the universe of large sets, then both $\sD^{\et}$ and $\dbare$ are $\V$-small. Denote by $\widehat{y^{\et}}$ and $\widehat{\overline{y}^{\et}}$ the Yoneda embeddings into presheaves of large $\i$-groupoids of $\sD^{\et}$ and $\dbare$ respectively. Let $\tilde l_!$ denote the left Kan extension $$\tilde l_!=\Lan_{\widehat{y^{\et}}}\left(\widehat{\overline{y}^{\et}}\circ l\right):\LPshi\left(\sD^{\et}\right) \to \LPshi\left(\dbare\right).$$
It is left adjoint to the obvious restriction functor $\tilde l^*.$ We claim that $\tilde l^*$ sends sheaves to sheaves. To see this, let $\cD$ be an object of $\sD.$ Then we can choose a small full subcategory $$r_\cD:\C_\cD \hookrightarrow \sD^{\et}/\cD$$ such that, with the same notation as in Definition \ref{dfn:strongblossom}, $$\xymatrix@C=1.5cm{\cD \ar@{^{(}->}[r]<-0.9ex>_-{R_\cD} & \Pshi\left(\C_\cD\right) \ar@<-0.5ex>[l]_-{L_{\cD}}}$$ is an accessible left exact localization. By the proof of Proposition \ref{prop:subcanonical} (where $\dbar$ is regarded as the strong \'etale blossom in the proposition), one concludes that (using the notation of Definition \ref{dfn:localargesheaf} and the proof of Proposition \ref{prop:pop}) the localization
$$\xymatrix{\LShi\left(\sD^{\et}/\cD\right) \ar@{^{(}->}[r]<-0.9ex>_-{\widehat{i}_\cD} & \LPshi\left(\sD^{\et}/\cD\right) \ar@<-0.5ex>[l]_-{\widehat{a}_\cD}}$$ factors as
\begin{equation}\label{eq:factorization}
\resizebox{4.8in}{!}{$\xymatrix{\LShi\left(\sD^{\et}/\cD\right) \simeq \widehat{S}_\C^{-1} \LPshi\left(\C_\cD\right) \ar@{^{(}->}[r]<-0.9ex> & \LPshi\left(\C_\cD\right) \ar@<-0.5ex>[l] \ar@{^{(}->}[r]<-0.9ex>_-{\left(\widehat{r}_\cD\right)_*} & \LPshi\left(\sD^{\et}/\cD\right) \ar@<-0.5ex>[l]_-{\left(\widehat{r}_\cD\right)^*} \ar@{^{(}->}[r]<-0.9ex>_-{\left(l/\cD\right)_*} & \LPshi\left(\dbare/\cD\right). \ar@<-0.5ex>[l]_-{\left(l/\cD\right)^*}}$}
\end{equation}
Moreover, notice that for all $\cD$ in $\sD,$ the following diagram commutes up to canonical homotopy
\begin{equation}\label{eq:sqqq}
\xymatrix{\sD^{\et}/\cD \ar[d] \ar@{^{(}->}[r]^-{l/\cD} & \dbare/\cD \ar[d]\\
\sD^{\et} \ar@{^{(}->}[r]^-{l} & \dbare.}
\end{equation}
It follows that $\tilde l^*$ restricts to a functor $$l^*:\LShi\left(\dbare\right) \to \LShi\left(\sD^{\et}\right),$$ which is right adjoint to $$l_!:= \widehat{\overline{a}}^{\et} \circ \tilde l_!,$$ where $\widehat{\overline{a}}^{\et}$ denotes sheafification.
We claim that $$l_! \dashv l^*$$ is an adjoint equivalence. To prove this, one needs to show that the unit and the co-unit are equivalences. Consider the unit $$\eta:id_{\LShi\left(\sD^{\et}\right)} \to l^*l_!.$$ The domain and codomain of $\eta$ are colimit preserving functors, and by Proposition 5.5.4.20 and Theorem 5.1.5.6 of \cite{htt}, composition with $$\widehat{a}_\sD \circ \widehat{y}^{\et}$$ induces a full and faithful functor $$\Fun^{L}\left(\LShi\left(\sD^{\et}\right),\LShi\left(\sD^{\et}\right)\right) \hookrightarrow \Fun\left(\sD^{\et},\LShi\left(\sD^{\et}\right)\right),$$ where $\Fun^{L}\left(\LShi\left(\sD^{\et}\right),\LShi\left(\sD^{\et}\right)\right)$ is the full subcategory of the functor category on all functors which preserve small colimits. It therefore suffices to check that the components of $\eta$ along representable sheaves are equivalences, which is clear by the definitions. We will now show that the co-unit $$\varepsilon:l_!l^* \to id_{\LShi\left(\dbare\right)}$$ is an equivalence. By an analogous argument, it suffices to check that the components of it along representable sheaves are equivalences. Let $\cE$ be an object of $\dbar.$ By Proposition \ref{prop:locblossom}, there exists a small strong \'etale blossom $\sD_\cE$ contained in $\sD$ such that $\cE$ is $\overline{\sD_\cE}.$ By Corollary \ref{cor:dense} combined with Lemma \ref{lem:holmol}, one concludes that $\cE$ is the colimit of the canonical functor $$\left(\sD_\cE\right)^{\et}/\cE \to \left(\sD_\cE\right)^{\et} \hookrightarrow \dbare.$$ In particular, we can write, in a canonical way, $\cE$ as a colimit $$\cE \simeq \colim \cD_\alpha$$ with each $\cD_\alpha$ in $\sD.$ Notice that Proposition \ref{prop:bothfuncolims} applied to the strong \'etale blossom $\dbar,$ implies that $$\overline{a}^{\et} \circ\overline{y}^{\et}$$ preserves small colimits. By the proof of Proposition 6.3.5.17 of \cite{htt}, it follows that the inclusion $$\Shi\left(\dbare\right) \hookrightarrow \LShi\left(\dbare\right)$$ preserves small colimits, and we conclude that  $$\widehat{\overline{a}}^{\et} \circ \widehat{\overline{y}}^{\et}$$ preserve small colimits as well. Moreover, notice that the factorization (\ref{eq:factorization}) implies that the induced map $$\left(l/\cD\right)^*:\LShi\left(\dbare/\cD\right) \to \LShi\left(\sD^{\et}/\cD\right)$$ is an equivalence. This observation combined with the homotopy commutative square (\ref{eq:sqqq}) implies that the following diagram commutes up to canonical homotopy
$$\xymatrix@C=2.5cm{\LShi\left(\dbare\right) \ar[r]^-{l^*} \ar[d] & \LShi\left(\sD^{\et}\right) \ar[d]\\
\LShi\left(\dbare/\cD\right) \ar[r]^-{\left(l/\cD\right)^*} & \LShi\left(\sD^{\et}/\cD\right),}$$
and hence, by Corollary \ref{cor:colimpset}, it follows that $l^*$ preserves $\V$-small colimits. It follows that
the composite $$l_!\circ l^* \circ \widehat{a_{\dbar}}^{\et} \circ \widehat{\overline{y}}^{\et}$$ preserves small colimits.
Finally, the following string of equivalences
\begin{eqnarray*}
l_!l^* \widehat{\overline{a}}^{\et} \widehat{\overline{y}}^{\et}\left(\cE\right) &\simeq & \colim l_!l^* \widehat{\overline{a}}^{\et} \widehat{\overline{y}}^{\et}\left(l\cD_\alpha\right)\\
&\simeq& \colim l_! \widehat{a}^{\et} \widehat{y}^{\et}\left(\cD_\alpha\right)\\
&\simeq& \colim \widehat{\overline{a}}^{\et} \widehat{\overline{y}}^{\et}\left(l\cD_\alpha\right)\\
&\simeq& \widehat{\overline{a}}^{\et} \widehat{\overline{y}}^{\et}\left(\colim l\cD_\alpha\right)\\
&\simeq& \widehat{\overline{a}}^{\et} \widehat{\overline{y}}^{\et}\left(\cE\right)
\end{eqnarray*}
show that the component of $\varepsilon$ along $\widehat{\overline{a}}^{\et} \widehat{\overline{y}}^{\et}\left(\cE\right)$ is an equivalence. It follows that
\begin{equation}\label{eq:adjeql}
l_! \dashv l^*
\end{equation}
is an adjoint equivalence.

We now wish to show that the adjoint equivalence (\ref{eq:adjeql}) restricts to an adjoint equivalence of the form (\ref{eq:smallet}). It follows immediately from definitions that $l^*$ restricts to a functor $$l^*:\Shi\left(\dbare\right) \to \Shi\left(\sD^{\et}\right).$$ We wish to show the same is true for $l_!.$ Since both the unit $\eta$ and the co-unit $\varepsilon$ have been shown to be equivalences, this implies that $l_!$ is also \emph{right} adjoint to $l^*.$ Now suppose that $G$ is in $\Shi\left(\sD^{\et}\right).$ We wish to show that $l_!\left(G\right)\left(\cE\right)$ is essentially small for all $\cE$ in $\dbar.$ We can again write $\cE$ as a small colimit $$\cE \simeq \colim \cD_\alpha$$ in $\dbare$ of objects in $\sD.$ By the Yoneda lemma, one then has
\begin{eqnarray*}
l_!\left(G\right)\left(\cE\right) &\simeq& \Hom\left(\widehat{\overline{a}}^{\et} \widehat{\overline{y}}^{\et}\left(\cE\right),l_!G\right)\\
&\simeq& \Hom\left(l^*\widehat{\overline{a}}^{\et} \widehat{\overline{y}}^{\et}\left(\cE\right),G\right)\\
&\simeq& \Hom\left(\colim l^*\widehat{\overline{a}}^{\et} \widehat{\overline{y}}^{\et}\left(l \cD_\alpha\right),G\right)\\
&\simeq& \lim \Hom\left(\widehat{\overline{a}}^{\et} \widehat{y}^{\et}\left(\cD_\alpha\right),G\right)\\
&\simeq& \lim G\left(\cD_\alpha\right),
\end{eqnarray*}
which is essentially small.
\end{proof}

\begin{cor}
If $\sD$ and $\sD'$ are two locally small strong \'etale blossoms for $\lbar,$ then there are canonical equivalences
\begin{equation*}
\Shi\left(\sD\right) \simeq \Shi\left(\sD'\right),
\end{equation*}
\begin{equation*}
\LShi\left(\sD\right) \simeq \LShi\left(\sD'\right),
\end{equation*}
\begin{equation*}
\Shi\left(\sD^{\et}\right) \simeq \Shi\left(\sD'^{\et}\right),
\end{equation*}
and
\begin{equation*}
\LShi\left(\sD^{\et}\right) \simeq \LShi\left(\sD'^{\et}\right).
\end{equation*}
\end{cor}

We learned the proof of the following lemma from Gijs Heuts:

\begin{lem}
Let $\C$ be an $\icat$ and let $$\rho:\C^\triangleright \to \C$$ be a cocone for the identity functor, with vertex $E$. By adjunction, this corresponds to a functor $$\bar \rho:\C \to \C/E.$$ Suppose that $$\bar\rho\left(E\right):E \to E$$ is an equivalence in $\C.$ Then $E$ is a terminal object.
\end{lem}
\begin{proof}
We wish to show that the canonical projection functor $$\C/E \to \C$$ is a trivial Kan fibration. By adjunction, the lifting problem
$$\xymatrix{\partial \Delta^n \ar[r] \ar[d] & \C/E \ar[d]\\
\Delta^n\ar[r] \ar@{-->}[ur] & \C}$$
is equivalent to the lifting problem
$$\xymatrix{\partial \Delta^{n+1} \ar[r]^-{f} \ar[d] & \C\\
\Delta^{n+1} \ar@{-->}[ur] &},$$
in which $f\left([n+1]\right)=E.$ Consider the diagram
$$\bar \rho \circ f:\partial \Delta^{n+1} \to \C/E,$$
which by adjunction corresponds to the diagram
$$\varphi:\Lambda^{n+2}_{n+2} \cong \left(\partial \Delta^{n+1}\right)^\triangleright \to \C$$ such that $$\varphi\left([n+2]\right)=E.$$ Note that $$\varphi\left(\Delta^{\left\{n+2,n+1\right\}}\right)=\bar\rho\left(E\right).$$
By Joyal's characterization of equivalences, (\cite{htt}, Proposition 1.2.4.3), since $$\varphi\left(\Delta^{\left\{n+2,n+1\right\}}\right)$$ is an equivalence, there exists an extension
$$\xymatrix{\Lambda^{n+2}_{n+2} \ar[r]^-{\varphi} \ar[d] & \C\\
\Delta^{n+2} \ar@{-->}[ur]_-{\psi} & }.$$
Finally, observe that $\psi|_{\Delta^{\left\{0,\cdots,n+1\right\}}}$ is a solution to the lifting problem for $f.$
\end{proof}

\begin{lem}
If $\C$ is an $\icat$ and $$\rho:\C^\triangleright \to \C$$ is a colimiting cocone for the identity functor, with vertex $E,$ then the component along $E$ of $\rho$ is an equivalence.
\end{lem}
\begin{proof}
The cocone $\rho$ induces a functor
$$\widehat{\rho}:\Hom_\C\left(E,E\right) \to \Hom_{\Fun\left(\C,\C\right)}\left(id_\C,\Delta_E\right)$$ by composition, uniquely defined up to a contractible spaces of choices, where $\Delta_E$ is the constant functor with vertex $E$. Since $\rho$ is colimiting, $\widehat{\rho}$ is an equivalence of $\i$-groupoids, by \cite{htt}, Lemma 4.2.4.3. Since $\rho$ is a cocone, it follows that $\widehat{\rho}\left(\bar\rho\left(E\right)\right)$ and $\rho=\widehat{\rho}\left(id_E\right)$ are equivalent cocones, i.e. equivalent objects in the $\i$-groupoid $$\Hom_{\Fun\left(\C,\C\right)}\left(id_\C,\Delta_E\right).$$ Since $\widehat{\rho}$ is an equivalence, it follows that $\bar\rho\left(E\right)$ is equivalent to $id_E$ in $\Hom_\C\left(E,E\right),$ and in particular, is an equivalence.
\end{proof}

\begin{cor}\label{cor:colimernal}
If $\C$ is an $\icat$ and $E \in \C$ is a colimit for the identity functor, then $E$ is a terminal object.
\end{cor}

We will now give a more Yoneda-like formulation of what it means to be a left Kan extension:\\ \newline
Suppose that we are given a diagram of $\i$-categories $$\xymatrix{\sA \ar[r]^-{h} \ar[d]_-{\omega} & \sE\\
\sB. & }$$ Consider the functor
\begin{eqnarray*}
L_\omega h :\Fun\left(\sB,\sE\right) &\to& \LiGpd\\
\sB \stackrel{\varphi}{\longrightarrow} \sE &\mapsto& \Hom_{\Fun\left(\sA,\sE\right)}\left(h,w^* \varphi\right).
\end{eqnarray*}
For each object $\psi$ of $\Fun\left(\sB,\sE\right)$ denote by $$y_{co}\left(\psi\right):\Fun\left(\sB,\sE\right) \to \LiGpd$$ the functor it co-represents. The data of a left Kan extension of $h$ along $\omega$ can be identified with the data of a functor $$\Lan_\omega\left(h\right):\sB \to \sE$$ together with an equivalence $$\theta:y_{co}\left(\Lan_\omega\left(h\right)\right) \stackrel{\sim}{\longrightarrow} L_\omega h$$ in $\Fun\left(\Fun\left(\sB,\sE\right),\LiGpd\right).$ Indeed, by the Yoneda lemma applied to $\Fun\left(\sB,\sE\right)^{op},$ one has equivalences
\begin{eqnarray*}
\Hom_{\Fun\left(\sA,\sE\right)}\left(h,w^*\Lan_\omega h\right)&=&\left(L_\omega h\right)\left(\Lan_\omega h\right)\\
&\simeq& \Hom_{\Fun\left(\Fun\left(\sB,\sE\right),\LiGpd\right)}\left(y_{co}\left(\Lan_\omega h\right),L_\omega h\right)
\end{eqnarray*}
under which $\theta$ corresponds to a morphism $$h \to w^*\Lan_\omega h$$ exhibiting $\Lan_\omega h$ as an extension of $h;$  since $\theta$ is an equivalence, this extension is in fact a left Kan extension. We will often abuse notation and suppress the mention of the equivalence $\theta.$ In particular, when $\sB$ is contractible, it follows that under the identification $$\Fun\left(\sB,\sE\right) \simeq \sE$$ the functor $\Lan_\omega h$ corresponds to a choice of an object $E \in \sE_0$ and $\theta$ corresponds to a colimiting cocone $$h \to \Delta_E,$$ where $\Delta_E$ is the constant functor with vertex $E.$

\begin{lem}\label{lem:kext}
Let $$\sA \stackrel{w}{\longrightarrow} \sB \stackrel{y}{\longrightarrow} \C$$  be a diagram of composable functors between $\i$-categories, and suppose that $$h:\sA \to \sE$$ is another such functor. Suppose that the left Kan extension $$\Lan_{w}h$$ exists. Then, if the left Kan extension $$\Lan_{yw}\left(h\right)$$ also exists, it can be exhibited as a left Kan extension of $\Lan_{w}h$ along $y,$ i.e. $$\Lan_{yw}\left(h\right)=\Lan_y\left(\Lan_w h\right).$$
\end{lem}
\begin{proof}
Notice that $\Lan_{yw}\left(h\right)$ satisfies the following universal property for all $$\varphi:\C \to \sE:$$
\begin{eqnarray*}
\Hom_{\Fun\left(\C,\sE\right)}\left(\Lan_{yw}h,\varphi\right) & \simeq& \Hom_{\Fun\left(\sA,\sE\right)}\left(h,\left(yw\right)^*h\right)\\
&\simeq& \Hom_{\Fun\left(\sA,\sE\right)}\left(h,w^*\left(y^*h\right)\right)\\
&\simeq& \Hom_{\Fun\left(\sB,\sE\right)}\left(\Lan_{w}h, y^*h\right),\\
\end{eqnarray*}
which is the universal property uniquely defining a left Kan extension of $\Lan_{w} h$ along $y.$
\end{proof}

\begin{lem}\label{lem:termexists}
If $\mathscr{E}$ is a cocomplete $\icat$ such that there exists a full and faithful inclusion $$i:\C \hookrightarrow \mathscr{E}$$ from a small subcategory $\C$ which is strongly generating, then $\mathscr{E}$ has a terminal object.
\end{lem}

\begin{proof}
Denote by $$t:\mathscr{E} \to 1$$ the essentially unique functor to the terminal $\icat$. Since $\mathscr{E}$ is cocomplete and $\C$ is small, a colimit for $i$ exists. Such a colimit, $\colim i$ is a left Kan extension of $i$ along $t \circ i$. Observe that by Lemma \ref{lem:kext} $\colim i$ is a left Kan extension of $$\Lan_i i \simeq id_\mathscr{E}$$ along $t,$ i.e. a colimit for $id_\mathscr{E}$. By Corollary \ref{cor:colimernal}, this must be a terminal object.
\end{proof}

\begin{cor}\label{cor:trex1}
If $\sL$ is a small subcategory of $\Str_\G,$ $\overline \sL^{\et}$ has a terminal object $\Uni.$
\end{cor}

\begin{proof}
Since $\sL$ is a small, by Proposition \ref{prop:existbloss} there exists a small strong \'etale blossom $\sD$ for $\lbar.$ Moreover, by Proposition \ref{prop:etlocsmall}, $\sD^{\et}$ is locally small.  Corollary \ref{cor:dense} and Lemma \ref{lem:holmol} then imply that $\sD^{\et}$ strongly generates $\lbare,$ and hence we are done by Lemma \ref{lem:termexists}.
\end{proof}

\begin{thm}\label{thm:trex2}
If $\sL$ is a small subcategory of $\Str_\G,$ $\overline \sL^{\et}$ is an $\i$-topos. Moreover, this $\i$-topos comes equipped with a canonical $\G$-structure $$\left(\overline \sL^{\et},\O_{\Uni}\right)$$ making it an object of $\overline \sL.$ When viewed as an object of $\overline \sL^{\et},$ this object is terminal.
\end{thm}

\begin{proof}
Let $\left(\Uni,\O_\Uni\right)$ denote the terminal object of $\overline \sL^{\et}.$ Then, on one hand, since it is terminal,
$$\overline \sL^{\et}/ \Uni \simeq \overline \sL^{\et}.$$
On the other hand, by Remark \ref{rmk:2352}, $$\overline \sL^{\et}/ \left(\Uni,\O_\Uni\right) \simeq \Uni.$$ Hence,
\begin{equation}\label{eq:uni}
\overline \sL^{\et} \simeq \Uni
\end{equation} is an $\i$-topos and under the identification (\ref{eq:uni}), $$\O_\Uni:\Uni \to \cB$$ is a canonical $\G$-structure on $\overline \sL^{\et},$ and $\overline \sL^{\et}$ with this $\G$-structure is a terminal object of $\overline \sL^{\et}.$
\end{proof}

\begin{dfn}
For $\sL$ a small subcategory of $\Str_\G,$ the $\G$-structured $\i$-topos
$$\left(\lbare=\Uni,\O_\Uni\right)$$ is the \textbf{universal $\sL$-\'etendue}.
\end{dfn}

\begin{thm}\label{thm:yonequil}
For $\sD$ a small strong \'etale blossom, the functor $$\dbare \stackrel{\left(i^{\et}_\sD\right)^* \circ y^{\et}}{\longlonglongrightarrow} \Shi\left(\sD^{\et}\right)$$ is an equivalence.
\end{thm}

\begin{proof}
By Lemma \ref{lem:sheavesthesame}, it suffices to show that the Yoneda embedding $$y^{\et}:\dbare \hookrightarrow \Shi\left(\dbare\right)$$ is an equivalence. Since the Yoneda embedding is full and faithful, it suffices to show it is essentially surjective. Let $F$ be an $\i$-sheaf on $\dbare.$ We wish to show that $F$ is representable. Since $F$ is an $\i$-sheaf on $\dbare,$ by Proposition \ref{prop:loclimsheaf}, $\pi_\Uni^*\left(F\right)$ is representable, where $$\pi_\Uni:\dbare/\Uni \to \dbare$$ is the canonical projection, which is an equivalence. Hence, $F$ is representable.
\end{proof}

The following proposition gives an explicit inverse functor for $\left(i^{\et}_\sD\right)^* \circ y^{\et}$:

\begin{prop}\label{prop:theta}
For $\sD$ a small strong \'etale blossom, the functor
\begin{eqnarray*}
\Shi\left(\sD^{\et}\right)=\Uni &\to& \dbare\\
F &\mapsto& \left(\Uni/F,\O_\Uni|_F\right)
\end{eqnarray*}
is right adjoint to $\left(i^{\et}_\sD\right)^* \circ y^{\et}.$ In particular, it is an equivalence.
\end{prop}

\begin{proof}
Let $F$ be an $\i$-sheaf on $\sD^{\et}$ and let $\cE$ be an object of $\dbare.$ Notice that by Theorem \ref{thm:yonequil} and Remark \ref{rmk:2352}, there is a canonical equivalence $$\Uni/\left(\left(i^{\et}_\sD\right)^*\overline{y}^{\et}\left(\cE\right)\right) \simeq \cE,$$ hence there is a canonical equivalence
$$\Hom_{\dbare}\left(\cE,\Uni/F\right) \simeq \Hom_{\dbare}\left(\Uni/\left(\left(i^{\et}_\sD\right)^*\overline{y}^{\et}\left(\cE\right)\right),\Uni/F\right).$$ Since $\Uni$ is terminal in $\dbare,$ it follows that one has a canonical equivalence
$$\Hom_{\dbare}\left(\Uni/\left(\left(i^{\et}_\sD\right)^*\overline{y}^{\et}\left(\cE\right)\right),\Uni/F\right) \simeq \Hom_{\dbare/\Uni}\left(\Uni/\left(\left(i^{\et}_\sD\right)^*\overline{y}^{\et}\left(\cE\right)\right) \to \Uni,\Uni/F \to \Uni\right).$$
By Proposition \ref{prop:trivkan}, one finally has the equivalence
$$\Hom_{\dbare/\Uni}\left(\Uni/\left(\left(i^{\et}_\sD\right)^*\overline{y}^{\et}\left(\cE\right)\right) \to \Uni,\Uni/F \to \Uni\right) \simeq \Hom_{\Shi\left(\sD^{\et}\right)}\left(\left(i^{\et}_\sD\right)^*\overline{y}^{\et}\left(\cE\right),F\right).$$
\end{proof}

\begin{rmk}\label{rmk:theta}
It follows from the Yoneda lemma that for a given $\i$-sheaf $F,$ $\left(\Uni/F,\O_\Uni|_F\right)$ is uniquely determined by the property that for each object $\left(D,\O_D\right)$ in $\sD,$ there is a equivalence of $\i$-groupoids
$$\Hom_{\dbare}\left(\left(D,\O_D\right),\left(\Uni/F,\O_\Uni|_F\right)\right) \simeq F\left(\left(D,\O_D\right)\right)$$ natural in $\left(D,\O_D\right).$
\end{rmk}

\begin{lem}\label{lem:ntruncatedlocalic}
If $\left(\cB,\G\right)$ is a geometric structure, with $\cB$ $n$-localic, then any $n$-localic object $\left(\cE,\O_\cE\right)$ in $\Str\left(\G\right)$ is $n$-truncated.
\end{lem}

\begin{proof}
We will closely follow the proof of Lemma 2.6.19 of \cite{htt}. Let $\left(\cF,\O_\cF\right)$ be an arbitrary object of $\Str\left(\G\right)$. Let $$\phi:\Hom\left(\left(\cF,\O_\cF\right),\left(\cE,\O_\cE\right)\right) \to \Hom\left(\cF,\cE\right)$$ be the obvious map. Since $\cE$ is $n$-localic, the right hand side is equivalent to $\Hom\left(\tau_{\le n-1} \cF,\tau_{\le n-1} \cE\right)$ which is an $n$-groupoid, i.e. an $n$-truncated object of the $\icat$ of $\i$-groupoids. By \cite{htt}, Proposition 5.5.6.14, $\Hom\left(\left(\cF,\O_\cF\right),\left(\cE,\O_\cE\right)\right)$ is $n$-truncated if and only if $\phi$ is. We therefore must prove that each homotopy fiber of $\phi$ is $n$-truncated. An arbitrary point of $\Hom\left(\cF,\cE\right)$ is given by a geometric morphism $$f:\cF \to \cE.$$ The homotopy fiber of $\phi$ over such a point can be identified with the full subcategory of $$\Hom_{\Geo\left(\cF,\cB\right)} \left(f^*\O_\cE,\O_\cE\right)$$ spanned by those morphisms in $\G^{\cF}_R.$ Since $\cB$ is $n$-localic, this is $n$-truncated so we are done.
\end{proof}

\begin{cor}\label{cor:ntruncatedlocalic}
If $\left(\cB,\G\right)$ is the underlying geometric structure, with $\cB$ $n$-localic, and $\sL$ is a small subcategory of $\Str_\G$ all of whose objects are $n$-localic, then the universal $\sL$-\'etendue $\Uni$ is $\left(n+1\right)$-localic.
\end{cor}
\begin{proof}
By Theorem \ref{thm:yonequil}, $$\overline{\sL}^{\et} \simeq \Shi\left(\sD^{\et}\right),$$ for a small and locally small strong \'etale blossom $\sD$, which we moreover may assume consists of $n$-localic objects by Proposition \ref{prop:stronglocsmallblossom}. It follows from Lemma \ref{lem:ntruncatedlocalic} that $\sD^{\et}$ is an $\left(n+1\right)$-category, so hence $\lbare$ is $\left(n+1\right)$-localic, as desired.
\end{proof}

Suppose that $\sD$ is a small and locally small strong \'etale blossom. Consider the inclusion $$j:\sD^{\et} \to \sD.$$ The left Kan extension
$$\Lan_y^{\et}\left(y \circ j\right):\Pshi\left(\sD^{\et}\right) \to \Shi\left(\sD\right),$$ has a right adjoint $R_j$, which by the Yoneda lemma is given by the formula $$R_j\left(F\right)\left(\cD\right) \simeq j^*F\left(\cD\right).$$ By Definition \ref{dfn:sheaveD}, this implies that the essential image of $R_j=j^*$ lands in sheaves. Hence, there is an induced functor $$j_!:\Shi\left(\sD^{\et}\right) \to \Shi\left(\sD\right),$$ which is left adjoint to $j^*.$

\begin{dfn}\label{dfn:prolongation}
We call the functor $j_!$ above the \textbf{\'etale prolongation functor} of $\sD.$
\end{dfn}

\begin{thm}\label{thm:classificationsmall}
An $\i$-sheaf $\X$ on a small and locally small strong \'etale  blossom $\sD$ is a Deligne-Mumford stack if and only if it is in the essential image of the \'etale prolongation functor $j_!$.
\end{thm}

\begin{proof}
By Corollary \ref{cor:dense} and Corollary \ref{cor:5.2.9.4}, $\dbare$ is a reflective subcategory of $\Pshi\left(\sD^{\et}\right).$ The reflector may be identified with the reflector $$a:\Pshi\left(\sD^{\et}\right) \to \Shi\left(\sD^{\et}\right)$$ via the equivalence $$\dbare \stackrel{\overline{y}^{\et}}{\underset{{\widetilde{}}} \longhookrightarrow \:} \Shi\left(\dbare\right) \stackrel{\left(i^{\et}_\sD\right)^*}{\underset{{\widetilde{}}} \longlonglongrightarrow  \:} \Shi\left(\sD^{\et}\right),$$ ensured by Theorem \ref{thm:yonequil} and Lemma \ref{lem:sheavesthesame}, and hence is accessible by Proposition \ref{prop:sheafification}.
We claim that the following diagram commutes up to canonical homotopy:
\begin{equation}\label{eq:25}
\xymatrix{ \dbare \ar[r]^-{{\overline{j}}} \ar@{^{(}->}[d]_-{{\overline{y}^{\et}}}^{\rotatebox[origin=C]{90}{$\sim$}} & \dbar \ar@{_{(}->}[d]^-{\overline{y}}\\
\Shi\left(\dbare\right) \ar[d]_-{\left(i^{\et}_\sD\right)^\ast}^{\rotatebox[origin=C]{90}{$\sim$}} & \Shi\left(\dbar\right) \ar[d]^-{\left(i_\sD\right)^\ast}_{\rotatebox[origin=C]{90}{$\sim$}}\\
\Shi\left(\sD^{\et}\right) \ar[r]^-{j_!} & \Shi\left(\sD\right).}
\end{equation}
The fact that $j_!$ is a left adjoint together with Proposition \ref{prop:bothfuncolims} implies that both composite functors $$\dbare \to \Shi\left(\sD\right)$$ preserve small colimits. As $\dbare$ is an accessible localization of $\Pshi\left(\sD^{\et}\right),$ by Proposition 5.5.4.20 and Theorem 5.1.5.6 of \cite{htt}, it suffices to show that both functors agree after being restricted to $\sD^{\et},$ which is true by definition. It follows that the essential image of $j_!$ is equivalent to the essential image of $i_\sD^* \circ y,$ which is precisely the $\icat$ $\dm\left(\sD\right)$ of Deligne-Mumford stacks, by definition.
\end{proof}

The following proposition gives a more explicit description of the Deligne-Mumford stack associated to an $\i$-sheaf on $\sD^{\et}$:

\begin{prop}\label{prop:explicit}
Let $\sD$ be a small and locally small strong \'etale blossom with associated universal \'etendue $\Uni.$ Then if $F$ is an $\i$-sheaf on $\sD^{\et},$ the associated Deligne-Mumford stack $j_!\left(F\right)$ is the functor of points of $\Uni/F.$
\end{prop}

\begin{proof}
Denote by $\psi$ the equivalence
\begin{eqnarray*}
\Shi\left(\sD^{\et}\right) &\stackrel{\sim}{\longlongrightarrow}& \dbare/\Uni\\
F &\mapsto& \left(\Uni/F \to \Uni\right)
\end{eqnarray*}
which is inverse to the equivalence
$$\dbare/\Uni \stackrel{\sim}{\longlongrightarrow} \dbare \stackrel{\left(i^{\et}_\sD\right)^* \circ {\overline{y}^{\et}}}{\underset{{\widetilde{}}} \longlonglongrightarrow  \:} \Shi\left(\sD^{\et}\right).$$ Since the diagram (\ref{eq:25}) commutes, the result follows.
\end{proof}


\begin{dfn}
A geometric morphism $f:\cE \to \cF$ between $n$-topoi, with $0 \le n \le \i,$ is \textbf{connected} if the inverse image functor $f^*:\cF \to \cE$ is full and faithful.
\end{dfn}

\begin{prop}\label{prop:connected}
Let $\sD$ be a small and locally small strong \'etale blossom. Then the \'etale prolongation functor $$j_!:\Shi\left(\sD^{\et}\right) \to \Shi\left(\sD\right),$$ induces a connected geometric morphism $$\Lambda:\Shi\left(\sD\right)/j_!\left(1\right) \to \Shi\left(\sD^{\et}\right).$$
\end{prop}

\begin{proof}
From the proof of Theorem \ref{thm:classificationsmall}, it follows that $$j_!\left(1\right)\simeq i_\cD^* y\left(\Uni\right).$$ Furthermore, the following diagram commutes up to canonical homotopy:
$$\xymatrix@C=2cm{ \dbare/\Uni \ar[d]_-{\rotatebox[origin=C]{90}{$\sim$}} \ar[r]^-{\overline{j}/\Uni} & \dbar/\Uni \ar@{_{(}->}[dd]^-{ y/\Uni}\\
\dbare \ar@{^{(}->}[d]_-{{\overline{y}^{\et}}}^{\rotatebox[origin=C]{90}{$\sim$}} & \\
\Shi\left(\dbare\right) \ar[d]_-{\left(i^{\et}_\sD\right)^\ast}^{\rotatebox[origin=C]{90}{$\sim$}} & \Shi\left(\dbar\right)/y\left(\Uni\right) \ar[d]^-{\left(i_\sD\right)^\ast/\Uni}_{\rotatebox[origin=C]{90}{$\sim$}}\\
\Shi\left(\sD^{\et}\right) \ar[r]^-{j_!/\Uni} & \Shi\left(\sD\right)/j_!\left(1\right).}$$
Notice that furthermore the diagram
$$\xymatrix@C=2.5cm{\dbare/\Uni \ar[d]_-{\rotatebox[origin=C]{90}{$\sim$}} \ar[r]^-{j/\Uni} & \dbar/\Uni \ar@{^{(}->}[r]^-{y/\Uni} & \Shi\left(\dbar\right)/\Uni\ar[d]\\
\dbare \ar[r]^-{j} & \dbar \ar@{^{(}->}[r]^-{y} & \Shi\left(\dbar\right),}$$
also commutes up to homotopy. It then follows from Proposition \ref{prop:bothfuncolims} and Proposition \ref{prop:htt1.2.13.8}, that the composite $y/\Uni \circ \overline{j}/\Uni$ preserves small colimits. Notice also that it follows from Corollary \ref{cor:2.3.20} that $\overline{j}/\Uni$ preserves finite limits, and since $y/\Uni$ preserves all small limits, we conclude that $y/\Uni \circ \overline{j}/\Uni$ preserves all finite limits. It follows that $j_!/\Uni$ is left exact and preserves small colimits. By Corollary 5.5.2.9 of \cite{htt}, it follows that $j_!/\Uni$ has a right adjoint. Let us denote $j_!/\Uni$ by  $\Lambda^*.$ It follows from the Yoneda lemma that its right adjoint can be described as the functor $\Lambda_*$ which sends an object $$f:F \to j_!\left(1\right)$$ in $\Shi\left(\sD\right)/j_!\left(1\right)$ to the $\i$-sheaf on $\sD^{\et}$ which assigns each $\cD$ in $\sD^{\et}$ the $\i$-groupoid $$\Lambda_*\left(f:F \to j_!\left(1\right)\right)\left(\cD\right) \simeq \Hom_{\Shi\left(\cD\right)/j_!\left(1\right)}\left(j_!y^{\et}\left(t_\cD\right),f\right),$$ where $t_\cD:\cD \to \Uni$ is the essentially unique map to the terminal object. Since $\Lambda^*$ is left exact, the adjoint pair $$\Lambda^* \dashv \Lambda_*$$ constitute a geometric morphism $$\xymatrix{\Shi\left(\sD\right)/j_!\left(1\right) \ar@<-0.5ex>[r]_-{\Lambda_*} & \Shi\left(\sD^{\et}\right) \ar@<-0.5ex>[l]_-{\Lambda^*}}.$$ It suffices to show that $\Lambda^*$ is full and faithful. Notice that for $G$ and $G'$ in $\Shi\left(\sD^{\et}\right),$ under slight abuse of notation induced by the equivalence $$\Uni \simeq \Shi\left(\sD^{\et}\right),$$ one has
$$\Hom_{\Shi\left(\cD\right)/j_!\left(1\right)}\left(j_!\left(t_G\right),j_!\left(t_{G'}\right)\right) \simeq \Hom_{\dbar/\Uni}\left(\Uni/G \to \Uni,\Uni/G' \to \Uni\right).$$
Since both of the maps in question are \'etale, it follows from Proposition \ref{prop:6.3.5.9} that this $\i$-groupoid is equivalent to
$$\Hom_{\dbare/\Uni}\left(\Uni/G \to \Uni,\Uni/G' \to \Uni\right) \simeq \Hom_\Uni\left(G,G'\right).$$
\end{proof}

\begin{thm}\label{thm:smalltruncatedversion}
Suppose that $\left(\cB,\G\right)$ is the underlying geometric structure, with $\cB$ $n$-localic. Let $\sL$ be a small and strongly locally small subcategory of $\Str_\G$ and suppose that $\lbar$ has a small and locally small strong \'etale blossom $\sD$ which is equivalent to an $n$-category. Denote by $$y_n:\sD \hookrightarrow \Sh_{n}\left(\sD\right)$$ and $$y^{\et}_n:\sD \hookrightarrow \Sh_{n}\left(\sD^{\et}\right)$$ the Yoneda embeddings into $n$-sheaves (i.e. sheaves of $n$-groupoids), where $\Sh_n\left(\sD\right)$ denotes the $n$-truncated objects of $\Shi\left(\sD\right),$ and similarly for $\Sh_{n}\left(\sD^{\et}\right),$ and denote by $j^n_!$ the left Kan extension $$\Lan_{y^{\et}_n}\left(y_n \circ j\right):\Sh_{n}\left(\sD^{\et}\right)\to \Sh_{n}\left(\sD\right)$$- call it the \textbf{$n$-truncated \'etale prolongation functor.} An $n$-stack in $\Sh_{n}\left(\sD\right)$ is Deligne-Mumford if and only if it is in the essential image of $j^n_!.$
\end{thm}

\begin{proof}
Notice that $j_!\left(1\right)\simeq i_\sD^* \overline{y}\left(\Uni\right),$ and $\Uni\simeq\Shi\left(\sD^{\et}\right)$ is $n$-localic, so by Lemma \ref{lem:ntruncatedlocalic} it is $n$-truncated, when regarded as an object of $\dbare$. Therefore, $j_!\left(1\right)$ is $n$-truncated, so may be regarded as an object of $\Sh_{n}\left(\sD\right).$ Suppose that $\X$ is a Deligne-Mumford $n$-stack. By Theorem \ref{thm:classificationsmall}, there exists (an essentially unique) $G$ in $\Shi\left(\sD^{\et}\right),$ such that $j_!\left(G\right) \simeq \X.$ By Proposition \ref{prop:explicit}, $$j_!\left(G\right) \simeq i_\sD^* \overline{y}\left(\Uni/G\right),$$ so it follows that $\Uni/G$ is $n$-truncated as an object of $\dbar$. By Theorem \ref{thm:yonequil}, it follows that $G \simeq \left(i^{\et}_\sD\right)^*\overline{y}^{\et}\left(\Uni/G\right),$ and hence $G$ is also $n$-truncated. Conversely, suppose that $G$ is $n$-truncated, then by Proposition \ref{prop:2.3.16}, $\Uni/G$ is $n$-localic, and hence $n$-truncated by Lemma \ref{lem:ntruncatedlocalic}. It follows that $\X=j_!\left(G\right)$ is $n$-truncated.

So far, we have shown that the Deligne-Mumford $n$-stacks are precisely the $\i$-sheaves in the essential image of the restriction of $j_!$ to $\Sh_{n}\left(\sD^{\et}\right).$ In other words, the functor $j_!$ restricts to a functor $$\tau_{\le n}j_!:\Sh_{n}\left(\sD^{\et}\right) \to \Sh_{n}\left(\sD\right),$$ the essential image of which is the Deligne-Mumford $n$-stacks. Moreover, by construction, this functor agrees with $j^n_!$ on representables. It suffices to show it preserves colimits. Notice that $j_!$ factors as
$$\Shi\left(\sD^{\et}\right) \stackrel{\Lambda^*}{\longlongrightarrow} \Shi\left(\sD\right)/j_!\left(1\right) \to \Shi\left(\sD\right).$$ Since $\Lambda^*$ is left exact, it preserves $n$-truncated objects by Proposition 5.5.6.16 of \cite{htt}. Observe that since $j_!\left(1\right)$ is $n$-truncated, by Lemma 5.5.6.14 of \cite{htt}, an object $$f:F \to j_!\left(1\right)$$ in $\Shi\left(\sD\right)/j_!\left(1\right)$ is $n$-truncated if and only if $F$ is. Since $\sD$ is an $n$-category, by Proposition \ref{prop:2.3.16}, $\Shi\left(\sD\right)/j_!\left(1\right)$ is $n$-localic, hence in particular $\left(n+1\right)$-localic. Its associated $\left(n+1\right)$-topos is just the full subcategory on the $n$-truncated objects of $\Shi\left(\sD\right)/j_!\left(1\right),$ and the latter we have just argued can be identified with $\Sh_{n}\left(\sD\right)/j_!\left(1\right).$ Since $\sD^{\et}$ is also an $n$-category, both the domain and codomain of the geometric morphism $$\Lambda:\Shi\left(\sD\right)/j_!\left(1\right) \to \Shi\left(\sD^{\et}\right)$$ are $n$-localic, hence also $\left(n+1\right)$-localic. Since the functor $$\nu_{n+1}:\mathfrak{Top}_{n+1} \hookrightarrow \T$$ is full and faithful (see Definition \ref{dfn:nloc}), $\Lambda$ corresponds to a geometric morphism $$\Lambda_{n+1}:\Sh_{n}\left(\sD\right)/j_!\left(1\right) \to \Sh_{n}\left(\sD^{\et}\right)$$ of $\left(n+1\right)$-topoi, such that the following diagram commutes:
$$\xymatrix@C=1.5cm{\Shi\left(\sD^{\et}\right) \ar[r]^-{\Lambda^*} & \Shi\left(\sD\right)/j_!\left(1\right) \ar[r] & \Shi\left(\sD\right)\\
\Sh_{n}\left(\sD^{\et}\right) \ar@{_{(}->}[u] \ar[r]^-{\Lambda_{n+1}^*} & \Sh_{n}\left(\sD\right)/j_!\left(1\right) \ar@{_{(}->}[u] \ar[r] & \Sh_{n}\left(\sD\right). \ar@{^{(}->}[u]}$$
Since the composite of the top row is $j_!,$ the composite of the bottom is its restriction to $\Sh_{n}\left(\sD^{\et}\right).$ However, we have factored  this bottom row by two functors which preserve small colimits, so we are done.
\end{proof}

\begin{prop}\label{prop:Ufirst}
Suppose that $\sU$ is a small and locally small strong \'etale blossom which is a subcategory of a locally small strong \'etale blossom $\sD,$ which need not be small. Denote the full and faithful inclusion by $$\lambda_\sU:\sU \hookrightarrow \sD.$$ Then the restriction functors $$\lambda^{\et\mspace{2mu} *}_\sU:\Pshi\left(\sD^{\et}\right) \to \Pshi\left(\sU^{\et}\right)$$ and $$\lambda_\sU^*:\Pshi\left(\sD\right) \to \Pshi\left(\sU\right)$$ send sheaves to sheaves.
\end{prop}

\begin{proof}
For every $\cD$ in $\sU,$ the following diagram commutes up to canonical homotopy:
$$\xymatrix{\cD \ar@{^{(}->}[r]^-{y} & \Pshi\left(\cD\right) \ar[r]^-{\sim} \ar[rd]_-{\rotatebox[origin=C]{-25}{$\sim$}} & \Pshi\left(\dbare/\cD\right) \ar[d]^{\rotatebox[origin=C]{90}{$\sim$}} \ar[r] &\Pshi\left(\sD^{\et}/\cD\right) \ar[d]^-{\lambda^{\et\mspace{2mu} \ast}_\sU} \\
& & \Pshi\left(\ubare/\cD\right) \ar[r] & \Pshi\left(\sU^{\et}/\cD\right),}$$
where the unlabeled functors are the obvious ones. This implies that $\left(\lambda^{\et}_\sU/\cD\right)^*$ restricts to a functor
$$\left(\lambda^{\et}_\sU/\cD\right)^*:\Shi\left(\sD^{\et}/\cD\right) \to \Shi\left(\sU^{\et}/\cD\right).$$ The result now follows.
\end{proof}

\begin{prop}\label{prop:relprolexist}
With the notation as in Proposition \ref{prop:Ufirst}, denote by $j_\sU$ the composite $$\sU^{\et} \stackrel{\lambda^{\et}_\sU}{\longlonghookrightarrow} \sD^{\et} \stackrel{j}{\longrightarrow} \sD.$$ Then the restriction functor
$$\Pshi\left(\sD\right) \to \Pshi\left(\sU^{\et}\right)$$ restricts to a functor $$j^*_\sU:\Shi\left(\sD\right) \to \Shi\left(\sU^{\et}\right)$$ which has a left adjoint $j^\sU_!.$
\end{prop}

\begin{proof}
The restriction functor in question is equivalent to $\lambda_\sU^* \circ j^*,$ and $j^*$ preserves the sheaf condition by definition, and $\lambda_\sU^*$ does by Proposition \ref{prop:Ufirst}. Consider the left Kan extension
$$\Lan_{y^{\et}}\left(y \circ j_\sU\right):\Pshi\left(\sU^{\et}\right) \to \Shi\left(\sD\right).$$ By the Yoneda lemma, it follows that this is left adjoint to $\lambda_\sU^* \circ j^* \circ i,$ where $$i:\Shi\left(\sD\right) \hookrightarrow \Pshi\left(\sD\right)$$ is the canonical inclusion. It follows that $\Lan_{y^{\et}}\left(y \circ j_\sU\right)$ restricts to a functor $$j^\sU_!:\Shi\left(\sU^{\et}\right) \to \Shi\left(\sD\right)$$ which is left adjoint to $j^*_\sU$.
\end{proof}

\begin{dfn}
We call the functor $j^\sU_!$ above the \textbf{relative \'etale prolongation functor of $\sU$ with respect to $\sD.$}
\end{dfn}

\begin{thm}\label{thm:classificationlarge}
Let $\sD$ be a locally small strong \'etale blossom. An $\i$-sheaf in $\Shi\left(\sD\right)$ is a Deligne-Mumford stack if and only if there exists a small strong \'etale blossom $\sU$ which is a subcategory of $\sD$ such that $\X$ is in the essential image of the relative prolongation functor $j^\sU_!$.
\end{thm}

\begin{proof}
Let $\X$ be a Deligne-Mumford stack. Then $\X \simeq i_\sD^*y\left(\cE\right)$ for some object $\cE$ in $\dbar.$ By Proposition \ref{prop:locblossom}, there exists small full subcategory $\lambda_\sU:\sU \hookrightarrow \sD$ such that $\sU$ is a small strong \'etale blossom, and $\cE$ is in $\ubar.$ In other words, $\X \simeq i_\sD^*y\left(j_\sU\cE\right).$ We claim the following diagram commutes up to canonical homotopy:
\begin{equation}\label{eq:relprol}
\xymatrix@R=1.5cm{\sU^{\et} \ar[r]^-{j_\sU} \ar@{^{(}->}[d]_-{\left(i^{\et}_\sU\right)^*\circ y^{\et}}^-{\rotatebox[origin=C]{90}{$\sim$}} & \dbar \ar@{_{(}->}[d]^{i_\sD^* \circ y} \\
\Shi\left(\sU^{\et}\right) \ar[r]^-{j^\sU_!} & \Shi\left(\sD\right).}
\end{equation}
By definition, $i_\sD^* \circ y \circ j_\sU=i_\sD^* \circ y \circ j \circ \lambda^{\et}_\sU,$ so by Corollary \ref{cor:dense} combined with Lemma \ref{lem:holmol} and Proposition \ref{prop:bothfuncolims}, it follows that this composite preserves small colimits. The other composite does as well by Theorem \ref{thm:yonequil} and the fact that $j^\sU_!$ is a left adjoint. By the same argument as in the proof of Theorem \ref{thm:classificationsmall}, it suffices to check that the restriction of both functors $$\ubare \to \Shi\left(\sD\right)$$ to $\sU^{\et}$ are equivalent, which is true by definition. It follows that $\X \simeq j^\sU_!\left(i^{\et}_\sU\right)^*y^{\et}\left(\cE\right).$ Conversely, suppose that $\X$ is in the essential image of $j^\sU_!$ for some $\sU$ as in the statement of the theorem. Then since the square (\ref{eq:relprol}) commutes up to homotopy, $\X$ is also in the essential image of $i_\sD^* \circ y,$ hence a Deligne-Mumford stack.
\end{proof}

\begin{thm}\label{thm:largetruncatedversion}
Suppose that $\left(\cB,\G\right)$ is the underlying geometric structure, with $\cB$ $n$-localic. Let $\sL$ be a strongly locally small subcategory of $\Str_\G$ and suppose that $\lbar$ has a small and locally small strong \'etale blossom $\sD$ which is equivalent to an $n$-category. For each small strong \'etale blossom $\sU$ which is a subcategory of $\sD,$ denote by $$y^{\sU,et}_n:\sU^{\et} \hookrightarrow \Sh_{n}\left(\sU^{\et}\right)$$ the Yoneda embedding into $n$-sheaves. Denote by $j^{\sU,n}_!$ the left Kan extension $$\Lan_{y^{\sU,et}_n}\left(y_n \circ \lambda_\sU\right):\Sh_{n}\left(\sU^{\et}\right)\to \Sh_{n}\left(\sD\right)$$- call it the \textbf{$n$-truncated relative \'etale prolongation functor of $\sU$ with respect to $\sD.$} An $n$-stack in $\Sh_{n}\left(\sD\right)$ is Deligne-Mumford if and only if it is in the essential image of one such $n$-truncated relative prolongation functor.
\end{thm}

\begin{proof}
Given Theorem \ref{thm:classificationlarge}, the proof is essentially the same as Theorem \ref{thm:smalltruncatedversion}, only some slight care must be taken since $\Shi\left(\sD\right)$ is not an infinity-topos. However, by essentially the same proof as Proposition \ref{prop:connected}, one can construct a left exact colimit preserving functor $$\Lambda^*_\sU:\Shi\left(\sU^{\et}\right) \to \Shi\left(\sD\right)/j^\sU_!\left(1\right).$$ By an analogous argument as in the proof of Theorem \ref{thm:smalltruncatedversion}, this restricts to a functor $$\tau_{\le n}\Lambda^*_{\sU}:\Sh_{n}\left(\sU^{\et}\right) \to \Sh_{n}\left(\sD\right)/j^\sU_!\left(1\right),$$ and one needs only argue this functor preserves colimits. By Corollary \ref{cor:colimps} combined with Proposition \ref{prop:locblossom}, it suffices to show that for every small strong \'etale blossom $\mathscr{V}$ such that $$\sU \hookrightarrow \mathscr{V} \hookrightarrow \sD,$$ the composite $$\Sh_{n}\left(\sU^{\et}\right) \stackrel{\Lambda^*_{\sU}}{\longlongrightarrow} \Sh_{n}\left(\sD\right)/j^\sU_!\left(1\right) \stackrel{\lambda^*_{\mathscr{V}}}{\longlongrightarrow} \Sh_{n}\left(\mathscr{V}\right)$$ preserves colimits. Now the same argument as in the proof of Theorem \ref{thm:smalltruncatedversion} can be applied to the induced geometric morphism $$\Shi\left(\sU^{\et}\right) \to \Shi\left(\mathscr{V}\right)/\lambda^*_{\mathscr{V}}j^{\sU}_!\left(1\right).$$
\end{proof}

We will end this section with a classification of universal \'etendues and $\i$-categories of the form $\lbare,$ for $\sL$ small:

\begin{thm}\label{thm:whatisK}
For $\K$ a full subcategory of $\Str_\G$ the following conditions are equivalent
\begin{itemize}
\item[a)] There exists a small subcategory $\sL$ of $\Str_\G$ such that $\K=\lbar.$
\item[b)] There exists a single object $\left(\cE,\O_\cE\right)$ in $\Str_\G$ such that  $\K=\overline{\{\left(\cE,\O_\cE\right)\}}.$
\item[c)] $\K^{\et}$ has a terminal object.
\item[d)] $\K^{\et}$ is an $\i$-topos.
\item[e)] $\K^{\et}$ is of small generation and cocomplete.
\end{itemize}
\end{thm}

\begin{proof}
Clearly one has that $b) \Rightarrow a).$ Moreover, $a) \Rightarrow c)$ follows from Corollary \ref{cor:trex1}. $d) \Rightarrow c)$ is immediate. Suppose that $\K^{\et}$ has a terminal object $\left(\cE,\O_\cE\right).$ Then the projection functor $$\K^{\et}/\left(\cE,\O_\cE\right) \to \K^{\et}$$ is an equivalence, which implies that one has $\K=\sl\left(\{\left(\cE,\O_\cE\right)\}\right).$ It then follows from Corollary \ref{cor:2352} that $\K^{\et} \simeq \cE,$ and hence $\K^{\et}$ is an $\i$-topos. $d) \Rightarrow e)$ is immediate and $e) \Rightarrow c)$ follows from Lemma \ref{lem:termexists}. It now suffices to show that $c) \Rightarrow b).$ Given $c)$ it follows that $\K^{\et}$ is an $\i$-topos, and that $\K= \sl\left(\{\left(\cE,\O_\cE\right)\}\right),$ where $\left(\cE,\O_\cE\right)$ is its terminal object. It follows immediately that $$\K \subseteq \overline{\{\left(\cE,\O_\cE\right)\}}.$$ Notice that $\K^{\et}$ is cocomplete and the inclusion $$\K^{\et} \hookrightarrow \Str^{\et}_\G$$ can be identified with the projection $$\Str^{\et}_\G/\left(\cE,\O_\cE\right) \to \Str^{\et}_\G.$$ The latter preserves colimits by Proposition \ref{prop:etcolims} and Proposition 6.3.5.14 of \cite{htt}. By Theorem \ref{thm:cosluniv}, it follows that $$\overline{\{\left(\cE,\O_\cE\right)\}} \subseteq \K,$$ hence $c) \Rightarrow b).$
\end{proof}

\begin{thm}
For $\left(\cE,\O_\cE\right)$ a $\G$-structured $\i$-topos, the following conditions are equivalent:
\begin{itemize}
\item[1)] The functor
\begin{eqnarray*}
\cE &\to& \sl\left(\{\left(\cE,\O_\cE\right)\}\right)^{\et}\\
E &\mapsto& \left(\cE/E,\O_\cE|_E\right)
\end{eqnarray*}
is an equivalence.
\item[2)] The functor
\begin{eqnarray*}
\cE &\to& \overline{\{\left(\cE,\O_\cE\right)\}}^{\et}\\
E &\mapsto& \left(\cE/E,\O_\cE|_E\right)
\end{eqnarray*}
is an equivalence.
\item[3)] The projection $\Str^{\et}_\G/\left(\cE,\O_\cE\right) \to \Str^{\et}_\G$ is full and faithful.
\item[4)] $\left(\cE,\O_\cE\right)$ is a terminal object $\K^{\et}$ for a full subcategory $\K$ of $\Str_\G$ satisfying the equivalent conditions of Theorem \ref{thm:whatisK}.
\item[5)] $\left(\cE,\O_\cE\right)$ is a subterminal object of $\Str^{\et}_\G.$
\end{itemize}
\end{thm}

\begin{proof}
To see that $1) \Rightarrow 3)$ notice that given $1),$ one has a factorization
$$\xymatrix{\cE \ar[r]^-{\sim} \ar[d] _-{\rotatebox[origin=C]{90}{$\sim$}} & \Str_\G^{\et}/\left(\cE,\O_\cE\right) \ar[r] & \Str_\G^{\et}\\
\sl\left(\{\left(\cE,\O_\cE\right)\}\right)^{\et}/\left(\cE,\O_\cE\right) \ar[r] & \sl\left(\{\left(\cE,\O_\cE\right)\}\right)^{\et}. \ar@{^{(}->}[ur] &}$$ Suppose now that $3)$ holds. Let $\K^{\et}$ be the essential image of $$\Str^{\et}_\G/\left(\cE,\O_\cE\right) \to \Str^{\et}_\G.$$ On one hand, it follows that $\K^{\et}=\sl\left(\{\left(\cE,\O_\cE\right)\}\right)^{\et},$ so that $\K^{\et}/\left(\cE,\O_\cE\right)\simeq \cE$ is an $\i$-topos, so $\K$ satisfies the conditions of Theorem \ref{thm:whatisK}. On the other hand, the projection $\K^{\et}/\left(\cE,\O_\cE\right) \to \K^{\et}$ is an equivalence, from whence it follows that $\left(\cE,\O_\cE\right)$ is a terminal object of $\K^{\et}.$ Hence, $3) \Rightarrow 4).$ To see $4) \Rightarrow 5),$ suppose that $\left(\cE,\O_\cE\right)$ is terminal in $\K^{\et}$ for $\K=\lbar$ for a small subcategory $\sL$ of $\Str_\G.$ Notice that if there exists an \'etale map $$\left(\cF,\O_\cF\right) \to \left(\cE,\O_\cE\right),$$ it follows that $\cF \in \sl\left(\K^{\et}\right)=\K^{\et},$ in which case $$\Hom_{\Str^{\et}_\G}\left(\left(\cF,\O_\cF\right),\left(\cE,\O_\cE\right)\right) = \Hom_{\K^{\et}}\left(\left(\cF,\O_\cF\right),\left(\cE,\O_\cE\right)\right) \simeq *.$$ Suppose now that $5)$ holds, then let $\K^{\et}$ be the full subcategory of $\Str^{\et}_\G$ on those $\left(\cF,\O_\cF\right)$ for which $\Hom_{\Str^{\et}_\G}\left(\left(\cF,\O_\cF\right),\left(\cE,\O_\cE\right)\right)$ is non-empty. It follows immediately that $\left(\cE,\O_\cE\right)$ is terminal in $\K^{\et},$ and moreover $\K^{\et}$ can be identified with $\sl\left(\{\left(\cE,\O_\cE\right)\}\right)^{\et},$ and hence from Corollary \ref{cor:2352} it follows that $\K^{\et}$ is an $\i$-topos, so $4)$ holds. To see that $4) \Rightarrow 1),$ observe that $\sl\left(\{\left(\cE,\O_\cE\right)\}\right)^{\et}$ is the essential image of $$\Str^{\et}_\G/\left(\cE,\O_\cE\right) \to \Str^{\et}_\G$$ which is the same as the essential image of $$\K^{\et}/\left(\cE,\O_\cE\right) \to \K^{\et},$$ which is simply $\K^{\et},$ and is also equivalent to $\cE.$ Notice that $4) \Rightarrow 2)$ is immediate, whereas $2) \Rightarrow  3)$ follows by essentially the same proof as $1) \Rightarrow 3).$
\end{proof}

\chapter{Examples}\label{chap:examples}

\section{Higher Differentiable Orbifolds and \'Etale Stacks}\label{sec:orbifolds}
In this subsection, we will lay out the basic definitions for higher differentiable orbifolds and \'etale stacks, and interpret the theorems of Chapter \ref{chap:etendues} in this context. We will not develop the full theory here, as this would be deserving of another paper.

We will work in the ambient $\i$-category $\Str_\G$ with $\G=\G\left(\g_{Zar}\left(\RR\right)\right)$ the geometric structure associated to the Zariski geometry over the reals, as defined in Section 2.5 of \cite{dag}. In other words, $\Str_\G$ will be the $\i$-category of $\i$-topoi locally ringed in commutative $\RR$-algebras.

\begin{rmk}
It would be more elegant to work with $\mathbf{C}^\i$-rings rather than $\RR$-algebras, however it will make no difference in the end, so we will not dwell on this here.
\end{rmk}

Notice that if $M$ is a smooth manifold, then viewing it as a locally ringed space gives rise to an induced structure of a locally ringed $\i$-topos $\left(\Shi\left(M\right),\mathbf{C}_M^\i\right).$ Since this structured $\i$-topos is simply an incarnation of $M$ itself (with its structure sheaf), when no confusion will arise, we reserve the right to denote $\left(\Shi\left(M\right),\mathbf{C}_M^\i\right)$ simply by $M.$ We will let $\sL_{\cinf}$ denote the full subcategory of $\Str_\G$ on those objects of the form $\left(\Shi\left(\RR^n\right),\mathbf{C}_{\RR^n}^\i\right),$ for $n \ge 0$. For each such $n,$ denote the full subcategory spanned by only $\left(\Shi\left(\RR^n\right),\mathbf{C}_{\RR^n}^\i\right)$ by $\sL^n_{\cinf}.$

\begin{dfn}\label{dfn:smoothetenduesi}
$\lbarcinf$ is the $\i$-category of \textbf{smooth $\i$-\'etendues}, and will be denoted by $\Etd$. An \'etale morphism in this $\icat$ will simply be called a \textbf{local diffeomorphism.} For each $n \ge 0,$ $\overline{\sL^n_{\cinf}}$, which is a full subcategory of $\Etd,$ will be denoted by $n\mbox{-}\Etd$. If an object $\cE$ of $\Etd$ is in $n\mbox{-}\Etd,$ it will be called \textbf{$n$-dimensional}.
\end{dfn}

If $\cE$ is an $n$-dimensional smooth $\i$-\'etendue, by definition, it has an \'etale covering family $\left(\varphi_\alpha:\cE_\alpha \to \cE\right)$ with each $\cE_\alpha \simeq \RR^n.$ This covering family by local diffeomorphisms plays the role of a smooth atlas. In other words, $\cE$ is locally diffeomorphic to $\RR^n.$ In particular, if $M$ is a smooth $n$-manifold, then its atlas gives it the unique structure of a smooth $n$-dimensional \'etendue.

\begin{dfn}
A smooth $\i$-\'etendue $\cE$ will be called \textbf{proper} if the underlying geometric morphism of its diagonal $$\cE \to \cE \times \cE$$ is a proper geometric morphism of $\i$-topoi, in the sense of Definition 7.3.1.4 of \cite{htt}.
\end{dfn}

Definition \ref{dfn:smoothetenduesi} is quite similar to the following existing definition of a smooth \'etendue (from Expos\'e iv, Exercice 9.8.2 of \cite{sga4}):

\begin{dfn}\label{dfn:smoothetendues}
A smooth \'etendue is a locally ringed topos $\left(\cE,\O_\cE\right)$ for which there exists an object $E$ of $\cE$ such that $$E \to 1$$ is an epimorphism (a so-called \textbf{well supported} object), and such that $\left(\cE/E,\O_\cE|_E\right)$ is equivalent to a (possibly non-Hausdorff or $2^{nd}$-countable) smooth manifold. Denote the corresponding $\left(2,1\right)$-category by $\Etdd$
\end{dfn}

\begin{prop}\label{prop:etenduesthesame}
There is a canonical equivalence of $\left(2,1\right)$-categories between the subcategory of $\Etd$ spanned by its $1$-localic objects (which are necessarily $1$-truncated), and the $\left(2,1\right)$-category $\Etdd$ of smooth \'etendues as in Definition \ref{dfn:smoothetendues}.
\end{prop}

\begin{proof}
Notice that the functor $$\nu^\i:\mathfrak{Top} \hookrightarrow \T$$ induces a full and faithful functor from the $\left(2,1\right)$-category of locally ringed topoi, to the $\icat$ of locally ringed $\i$-topoi. We will show that this functor restricts to the desired equivalence. Suppose that $\cE$ is a smooth \'etendue, and let $E$ be a well supported object such that $$\cE/E \simeq \Sh\left(M\right)$$ for some (possibly non-Hausdorff) manifold $M$ (we are suppressing the structure sheaf from the notation). For simplicity, assume that $M$ has dimension $n.$ Let $$\left(\varphi_\alpha:U_\alpha \hookrightarrow M\right)$$ be a smooth atlas for $M$, with each $U_\alpha\cong \RR^n.$ Then $$\left(\Shi\left(U_\alpha\right) \to \Shi\left(M\right) \to \nu^\i\left(\cE\right)\right)$$ is an \'etale covering family, so one concludes that $\nu^\i\left(\cE\right)$ is a $1$-localic smooth $\i$-\'etendue. Conversely, suppose that $\cF$ is a $1$-localic smooth $\i$-\'etendue. Let $\left(E_\alpha\right)_\alpha$ be a collection of objects such that $$\coprod\limits_\alpha E_\alpha \to 1$$ is an epimorphism, and such that each $\cE/E_\alpha$ is equivalent to $\RR^{n_\alpha}$ for some $n_\alpha$. By Proposition \ref{prop:2.3.16}, each $E_\alpha$ must be $1$-truncated. It follows that for each $\alpha$ there exists an epimorphism $$p_\alpha:H_\alpha \to E_\alpha$$ from a $0$-truncated object such that $E_\alpha$ may be identified with the stack of torsors for the groupoid object $$H_\alpha \times_{E_\alpha} H_\alpha \rightrightarrows H_\alpha.$$ Each morphism $p_\alpha$ may be identified with a $0$-truncated object in $$\cE/E_\alpha \simeq \Shi\left(\RR^{n_\alpha}\right),$$ i.e. a sheaf $P_\alpha$ on $\RR^{n_\alpha}.$ It follows that $$\cE/H_\alpha \simeq \Shi\left(\RR^{n_\alpha}\right)/P_\alpha,$$ and the later is equivalent to $\Shi\left(L\left(P_\alpha\right)\right),$ where $$L\left(P_\alpha\right) \to \RR^{n_\alpha}$$ is the \'etal\'e space of the sheaf $P_\alpha.$ In particular, each $\cE/H_\alpha$ is equivalent to a (possibly non-Hausdorff) smooth manifold. The object $H:=\coprod\limits_\alpha H_\alpha$ is $0$-truncated, and hence may be identified with a well-supported object of the $1$-topos associated to $\cE.$ It follows that this $1$-topos (with its structure sheaf) is a smooth \'etendue.
\end{proof}

Let $\sD$ denote the full subcategory of $\Etd$ spanned by open submanifolds of $\RR^n,$ for all $n \ge 0.$ It follows immediately from Proposition \ref{prop:existbloss} that $\sD$ is a strong \'etale blossom for $\Etd.$ Denote by $\Mfd$ the category of all smooth manifolds. For convenience, let us assume that each manifold $M$ in $\Mfd$ is $2^{nd}$ countable and Hausdorff so that $\Mfd$ is essentially small by Whitney's embedding theorem. It is a proper full subcategory of $\Etd$ containing $\sD$, hence also a small and locally small strong \'etale blossom.

\begin{lem}
An object $\cE$ of $\Etd$ is $1$-truncated if and only if it is $1$-localic. The same statement also holds in $\Etd^{\et}.$
\end{lem}

\begin{proof}
If $\cE$ is $1$-localic, it is $1$-truncated in $\Etd$ (and hence also in $\Etd^{\et}$) by Lemma \ref{lem:ntruncatedlocalic}. Conversely, by the proof of Theorem \ref{thm:smalltruncatedversion}, $\cE$ is $1$-truncated in $\Etd$ if and only if $i^*\overline{y}\left(\cE\right) \simeq j_!\left(\X\right),$ for $\X$ a $1$-truncated object of $\Shi\left(\Mfd^{\et}\right).$ By Proposition \ref{prop:theta}, it follows that $\cE \simeq \Uni/j_!\left(\X\right),$ and hence by Proposition \ref{prop:2.3.16}, $\cE$ is $1$-localic. If instead, $\cE$ is $1$-truncated in $\Etd^{\et},$ then since $$\Etd^{\et}\simeq \Etd^{\et}/\Uni,$$ the essentially unique \'etale map $\cE \to \Uni$ is $1$-truncated. By Proposition \ref{prop:trivkan}, this corresponds to a $1$-truncated object $E$ of $\Uni,$ such that $\Uni/E \simeq \cE.$ By Theorem \ref{thm:yonequil}, $$\Uni \simeq \Shi\left(\Mfd^{\et}\right),$$ and hence is $1$-localic. The result now follows from Proposition \ref{prop:2.3.16}.
\end{proof}

\begin{cor}\label{cor:1trunete}
The $\left(2,1\right)$-category $\Etdd$ is canonically equivalent to the full subcategory of $\Etd$ spanned by the $1$-truncated objects. Similarly, the $\left(2,1\right)$-category $\Etdd^{\et}$ is equivalent to the full subcategory of $\Etd^{\et}$ spanned by $1$-truncated objects.
\end{cor}

The $1$-category $\Mfd$ carries a canonical Grothendieck topology generated by open covers. An $\i$-sheaf $F$ on $\Mfd$ in the sense of Definition \ref{dfn:sheaveD} is a functor $$\Mfd^{op} \to \iGpd$$ which restricted to each manifold $M$ in $\Mfd$ is an $\i$-sheaf over the topological space $M.$ This definition of $\i$-sheaf is easily seen to be equivalent to an $\i$-sheaf for the open cover topology on $\Mfd.$


\begin{rmk}\label{rmk:hypermanifold}
Notice that every smooth manifold $M$ is of finite covering dimension, hence $\Shi\left(M\right)$ is hypercomplete. Since hypercompleteness is a local condition, it follows that if $\left(\cE,\cinf\left(\cE\right)\right)$ is a smooth $\i$-\'etendue, $\cE$ is hypercomplete.
\end{rmk}

Denote by $$i:\Mfd \hookrightarrow \Etd$$ the full and faithful inclusion. It follows from Theorem \ref{thm:fullyfaithful} that the induced functor $$\Etd \stackrel{\overline{y}}{\longhookrightarrow} \Shi\left(\Etd\right) \stackrel{i^*}{\longrightarrow} \Shi\left(\Mfd\right)$$ is full and faithful.

\begin{dfn}
An $\i$-sheaf $\X$ on $\Mfd$ is called an \textbf{\'etale differentiable $\i$-stack} if it is in the essential image of $i^* \circ \overline{y},$ i.e. if it is the functor of points of a smooth $\i$-\'etendue. Denote the corresponding $\icat$ by $\Etds.$ Similarly, $$\X \simeq i^* \overline{y}\left(\cE\right)$$ is said to be \textbf{$n$-dimensional} if $\cE$ is. Denote the corresponding $\icat$ by $\Etdsn.$
\end{dfn}

\begin{rmk}\label{rmk:etenduesandstacks}
It follows immediately that there are canonical equivalences of $\i$-categories
$$\Etd \simeq \Etds$$
and
$$n\mbox{-}\Etd \simeq \Etdsn.$$
\end{rmk}

There is a potential danger of terminology since there is already a well-established definition of an \'etale differentiable stack, namely:

\begin{dfn}\label{dfn:etalstackmfd}
An \textbf{\'etale Lie groupoid} is a groupoid object $\G$ in the category $\widetilde{\Mfd}^{\et}$ of possibly non-Hausdorff or $2^{nd}$-countable smooth manifolds and local diffeomorphisms between them. Such a groupoid object induces a presheaf of groupoids on $\Mfd,$ by assigning each manifold $M$ the groupoid $$\Hom\left(M,\G_1\right) \rightrightarrows \Hom\left(M,\G_1\right).$$ The stackification of this presheaf to a stack of groupoids is denoted by $\left[\G\right].$ An \textbf{\'etale differentiable stack} is a stack $\X$ on $\Mfd$ which is equivalent to a stack of the form $\left[\G\right]$ for $\G$ an \'etale Lie groupoid. Denote the corresponding $\left(2,1\right)$-category by $\Etdsd.$
\end{dfn}

\begin{rmk}
For $\G$ an \'etale Lie groupoid, the stack $\left[\G\right]$ may also be described as the stack of principal $\G$-bundles.
\end{rmk}

\begin{rmk}
An equivalent formulation of Definition \ref{dfn:etalstackmfd} is that $\X$ is an \'etale differentiable stack if and only if there exists a representable local diffeomorphism $$p:X \to \X$$ from a manifold $X,$ such that $p$ is also an epimorphism. Being a representable local diffeomorphism means that for any map from a manifold $M \to \X,$ (viewing $M$ as a representable sheaf), the pullback $M\times_\X X$ is equivalent to a (possibly non-Hausdorff or $2^{nd}$-countable) manifold, and the induced map $$M\times_\X X \to M$$ is a local diffeomorphism. Such a map $p$ is called an \textbf{\'etale atlas} of $\X.$ Indeed, given such an atlas, $$X \times_\X X \rightrightarrows X$$ is an \'etale Lie groupoid whose associated stack is equivalent to $\X,$ and conversely, given an \'etale Lie groupoid $\G,$ the canonical map $$\G_0 \to \left[\G_0\right]$$ is easily seen to be an \'etale atlas.
\end{rmk}

\begin{thm}\label{thm:etalstthesame}
An $\i$-sheaf $\X$ on $\Mfd$ is an \'etale differentiable stack if and only if it is an \'etale differentiable $\i$-stack, and is $1$-truncated (i.e. a stack of groupoids).
\end{thm}

\begin{proof}
In Theorem 41 of \cite{Dorette}, Pronk constructs a functor $$S:\Etdd \to \Etdsd$$ which she shows to be an equivalence. This functor is easily seen to be nothing but the restriction of functor $$i^* \circ \overline{y}:\Etd \to \Shi\left(\Mfd\right)$$ to $\Etdd$ along $\nu^\i.$ The above functor induces an equivalence between $1$-truncated objects, hence it follows from Corollary \ref{cor:1trunete} that $\Etdsd$ is equivalent to the full subcategory of $\Etds$ on its $1$-truncated objects.
\end{proof}

\begin{rmk}\label{rmk:etalespet}\textbf{An \'Etal\'e Space Construction for $\i$-Sheaves}:\\
Suppose that $M$ is a manifold and $F$ is an $\i$-sheaf on $M.$ Then since $$\sl\left(\Etd\right)=\Etd,$$ $\Shi\left(M\right)/F$ has the canonical structure of a smooth $\i$-\'etendue, and it comes equipped with a local diffeomorphism (\'etale map) $$\pi_F:\Shi\left(M\right)/F \to \Shi\left(M\right),$$ which should be regarded as a local diffeomorphism with codomain $M.$ Let $U$ be an open subset of $M.$ Consider the induced \'etale morphism $$i_U:\Shi\left(U\right) \to \Shi\left(M\right)$$ corresponding to the open inclusion $U \hookrightarrow M.$ The $\i$-groupoid of sections of $\pi_F$ is equivalent to the $\i$-groupoid $$\Hom_{\Etd^{\et}/M}\left(i_U,\pi_F\right).$$ Under the equivalence $$\Etd^{\et}/M \simeq \Shi\left(M\right),$$ this $\i$-groupoid is equivalent to $$\Hom_{\Shi\left(M\right)}\left(y\left(U\right),F\right) \simeq F\left(U\right).$$ In summary, $$\pi_F:\Shi\left(M\right)/F \to \Shi\left(M\right)$$ should be regarded as the ``\'etal\'e space'' of the higher sheaf $F.$ By Remark \ref{rmk:etenduesandstacks}, $\pi_F$ may be regarded as a local diffeomorphism into $M$ with domain an \'etale differentiable $\i$-stack. When $F$ is $1$-truncated, i.e. when $F$ is a stack of groupoids, it follows from Theorem \ref{thm:etalstthesame} that $\pi_F$ must agree with the \emph{\'etale realization} of $F$ in the sense of \cite{etalspme}.

Suppose that $\cE$ is an arbitrary smooth $\i$-\'etendue. The $\i$-topos $\cE$, by Remark \ref{rmk:immm}, is canonically equivalent to the $\i$-topos of sheaves over $\Mfd^{\et}/\cE.$ The $\i$-category $\Mfd^{\et}/\cE$ has as its objects local diffeomorphisms $M \to \cE,$ with $M$ a manifold, and its morphisms are (homotopy coherent) commutative triangles. (If $\cE$ is $1$-localic (so we may identify it with an \'etale differentiable stack), then $\Mfd^{\et}/\cE$ recovers the site described in the last Remark of Section 3 of \cite{etalspme}.) There is a canonical structure of a Grothendieck site on $\Mfd^{\et}/\cE,$ where the covering sieves are generated by open covers of manifolds, and $$\cE \simeq \Shi\left(\Mfd^{\et}/\cE\right).$$ By an \textbf{$\i$-sheaf $F$ on $\cE$}, we mean an $\i$-sheaf over $\Mfd^{\et}/\cE.$ If $F$ is such an $\i$-sheaf, it can be identified canonically with an object $E$ of $\cE.$ Denote the associated local diffeomorphism by $$\pi_F:\cE/E \to \cE.$$ By an analogous argument as in the manifold case, if $f:M \to \cE$ is a local diffeomorphism, the $\i$-groupoid of sections of $\pi_F$ over $f$ is canonically equivalent to the $\i$-groupoid $F\left(f\right).$ Hence $\pi_F$ plays the role of the \'etal\'e space of the $\i$-sheaf $F.$ When $F$ is $1$-truncated, it follows from Theorem \ref{thm:etalstthesame} that $\pi_F$ must agree with the \emph{\'etale realization} of $F$ in the sense of \cite{etalspme}.
\end{rmk}

\begin{dfn}
An $\i$-sheaf $\X$ on $\Mfd$ is called a \textbf{differentiable $\i$-orbifold} if $\X$ is the functor of points of a proper smooth $\i$-\'etendue $\cE.$
\end{dfn}

\begin{rmk}
If $\X$ is a differentiable $\i$-orbifold, and $\X$ is $1$-truncated, it follows immediately from Corollary \ref{cor:1trunete} and Theorem 4.1 of \cite{sheavesonorbifolds} that $\X$ is an orbifold.
\end{rmk}

\begin{rmk}
A general \'etale differentiable $\i$-stack (or smooth $\i$-\'etendue) should be regarded as a generalized higher orbifold. The principal difference is that for any geometric point $p:* \to \X$ in an $\i$-orbifold, its $\i$-group of automorphisms is finite, whereas for an arbitrary \'etale stack, a point's automorphism $\i$-groups are constrained only to be discrete. This is in contrast to a general (Artin-type) differentiable $\i$-stack, where an arbitrary point's automorphisms can form a Lie $\i$-group of positive dimension.
\end{rmk}

Denote by $\widetilde{\Mfd}$ the subcategory of $\Etd$ spanned by all smooth manifolds, where we no longer assume a smooth manifold must be Hausdorff or $2^{nd}$-countable.  Notice that $\widetilde{\Mfd}^{\et}$ is canonically equivalent to the category of such smooth manifolds and their local diffeomorphisms.

\begin{dfn}
An \textbf{\'etale simplicial manifold} is a functor $$\G_\bullet:\Delta^{op} \to \widetilde{\Mfd}^{\et}.$$
\end{dfn}
Notice that since any manifold $M$ in $\widetilde{\Mfd}$ is faithfully represented by its induced sheaf on $\Mfd,$ an \'etale simplicial manifold $\G_\bullet$ yields a simplicial $\i$-sheaf
$$\tilde \G_\bullet:\Delta^{op} \to \Shi\left(\Mfd\right).$$ Denote by $$\left[\G_\bullet\right]:=\colim \tilde \G_\bullet.$$

\begin{rmk}
If $\G_\bullet$ is the nerve of an \'etale Lie groupoid $\G$, then $\left[\G_\bullet\right]$ is the stack (with values in groupoids) of principal $\G$-bundles- i.e. the \'etale differentiable stack associated to $\G.$
\end{rmk}

\begin{rmk}
One may replace $\widetilde{\Mfd}$ with the category $\Mfd'$ consisting of manifolds which can be obtained by taking arbitrary coproducts of manifolds in $\Mfd$ without changing anything essential. In other words, an arbitrary manifold $M$ in $\Mfd'$ is of the form $$M \cong \coprod\limits_\alpha M_\alpha,$$ with each $M_\alpha$ in $\Mfd.$ This will of course exclude the enriched-nerve of a non-Hausdorff \'etale Lie groupoid $N\left(\G\right)_\bullet$ as an example of an \'etale simplicial manifold, however, any such nerve can be refined to a \'etale simplicial manifold $\G'_\bullet$ such that $$\left[\G\right] \simeq \left[ \G'_\bullet\right].$$ In fact, $N\left(\G\right)_\bullet$ may even be arranged to be a Lie $2$-groupoid.
\end{rmk}

\begin{prop}
An $\i$-sheaf $\X$ on $\Mfd$ is an \'etale differentiable $\i$-stack if and only if it is equivalent to one of the form $\left[\G_\bullet\right]$ with $\G_\bullet$ an \'etale simplicial manifold.
\end{prop}

\begin{proof}
Suppose that $\G_\bullet$ is an \'etale simplicial manifold. Then, since $\Etd^{\et}$ is cocomplete (Lemma \ref{lem:holmol}), the colimit $$\cE:= \colim \left(\Delta^{op} \stackrel{\G_\bullet}{\longlongrightarrow} \widetilde{\Mfd}^{\et} \stackrel{i'^{\et}}{\longlongrightarrow} \Etd^{\et}\right)$$ exists, where $i'$ is the canonical inclusion. By Proposition \ref{prop:bothfuncolims}, it follows that $i^*\overline{y}\left(\cE\right) \simeq \left[\G_\bullet\right].$ Conversely, suppose that $\cE$ is a smooth $\i$-\'etendue. Since $\Mfd$ is a strong \'etale blossom, $\cosl\left(\Mfd\right)=\Etd.$ It follows from the definition of $\cosl\left(\Mfd\right)$ that one can construct a hypercover $$F_\bullet:\Delta^{op} \to \Etd^{\et}/\cE\simeq \cE$$ such that each $F\left(n\right):\G\left(n\right) \to \cE,$ has $\G\left(n\right)$ in $\widetilde{\Mfd}.$ By Remark \ref{rmk:hypermanifold}, $\cE$ is hypercomplete, so this hypercover converges. Hence if $\G_\bullet:\Delta^{op} \to \widetilde{\Mfd}^{\et}$ is the \'etale simplicial manifold induced by $F_\bullet,$ one has that $\cE \simeq \colim \left(\G_0 \leftleftarrows \G_1 \lllarrows \G_2 \ldots\right)$ (with the colimit taken in $\Etd$). It follows from Proposition \ref{prop:bothfuncolims} that $i^* \overline{y}\left(\cE\right)$ is equivalent to $\left[\G_\bullet\right].$
\end{proof}

\begin{rmk}
Morally, an \'etale simplicial manifold should be thought of as an \'etale Lie $\i$-groupoid. For this to really be the case, one would want to impose a Kan condition on the simplicial manifold (see e.g. \cite{zhu} and \cite{jesse}). We expect that every \'etale simplicial manifold can be rectified to one satisfying a Kan condition which induces the same functor of points, but we will not pursue this here.
\end{rmk}

The following theorem is a special case of Theorem \ref{thm:yonequil} and Proposition \ref{prop:theta}:

\begin{thm}\label{thm:fields}
There is an adjoint equivalence
$$\xymatrix@C=2.5cm{\Etd^{\et} \ar@<-0.5ex>[r]_-{\left(i^{\et}\right)^* \circ y^{\et}} & \Shi\left(\Mfd^{\et}\right) \ar@<-0.5ex>[l]_-{\Theta}}$$
between the $\i$-category of smooth $\i$-\'etendues and local diffeomorphisms and the $\i$-category of $\i$-sheaves on the site of smooth manifolds and local diffeomorphisms. Similarly, there is an adjoint equivalence
$$\xymatrix@C=2.5cm{n\mbox{-}\Etd^{\et} \ar@<-0.5ex>[r]_-{\left(i_n^{\et}\right)^* \circ y_n^{\et}} & \Shi\left(n\mbox{-}\Mfd^{\et}\right) \ar@<-0.5ex>[l]_-{\Theta_n}}$$ between the $\i$-category of smooth $n$-dimensional $\i$-\'etendues and local diffeomorphisms and the $\i$-category of $\i$-sheaves on the site of smooth $n$-manifolds and local diffeomorphisms.
\end{thm}

\begin{rmk}
In particular, since there is a canonical equivalence of $\i$-categories between smooth $\i$-\'etendues and \'etale differentiable $\i$-stacks, one deduces there is an equivalence between the $\i$-category of ($n$-dimensional) \'etale differentiable $\i$-stacks and local diffeomorphisms and the $\i$-category of $\i$-sheaves on the site of smooth ($n$-)manifolds and local diffeomorphisms.
\end{rmk}

\begin{rmk}
The $\i$-category $\Shi\left(n\mbox{-}\Mfd^{\et}\right)$ is the same as the $\i$-category of \emph{classical fields for an $n$-dimensional field theory} as defined by Freed and Teleman in \cite{field}. Therefore, one may interpret Theorem \ref{thm:fields} as saying that $n$-dimensional smooth $\i$-\'etendues (generalized higher orbifolds) are geometric incarnations of classical fields for an $n$-dimensional field theory. In particular, this gives a way of assigning to a classical field a weak homotopy type. We work out a rather simple formula for this homotopy type in \cite{homtme}.
\end{rmk}

\begin{rmk}
Let $F$ be any $\i$-sheaf on the site of manifolds and local diffeomorphisms. It follows from Remark \ref{rmk:theta} that the smooth $\i$-\'etendue $\Theta\left(F\right)$ satisfies the universal property that for any smooth manifold $M,$ the $\i$-groupoid $F\left(M\right)$ is canonically equivalent to the $\i$-groupoid of local diffeomorphisms $M \to \Theta\left(F\right).$
\end{rmk}

Denote by $j:\Mfd^{\et} \to \Mfd$ the canonical functor from the category of smooth manifolds and local diffeomorphisms to the category of smooth manifolds and all smooth maps. As in Definition \ref{dfn:prolongation}, denote the left adjoint to the restriction functor $$j^*:\Shi\left(\Mfd\right) \to \Shi\left(\Mfd^{\et}\right)$$ by $j_!$ and call it the \'etale prolongation functor.

The following theorem is a special case of Theorem \ref{thm:classificationsmall}:

\begin{thm}
An $\i$-sheaf $\X$ on $\Mfd$ is an \'etale differentiable $\i$-stack if and only if it is in the essential image of the \'etale prolongation functor $j_!$.
\end{thm}

Notice that the classifying $\i$-topos for local rings is $1$-localic. By abuse of notation, consider the induced restriction functor $$j^*:\St\left(\Mfd\right) \to \St\left(\Mfd^{\et}\right)$$ between stacks of groupoids, which has a left adjoint, which we will also denote by $j_!.$ The following is a special case of Theorem \ref{thm:smalltruncatedversion}:

\begin{thm}\label{thm:prolst}
A stack in $\St\left(\Mfd\right)$ is an \'etale differentiable stack if and only if it is in the essential image of $j_!.$
\end{thm}

\begin{rmk}
Theorem \ref{thm:prolst} appears in \cite{prol} as Corollary 3.4.
\end{rmk}

\begin{dfn}\label{dfn:haf}\cite{Haefliger}
Let $M$ be a smooth manifold. Consider the presheaf $$\mathit{Emb}:\mathcal{O}\left(X\right)^{op} \to \Set,$$ which assigns an open subset $U$ the set of smooth open embeddings of $U$ into $M$. Denote the \'etal\'e space of the associated sheaf by $$s:\Ha\left(M\right)_1 \to M.$$ The space $\Ha\left(M\right)_1$ has an induced structure of a (non-Hausdorff) smooth manifold unique with the property that it makes the map $s$ a local diffeomorphism. The fiber of $s$ over a point $x$ is the same as the stalk at $x$ of the sheafification of $\mathit{Emb};$ concretely it is the set of germs at $x$ of locally defined diffeomorphisms of $M.$ If $germ_x\left(f\right)$ is one such germ, the element $f\left(x\right) \in M$ is well-defined. This assignment produces a smooth map $$t:\Ha\left(M\right)_1 \to M$$ which is also a local diffeomorphism. Composition of such germs induces the natural structure of an \'etale Lie groupoid $\Ha\left(M\right)$ with objects $M$, called the \textbf{Haefliger groupoid} of $M$. The Haefliger groupoid of $\RR^n,$ $\Ha\left(\RR^n\right),$ will also be called the \textbf{$n^{th}$ Haefliger groupoid}. The \textbf{$n^{th}$ Haefliger stack} is the \'etale differentiable stack $\HA^n:=\left[\Ha\left(\RR^n\right)\right],$ and the \textbf{Haefliger stack} is the disjoint union $$\coprod\limits_{n \ge 0} \HA^n.$$
\end{dfn}

\begin{rmk}
By Corollary 3.1 of \cite{prol}, if $M$ is any $n$-manifold, one has an equivalence of \'etale differentiable stacks $$\left[\Ha\left(M\right)\right] \simeq \HA^n.$$
\end{rmk}


\begin{thm}
The $n^{th}$ Haefliger stack is the terminal object in $\Etdsn^{\et},$ and the Haefliger stack is the terminal object in $\Etds^{\et}.$
\end{thm}

\begin{proof}
By Theorem 3.3 of \cite{prol}, $\HA^n$ and $\HA$ are terminal in $n\mbox{-}\Etdsd^{\et}$ and $\Etdsd^{\et}$ respectively. However, these $\left(2,1\right)$-categories are equivalent to the full subcategories of $1$-truncated objects of $n\mbox{-}\Etds^{\et}$ $\Etds^{\et}$ respectively by Corollary \ref{cor:1trunete}, Remark \ref{rmk:etenduesandstacks}, and Theorem \ref{thm:etalstthesame}. Notice that $n\mbox{-}\Etds^{\et}$ $\Etds^{\et}$ are $\i$-topoi by Theorem \ref{thm:fields} and Remark \ref{rmk:etenduesandstacks}, hence in particular, presentable $\i$-categories. The result now follows since the full subcategory of $1$-truncated objects in a presentable $\i$-category is a reflective subcategory by Proposition 5.5.6.18 of \cite{htt}.
\end{proof}

If $\G$ is an \'etale Lie groupoid, there is a canonical homomorphism Lie groupoids $$i_\G:\G \to \Ha\left(\G_0\right).$$ On objects, $i_\G$ is the identity. Given an arrow $g,$ one can find a neighborhood $U$ of $g$ in $\G_1$ over which the source and target maps $s$ and $t$ of the groupoid restrict to smooth open embeddings. The germ at $s\left(g\right)$ of the local diffeomorphism $$t\circ s^{-1}|_U$$ with source $s\left(U\right)$ does not depend on the choice of the neighborhood $U,$ and this assignment assembles into smooth map $$\left(i_\G\right)_1:\G_1 \to \Ha\left(\G_0\right)_1$$ (which is even a local diffeomorphism). It is easy to check that this is indeed a homomorphism of Lie groupoids. The image of $i_\G$ is an open subgroupoid $\mathbf{Eff}\left(\G\right)$ of $\Ha\left(\G_0\right).$

\begin{dfn}
An \'etale Lie groupoid $\G$ is \textbf{effective} if the induced map $$\G \to \mathbf{Eff}\left(\G\right)$$ is an isomorphism. An \'etale differentiable stack $\X$ is \textbf{effective} if $\X \simeq \left[\G\right]$ for $\G$ an effective \'etale Lie groupoid. Similarly, a smooth \'etendue is called \textbf{effective} if its associated \'etale differentiable stack is. (See Section 2 of \cite{prol} for a more in depth discussion about effectivity.)
\end{dfn}

The following proposition is an immediate consequence of Theorem 4.1 of \cite{prol}:

\begin{prop}\label{prop:eff}
An \'etale differentiable stack $\X$ is effective if and only if the unique map $\X \to \HA$ is $0$-truncated. Equivalently, $\overline{y}^{\et}\left(\Theta\left(\X\right)\right)$ is a sheaf.
\end{prop}

\begin{dfn}\label{dfn:gerbe}
Let $\cE$ be an $\i$-topos. An object $\g$ in $\cE$ is an \textbf{$\i$-gerbe} if $\pi_0\left(\g\right)$ is terminal, where $\pi_0$ is the left adjoint to the inclusion of the full subcategory of $0$-truncated objects.
\end{dfn}

\begin{rmk}
By Proposition 4.1 of \cite{prol}, if $E$ is a $1$-truncated object of $\cE$, then $E$ is an $\i$-gerbe if and only if $E$ is a gerbe in the classical sense of Giraud in \cite{Giraud} (in the $1$-localic reflection of $\cE$). By a similar proof, it can be shown that if $E$ is $2$-truncated, then it is an $\i$-gerbe if and only if it is a $2$-gerbe in the sense of Breen in \cite{breen}. It should be noted that what Lurie calls an $n$-gerbe in Definition 7.2.2.20 is a particular case of an $\i$-gerbe, but it is much more restrictive; a more descriptive name for what Lurie calls an $n$-gerbe would be an Eilenberg-Maclane object, as $\X$ is an $n$-gerbe if and only if $\pi_n\left(\X\right)$ is its only non-trivial homotopy sheaf. For example, a $2$-gerbe in the sense of Lurie is a stack $\X$ of $2$-groupoids such that $\pi_0\left(\X\right)$ and $\pi_1\left(\X\right)$ are trivial, whereas a $2$-gerbe in the sense of Breen need only have $\pi_0\left(\X\right)$ trivial.
\end{rmk}

\begin{dfn}
Let $\left(\cE,\O_\cE\right)$ be a smooth $\i$-\'etendue. An \textbf{$\i$-gerbe on $\cE$} is an $\i$-gerbe in the underlying $\i$-topos $\cE.$ Given $\g$ an $\i$-gerbe on $\cE,$ the \textbf{\'etal\'e space of $\g$} is the smooth $\i$-\'etendue $\cE/\g.$
\end{dfn}

\begin{rmk}
By Remark \ref{rmk:etalespet}, an $\i$-gerbe on $\cE$ may be identified with an $\i$-gerbe on the site $\Mfd^{\et}/\cE.$
\end{rmk}

\begin{thm}
If $\cE$ is a smooth $\i$-\'etendue, $\cE$ is equivalent to the \'etal\'e space of an $\i$-gerbe on an effective smooth \'etendue.
\end{thm}

\begin{proof}
Let $\cE$ be a smooth $\i$-\'etendue. Let $\cF:=\Theta\left(\pi_0\left(\left(i^{\et}\right)^*\overline{y}^{\et}\left(\cE\right)\right)\right),$ with the notation as in Theorem \ref{thm:fields}. Then $\cF$ is $1$-localic and effective by Proposition \ref{prop:eff}. The canonical map $$\left(i^{\et}\right)^*\overline{y}^{\et}\left(\cE\right) \to \pi_0\left(\left(i^{\et}\right)^*\overline{y}^{\et}\left(\cE\right)\right)$$ is $1$-connective. It follows that the associated object $F$ in $$\Shi\left(\Mfd^{\et}\right)/\pi_0\left(\left(i^{\et}\right)^*\overline{y}^{\et}\left(\cE\right)\right)$$ is an $\i$-gerbe. The result now follows.
\end{proof}

\begin{rmk}
One could equally easily work with $\mathbf{C}^k$-manifolds for any $k,$ including $k=\omega.$ Also, one can just as easily work with complex manifolds or even supermanifolds by passing to supercommutative $\RR$-algebras. All of the definitions and results in this section easily extend to these settings.
\end{rmk}

\begin{rmk}
As explained in Section 10.1 of \cite{spivak} and Section 4.4 of \cite{dag}, the $\i$-category of (quasi-smooth) derived manifolds as defined in loc. cit should embed into an appropriate category of structured topoi of the form $\Str_\G.$ It should have a small strong \'etale blossom consisting of all derived manifolds which can be obtained as the derived zero set of a smooth map $$f:\RR^n \to \RR^k$$ for some $n$ and $k,$ and hence the definitions and results in this section should easily extend to derived manifolds.
\end{rmk}

\begin{rmk}
One can analogously define topological $\i$-\'etendues and \'etale topological $\i$-stacks, by working with the trivial geometric structure with base $\i$-topos $\iGpd,$ and letting the role of smooth manifolds be played by arbitrary topological spaces. Some, but not all of the results of this section extend to this setting. In particular, there exists topological spaces whose $\i$-topoi of sheaves are not hypercomplete, so in principal, there could exist \'etale topological $\i$-stacks which do not arise from \'etale simplicial spaces. Moreover, there certainly does not exist a small strong \'etale blossom in this setting so there is no analogue to Theorem \ref{thm:fields}, and one can only appeal to Theorem \ref{thm:classificationlarge}. The rest of the results easily generalize. We leave the details to the reader as an easy exercise.
\end{rmk}

\section{Deligne-Mumford Stacks for a Geometry}\label{sec:dmgeom}
In this section, we will interpret the results of Chapter \ref{chap:etendues} in the context of classical Deligne-Mumford stacks, and their higher, derived, and spectral variants as developed in \cite{dag} and \cite{spectral}. In particular, we will arrive at new categorical characterizations of such Deligne-Mumford stacks.

Let $\g$ be a geometry in the sense of \cite{dag}, Definition 1.2.5, and let $\G=\G\left(\g\right)$ be its associated geometric structure (see Example 1.4.4 of op. cit.). To more closely follow the notation introduced in op. cit., we will denote the underlying $\i$-category of the idempotent complete essentially algebraic $\i$-theory (in the sense of Definition \ref{dfn:essalg}) also by $\g.$ Recall from Section \ref{sec:geometry}, there exists a full and faithful spectrum functor $$\mathbf{Spec}_\g:\mathbf{Pro}\left(\g\right) \hookrightarrow \Str_\G,$$ which is left adjoint to the global sections functor
$$\Gamma_\g:\Str_{\G} \to \mathbf{Pro}\left(\g\right).$$ (See Section 2 of \cite{dag} for more detail). Structured $\i$-topoi in the essential image of $\Specg$ are called affine $\g$-schemes. Let $\sL=\Affg$ be the full subcategory of $\Str_\G$ on the affine $\g$-schemes. By Example \ref{ex:affines}, an $\sL$-\'etendue in the sense of Definition \ref{dfn:letendue} is what Lurie calls a $\g$-scheme in Definition 2.3.9 of \cite{dag}. Denote the corresponding $\i$-category $\lbar$ of $\g$-schemes by $\gsch$. By Example \ref{ex:affineblossom}, $\Affg$ is a strong \'etale blossom, so we may represent $\g$-schemes as $\i$-sheaves over $\Affg.$ Denote by $$i:\Affg \hookrightarrow \gsch$$ the full and faithful inclusion, and denote by $$\overline{y}:\gsch \hookrightarrow \Shi\left(\gsch\right)$$ the Yoneda embedding into $\i$-sheaves (in the sense of Definition \ref{dfn:sheaveD}) for $\g$-schemes. By Lemma \ref{lem:sheavesthesame}, the induced functor $$i^*:\Shi\left(\gsch\right) \to \Shi\left(\Affg\right)$$ is an equivalence.

\begin{dfn}\label{dfn:gDM}
A \textbf{Deligne-Mumford} $\g$-stack is an $\i$-sheaf $\X$ on $\Affg$ in the essential image of $i^* \circ \overline{y}$. Denote the corresponding $\i$-category by $\Dmg$
\end{dfn}

\begin{rmk}
By Remark \ref{dfn:sheaveD}, one can equivalently define Deligne-Mumford $\g$-stacks as $\i$-sheaves over $\Pro\left(\g\right).$
\end{rmk}

\begin{rmk}
It follows immediately that there is a canonical equivalence $$\gsch \simeq \Dmg.$$
\end{rmk}

\begin{ex}\textbf{Higher Zariski \'Etale Stacks}:\\
Let $k$ be a commutative ring and let $\g=\g_{Zar}\left(k\right)$ be the Zariski geometry over $k$ as defined in Section 2.5 of \cite{dag}. The underlying $\i$-category is the category of finitely presented commutative $k$-algebras, and $\Pro\left(\g_{Zar}\left(k\right)\right)$ can be identified with the category of affine $k$-schemes. If $\G$ is the associated geometric structure, then $\Str_\G$ is the $\i$-category of $\i$-topoi locally ringed in commutative $k$-algebras. The $\i$-category $\sch{\g_{Zar}\left(k\right)}$ contains the category of classical schemes over $k,$ (viewed as locally ringed spaces, regarded as locally ringed $\i$-topoi), but it also contains much more. The $\i$-category of $\g_{Zar}\left(k\right)$-schemes (or equivalently the $\i$-category of Deligne-Mumford $\g_{Zar}\left(k\right)$-stacks) is equivalent to the $\i$-category of higher categorical versions of what we called Zariski \'etale stacks in \cite{prol}. The classical category of schemes over $k$ can be recovered as the full subcategory of $\sch{\g_{Zar}\left(k\right)}$ on the $0$-localic objects, by Theorem 2.5.16 of \cite{dag}. We will call the $\i$-category of Deligne-Mumford $\g_{Zar}\left(k\right)$-stacks \textbf{Zariski \'etale $\i$-stacks}.
\end{ex}

\begin{ex}\textbf{Higher Deligne-Mumford Stacks}:\\
Let $k$ be a commutative ring and let $\g=\g_{\acute{e}t}\left(k\right)$ be the \'etale geometry over $k$ as defined in Section 2.6 of \cite{dag}. As an $\i$-category, it is the same as $\g_{Zar}\left(k\right).$ If $\G$ is the associated geometric structure, then $\Str_\G$ is the $\i$-category of $\i$-topoi with a structure sheaf of strictly Henselian $k$-algebras. Notice that $\Pro\left(\g_{\acute{e}t}\left(k\right)\right)$ is equivalent to the opposite of the category of commutative $k$-algebras, $\mathbf{Com}_k$. The \'etale spectrum functor $$\mathbf{Spec}_{\acute{e}t}:\mathbf{Com}^{op}_k \hookrightarrow \Str_\G$$ sends a commutative $k$-algebra $A$ to $\Shi\left(A_{\acute{e}t}\right),$ -the incarnation of its small \'etale topos as an $\i$-topos- together with an appropriate structure sheaf. The $\i$-category of Deligne-Mumford stacks for $\g_{\acute{e}t}\left(k\right)$ contains the $\left(2,1\right)$-category of classical Deligne-Mumford stacks over $k$ as a full subcategory. More precisely, by Theorem 2.6.18 of \cite{dag}, it follows that a Deligne-Mumford $\g_{\acute{e}t}\left(k\right)$-stack $\X$ in the sense of Definition \ref{dfn:gDM} is a classical Deligne-Mumford stack over $k$ if and only if $\X \simeq i^* \overline{y}\left(\cE\right),$ for $\cE$ a $1$-localic $\g_{\acute{e}t}\left(k\right)$-scheme.
\end{ex}

\begin{rmk}
The definition of Deligne-Mumford stack used in Theorem 2.6.18 of \cite{dag} is slightly non-standard since it imposes no separation conditions on the diagonals of the stacks in question. The relationship to classical Deligne-Mumford stacks which have separable and quasi-compact diagonals is analogous to the relationship between \'etale differentiable stacks and smooth orbifolds.
\end{rmk}

\begin{thm}\label{thm:DM1trun}
Let $k$ be a commutative ring. An $\i$-sheaf $\X$ on affine $k$-schemes, with respect to the \'etale topology, is a Deligne-Mumford stack if and only if it is a Deligne-Mumford $\g_{\acute{e}t}\left(k\right)$-stack $\X$ in the sense of Definition \ref{dfn:gDM} and $\X$ is $1$-truncated, i.e. it is a stack of groupoids.
\end{thm}

\begin{proof}
Suppose that $\X$ is a classical Deligne-Mumford stack. Then $\X$ is clearly $1$-truncated, and by Theorem 2.6.18 of \cite{dag}, $\X$ is a Deligne-Mumford $\g_{\acute{e}t}\left(k\right)$-stack. Conversely, suppose that $\X$ is a $1$-truncated Deligne-Mumford $\g_{\acute{e}t}\left(k\right)$-stack. Then $\X \simeq i^* \overline{y}\left(\cE\right),$ for $\cE$ some $\g_{\acute{e}t}\left(k\right)$-scheme. By Example \ref{ex:affineblossom}, $\cE$ is in $\overline{\mathbf{Aff}^\cE_{\g_{\acute{e}t}\left(k\right)}}.$ Denote by $\Uni^\cE$ the universal \'etendue associated with the strong \'etale blossom $\mathbf{Aff}^\cE_{\g_{\acute{e}t}\left(k\right)},$ i.e. the terminal object in $\overline{\mathbf{Aff}^\cE}^{\et}_{\g_{\acute{e}t}\left(k\right)},$ which exists by Remark \ref{rmk:esssmallad} and Corollary \ref{cor:trex1}. By Theorem \ref{thm:yonequil}, $$\Uni^\cE \simeq \Shi\left({\mathbf{Aff}^\cE_{\g_{\acute{e}t}\left(k\right)}}^{\et}\right),$$ so $\Uni^\cE$ is $1$-localic. Since $\X$ is $1$-truncated, the essentially unique \'etale map $$\cE \to \Uni^\cE$$ is $1$-truncated, and corresponds to a unique $1$-truncated object $V$ in $\Uni^\cE$ such that $\Uni^\cE/V \simeq \cE.$ Hence $\cE$ is $1$-localic by Proposition \ref{prop:2.3.16}, so we are done by Theorem 2.6.18 of \cite{dag}.
\end{proof}

In light of the above theorem, we make the following definition:

\begin{dfn}
A \textbf{Deligne-Mumford $\i$-stack} over $k$ is a Deligne-Mumford $\g_{\acute{e}t}\left(k\right)$-stack.
\end{dfn}

\begin{ex}\textbf{Derived Zariski \'Etale Stacks}:\\
Let $k$ be a commutative ring. Let $\g=\g^{der}_{Zar}\left(k\right)$ be the derived Zariski geometry as defined in Section 4.2 of \cite{dag}. As an $\i$-category it is the $\i$-category of compact objects in the $\i$-category of simplicial commutative $k$-algebras, $\mathbf{SCR}_k$. In particular, one has a canonical identification $$\mathbf{Ind}\left(\g^{der}_{Zar}\left(k\right)\right)\simeq \mathbf{SCR}_k.$$
Let $\G$ denote the associated geometric structure. The derived Zariski spectrum functor $$\Spec^{der}:\Pro\left(\g^{der}_{Zar}\left(k\right)\right) \hookrightarrow \Str_\G,$$ under the identification $$\mathbf{SCR}_k \simeq \mathbf{Ind}\left(\g^{der}_{Zar}\left(k\right)\right) \simeq \Pro\left(\g^{der}_{Zar}\left(k\right)\right)^{op},$$ sends a simplicial commutative $k$-algebra $\mathbf{A}$ to the classical spectrum of the commutative $k$-algebra $\pi_0\left(\mathbf{A}\right)$ (regarded as an $\i$-topos) but with an appropriately simplicially enhanced structure sheaf. The $\i$-category of $\g^{der}_{Zar}\left(k\right)$-schemes consists of derived versions of Zariski \'etale $\i$-stacks. In particular, the full subcategory on the $0$-localic objects is equivalent to the $\i$-category of \textbf{derived schemes over $k$}.
\end{ex}

\begin{ex}\label{ex:ddms}\textbf{Derived Deligne-Mumford Stacks}:\\
Let $k$ be a commutative ring. Let $\g=\g^{der}_{\acute{e}t}\left(k\right)$ be the derived \'etale geometry over $k$ as in Definition 4.3.13 of \cite{dag}. As an $\i$-category, $\g^{der}_{\acute{e}t}\left(k\right)$ is the same as $\g^{der}_{Zar}\left(k\right)$. 
Recall that a morphism $\varphi:\mathbf{A} \to \mathbf{B}$ of simplicial commutative $k$-algebras is \textbf{\'etale} if
\begin{itemize}
\item [i)] the induced morphism $\pi_0\left(\mathbf{A}\right) \to \pi_0\left(\mathbf{B}\right)$ is \'etale and
\item[ii)] for all $i>0,$ the induced morphism of abelian groups $$\pi_i\left(\mathbf{A}\right)\otimes_{\pi_0\left(\mathbf{A}\right)} \pi_0\left(\mathbf{B}\right) \to \pi_i\left(\mathbf{B}\right)$$ is an isomorphism.
\end{itemize}

In the geometry $\g=\g^{der}_{Zar}\left(k\right),$ the admissible morphisms are precisely the \'etale ones. Let $\G$ denote the associated geometric structure. The \'etale spectrum functor
$$\Spec_{\acute{e}t}:\mathbf{SCR}^{op}_k \hookrightarrow \Str_\G$$ sends a simplicial commutative $k$-algebra $\mathbf{A},$ to $$\Shi\left(\pi_0\left(\mathbf{A}\right)_{\acute{e}t}\right),$$ with an appropriate structure sheaf of strictly Henselian simplicial $k$-algebras. The $\i$-category of $\g^{der}_{\acute{e}t}\left(k\right)$-schemes is what Lurie defines to be the $\i$-category of \emph{derived Deligne-Mumford stacks over $k$}. In particular, the $\i$-category of Derived Deligne-Mumford stacks over $k$ is canonically equivalent to the $\i$-category of Deligne-Mumford $\g^{der}_{\acute{e}t}\left(k\right)$-stacks in the sense of Definition \ref{dfn:gDM}.
\begin{itemize}
\item In order to ease terminology, we will refer to $\g^{der}_{\acute{e}t}\left(k\right)$-schemes as \textbf{\'etale derived schemes}, and reserve the terminology \textbf{derived Deligne-Mumford stacks} for their functors of points. Both $\i$-categories of course are canonically equivalent.
\end{itemize}
\end{ex}

\begin{ex}\textbf{Spectral Deligne Mumford Stacks}:\\
Let $\bE$ be an $\Ei$-ring spectrum. Let $\g=\g^{nSp}_{\acute{e}t}\left(\bE\right)$ be the (non-connective) \'etale spectral geometry over $\bE$ as in Definition 8.11 of \cite{spectral}. As an $\i$-category, $\g^{nSp}_{\acute{e}t}\left(\bE\right)$ is the $\i$-category of compact $\bE$-algebras. The admissible morphisms are precisely the \'etale morphisms of $\Ei$-rings, which are defined in the same way as \'etale maps of simplicial rings as in Example \ref{ex:ddms}. Let $\G$ denote the associated geometric structure. The \'etale spectrum functor
$$\Spec_{\acute{e}t}:\bE\mbox{-}\mathbf{Alg}^{op} \hookrightarrow \Str_\G$$ sends an $\bE$-algebra $\mathbf{F},$ to $$\Shi\left(\pi_0\left(\mathbf{F}\right)_{\acute{e}t}\right),$$ with an appropriate structure sheaf of strictly Henselian $\bE$-algebras. The $\i$-category of $\g^{nSp}_{\acute{e}t}\left(\bE\right)$-schemes is what Lurie defines to be the $\i$-category of \emph{non-connective spectral Deligne-Mumford stacks over $\bE$}. In particular, the $\i$-category of non-connective spectral Deligne-Mumford stacks over $\bE$ is canonically equivalent to the $\i$-category of Deligne-Mumford $\g^{nSp}_{\acute{e}t}\left(\bE\right)$-stacks in the sense of Definition \ref{dfn:gDM}.

\begin{itemize}
\item In order to ease terminology, we will refer to $\g^{nSp}_{\acute{e}t}\left(\bE\right)$-schemes as \textbf{non-connective \'etale spectral schemes}, and reserve the terminology \textbf{non-connective spectral Deligne-Mumford stacks} for their functors of points. Both $\i$-categories of course are canonically equivalent.
\end{itemize}

If $\bE$ is connective, one can also consider the connective \'etale spectral geometry over $\bE$, $\g^{Sp}_{\acute{e}t}\left(\bE\right),$ see Definition 8.14 of \cite{spectral}. As an $\i$-category, it is the $\i$-category of compact connective $\bE$-algebras. The affine schemes are the full subcategory of the affine schemes for $\g^{nSp}_{\acute{e}t}\left(\bE\right)$ on those which are the \'etale spectrum of a connective $\bE$-algebra. The $\i$-category of $\g^{Sp}_{\acute{e}t}\left(\bE\right)$-schemes is what Lurie calls the $\i$-category of \emph{spectral Deligne-Mumford stacks over $\bE$}.

\begin{itemize}
\item We will refer to $\g^{Sp}_{\acute{e}t}\left(\bE\right)$-schemes as \textbf{\'etale spectral schemes}, and reserve the terminology \textbf{spectral Deligne-Mumford stacks} for their functors of points. Both $\i$-categories of course are canonically equivalent.
\end{itemize}
In particular, the $\i$-category of spectral Deligne-Mumford stacks over $\bE$ is canonically equivalent to the $\i$-category of Deligne-Mumford $\g^{Sp}_{\acute{e}t}\left(\bE\right)$-stacks.
\end{ex}

\begin{rmk}
These examples are not a complete list. In particular, there are also connective and non-connective Zariski spectral geometries. See Definition 2.10 of \cite{spectral}.
\end{rmk}

\begin{dfn}\label{dfn:admisclosed}
Let $\C$ be a full subcategory of $\Pro\left(\g\right).$ $\C$ is \textbf{admissibly closed} if for every object $C$ in $\C,$ if $f:D \to C$ is an admissible morphism, then $D$ is in $\C.$ Let $\sL$ be a full subcategory of $\Affg.$ $\sL$ is \textbf{admissibly closed} if it is the image of an admissibly closed subcategory of $\Pro\left(\g\right)$ under $\Spec_\g.$
\end{dfn}

\begin{prop}\label{prop:admicls}
If $\sL$ is an admissibly closed subcategory of $\Affg,$ then $\sL$ is a strong \'etale blossom.
\end{prop}

\begin{proof}
If $\Spec_\g\left(A\right)$ is any object of $\sL,$ $\Pro\left(\g\right)^{ad}/A$ is a site for its underlying $\i$-topos. Notice that the spectrum functor $\Spec_\g$ identifies $\Pro\left(\g\right)^{ad}/A$ with a full subcategory of $\sL^{\et}/\Spec_\g\left(A\right).$ The result now follows.
\end{proof}

\begin{dfn}
Let $\C$ be a full subcategory of $\Pro\left(\g\right).$ The \textbf{admissible closure} of $\C,$ denoted $\widetilde{\C},$ is the smallest admissibly closed subcategory of $\Pro\left(\g\right)$ containing $\C.$ Similarly, if $\sL$ is any subcategory of $\Affg,$ the \textbf{admissible closure} of $\sL,$ denoted $\widetilde{\sL},$ is the smallest admissibly closed subcategory of $\Affg$ containing $\sL.$
\end{dfn}

\begin{rmk}
If $\sL$ is a full subcategory of $\Affg,$ then $\widetilde{\sL} \subset \sl\left(\sL\right),$ so it follows from Remark \ref{rmk:slclsd} that $$\overline{\sL}=\overline{\widetilde{\sL}\mspace{10mu}}=\cosl\left(\widetilde{\sL}\mspace{10mu}\right).$$
\end{rmk}

\begin{rmk}\label{rmk:esssmallad}
By Remark 2.2.7 of \cite{dag}, for all $A$ in $\Pro\left(\g\right)$, the $\i$-category $\Pro\left(\g\right)^{ad}/A$ is essentially small. It follows that if $\C$ is an essentially small subcategory of $\Pro\left(\g\right),$ its admissible closure $\widetilde{\C}$ is essentially small. Similarly, if $\sL$ is an essentially small subcategory of $\Affg,$ its admissible closure $\widetilde{\sL}$ is essentially small.
\end{rmk}

\begin{ex}
If $\cE$ is a $\g$-scheme, then $\mathbf{Aff}^\cE_\g,$ as defined in Example \ref{ex:affines}, is an essentially small admissibly closed subcategory of $\Affg.$
\end{ex}

\begin{ex}
Let $\C$ be a full subcategory of the category of commutative $k$-algebras. The category $\C$ may be canonically identified with (the opposite of) a full subcategory of $\Pro\left(\g_{\acute{e}t}\left(k\right)\right),$ and by the \'etale spectrum functor $\Spec_{\acute{e}t},$ it may also be identified with a full subcategory $\sL_\C$ of $\Aff_{\g_{\acute{e}t}\left(k\right)}.$ (Recall that the latter category is equivalent to affine $k$-schemes). The admissible closure of $\sL_\C$ can be identified with the full subcategory of affine $k$-schemes on all those of the form $\Spec\left(B\right),$ where $B$ is a commutative $k$-algebra such that there exists an \'etale map of $k$-algebras $A \to B,$ with $A$ in $\C.$ An analogous statement holds for derived commutative $k$-algebras with respect to the \'etale geometry.

$\C$ may also be identified with (the opposite of) a full subcategory of $\Pro\left(\g_{Zar}\left(k\right)\right),$ and by the Zariksi spectrum functor $\Spec$ it may also be identified with a full subcategory $\sL_\C$ of $\Aff_{\g_{Zar}\left(k\right)}$ (and the latter is again equivalent to the category affine $k$-schemes). The admissible closure of $\sL_\C$ can be identified with the full subcategory of affine $k$-schemes on all those of the form $\Spec\left(A\left[\frac{1}{a}\right]\right),$ with $A$ a commutative $k$-algebras in $\C,$ and $a \in A.$ A slight modification of the Zariski geometry is possible, which does not change the theory in any essential way in that it has the same geometric structure and the same $\i$-category of schemes, in which the admissible morphisms are all open immersions (see Remark 2.5.11 and Remark 4.2.1 of \cite{dag}), so that the admissible closure of $\sL_\C$ can then be identified with the full subcategory of affine schemes which admit an open immersion into some $\Spec\left(A\right),$ with $A$ in $\C.$ An analogous statement holds for derived commutative $k$-algebras with respect to the Zariski topology.
\end{ex}

\begin{ex}
Let $\C$ be a full subcategory of $\bE$-algebras, with $\bE$ an $\Ei$-ring spectrum. The $\i$-category $\C$ can canonically be identified with (the opposite of) a full subcategory of $\Pro\left(\g^{nSp}_{\acute{e}t}\left(\bE\right)\right),$ and by the spectral \'etale spectrum functor $\Spec_{\acute{e}t},$ it may also be identified with a full subcategory $\sL_\C$ of the $\i$-category of affine $\bE$-schemes, $\Aff_{\g^{nSp}_{\acute{e}t}\left(\bE\right)}.$ The admissible closure of $\sL_\C$ can be identified with the full subcategory of $\Aff_{\g^{nSp}_{\acute{e}t}\left(\bE\right)}$ spanned by the image under the spectral \'etale spectrum functor of those $\bE$-algebras $\mathbf{F}$ which admit an \'etale morphism $\mathbf{A} \to \mathbf{F},$ with $\mathbf{A}$ in $\C.$ An analogous statement holds when $\bE$ is connective for the connective geometry $\g^{Sp}_{\acute{e}t}\left(\bE\right).$
\end{ex}

\begin{ex}\label{ex:LFP}
Let $\g$ be of the form $\g_{\acute{e}t}\left(k\right),$ $\g^{der}_{\acute{e}t}\left(k\right),$ $\g^{nSp}_{\acute{e}t}\left(\bE\right),$ or $\g^{Sp}_{\acute{e}t}\left(\bE\right).$ Denote by $\mathbf{LFP}_\g$ the image of $\g$ under Yoneda embedding $$\g \hookrightarrow \Pro\left(\g\right)$$ composed with $\Spec_\g.$ Because \'etale maps are locally of finite presentation, $\mathbf{LFP}_\g$ is admissibly closed, and hence a strong \'etale blossom by Proposition \ref{prop:admicls}. Moreover, since $\g$ is essentially small, $\mathbf{LFP}_\g$ is also essentially small.
\end{ex}

\begin{dfn}
Let $\C$ be a subcategory of $\Pro\left(\g\right).$ A \textbf{$\C$-scheme} is an $\sL_\C$-\'etendue, where $\sL_\C$ is the image of $\C$ in $\Str_\G$ under $\Spec_\g$. Denote the associated $\i$-category by $\sch{\C}.$ A \textbf{$\C$-Deligne-Mumford stack} is an $\i$-sheaf on $\widetilde{\sL_\C}=:\Aff_\C$ which is the functor of points of a $\C$-scheme. (Objects in the $\i$-category $\Aff_\C$ will be called \textbf{affine $\C$-schemes}.) Denote the associated $\i$-category by $\Dm{\C}.$
\end{dfn}

\begin{rmk}
There is a canonical equivalence $\sch{\C} \simeq \Dm{\C}.$
\end{rmk}

\begin{rmk}
One may also define a $\C$-Deligne-Mumford stack as a stack on $\widetilde{\C}$ with the induced Grothendieck topology from $\Pro\left(\g\right).$ Both $\i$-topoi are canonically equivalent.
\end{rmk}

\begin{dfn}
Let $\mathbf{LFP}_\g$ be as in Example \ref{ex:LFP}. The $\i$-category $\overline{\mathbf{LFP}_\g}$ of $\mathbf{LFP}_\g$-\'etendues is a full subcategory of $\g$-schemes; it is what Lurie calls the $\i$-category of $\g$-schemes which are \textbf{locally of finite presentation.} The $\i$-category of Deligne-Mumford $\mathbf{LFP}_\g$-stacks will be called (derived, spectral..) Deligne-Mumford stacks of locally of finite presentation (a.k.a locally of finite type).
\end{dfn}

The following theorem is a special case of Theorem \ref{thm:yonequil} and Proposition \ref{prop:theta}:

\begin{thm}
Let $\C$ be an essentially small subcategory of $\Pro\left(\g\right),$ for a geometry $\g.$ Then there is an adjoint equivalence
$$\xymatrix@C=2.5cm{\sch{\C}^{\et} \ar@<-0.5ex>[r]_-{\left(i^{\et}\right)^* \circ y^{\et}} & \Shi\left(\Aff^{\et}_\C\right) \ar@<-0.5ex>[l]_-{\Theta}}$$
between the $\i$-category of $\C$-schemes and \'etale maps and the $\i$-category of $\i$-sheaves on the site of affine $\C$-schemes and \'etale maps.
\end{thm}

\begin{dfn}
A morphism $\varphi:\X \to \Y$ between Deligne-Mumford $\C$-stacks will be called \textbf{\'etale} if the associated map in $\sch{\C}$ is.
\end{dfn}

\begin{ex}
If $\C$ is a full subcategory of $\Pro\left(\g_{Zar}\left(k\right)\right),$ an \'etale map between Deligne-Mumford $\C$-stacks  $\varphi:\X \to \Y$ will be called an \textbf{open immersion}. We will make similar conventions for the derived Zariski geometry.
\end{ex}

\begin{ex}
If $\C$ is a full subcategory of $\Pro\left(\g_{\acute{e}t}\left(k\right)\right),$ an \'etale map between Deligne-Mumford $\C$-stacks $\varphi:\X \to \Y$ will be called \textbf{\'etale}. That is to say, this is the correct definition of an \'etale morphism in the algebraic-geometric sense. In particular, if $\X$ and $\Y$ are both $1$-truncated, so they can be regarded as classical Deligne-Mumford stacks by Theorem \ref{thm:DM1trun}, this recovers the classical definition of an \'etale morphism of Deligne-Mumford stacks. We will make similar conventions for the derived and spectral \'etale geometries.
\end{ex}

\begin{cor}
For $\C$ essentially small, there is a canonical equivalence
$$\Dm{\C}^{\et} \simeq \Shi\left(\Aff^{\et}_\C\right)$$
between the $\i$-category of Deligne Mumford $\C$-stacks and \'etale maps and the $\i$-category of $\i$-sheaves on the site of affine $\C$-schemes and \'etale maps.
\end{cor}

\begin{cor}\label{cor:yoneqiDM}
Let $k$ be a commutative ring. There is a canonical equivalence $$\Dm{\mathbf{LFP}_k}^{\et}_\i \simeq \Shi\left(\Aff^{\mathit{lfp},\mspace{2mu}{\et}}_k\right)$$
between the $\i$-category of Deligne-Mumford $\i$-stacks locally of finite presentation over $k$ and their \'etale morphisms and the $\i$-category of $\i$-sheaves on the site of affine $k$-schemes locally of finite presentation and their \'etale morphisms (with respect to the \'etale topology).
\end{cor}

\begin{cor}
Let $k$ be a commutative ring. There is a canonical equivalence $$\Dm{\mathbf{LFP}_k}^{\et} \simeq \St\left(\Aff^{\mathit{lfp},\mspace{2mu}{\et}}_k\right)$$
between the $\left(2,1\right)$-category of classical Deligne-Mumford stacks locally of finite presentation over $k$ and their \'etale morphisms and the $\left(2,1\right)$-category of stacks of groupoids on the site of affine $k$-schemes locally of finite presentation and their \'etale morphisms (with respect to the \'etale topology).
\end{cor}

\begin{proof}
This follows immediately from Corollary \ref{cor:yoneqiDM} by passing to $1$-truncated objects and invoking Theorem \ref{thm:DM1trun}.
\end{proof}

\begin{cor}\label{cor:yoneqDMd}
Let $k$ be a commutative ring. There is a canonical equivalence $$\dm^{der}\left(\mathbf{LFP}_k\right)^{\et} \simeq \Shi\left(\mathbf{DAff}^{\mathit{lfp},\mspace{2mu}{\et}}_k\right)$$
between the $\i$-category of derived Deligne-Mumford stacks locally of finite presentation over $k$ and their \'etale morphisms and the $\i$-category of $\i$-sheaves on the site of derived affine $k$-schemes locally of finite presentation and their \'etale morphisms (with respect to the \'etale topology).
\end{cor}

\begin{ex}
Suppose that $k$ is a field of characteristic zero so that we may identify the $\i$-category of simplicial commutative $k$-algebras with the $\i$-category of non-positively graded differential $k$-algebras. Fix a non-positive integer $n$. We will explain the existence of a derived Deligne-Mumford moduli stack of $n$-shifted symplectic structures, in the sense of Definition 1.18 of \cite{shifted}. This construction works for $n>0,$ however as there are no non-trivial $n$-shifted symplectic structures on derived Deligne-Mumford stacks, the resulting stack is empty in this case. Moreover, shifted symplectic structures are only defined in the case of a perfect cotangent complex, so it is convenient to restrict to the case of affine derived $k$-schemes locally of finite type, as such schemes have a perfect cotangent complex. The assignment of $n$-shifted symplectic structures is not functorial with respect to all maps of derived schemes, since in general, non-degeneracy is not preserved under pullback. However, one can pull back $n$-shifted symplectic structures along \emph{\'etale} maps. It follows that for all integers $n,$ there exists a functor $$\Symp[n]_k:\left(\mathbf{DAff}^{\mathit{lfp},\mspace{2mu}{\et}}_k\right)^{op} \to \iGpd$$ which assigns a derived affine $k$-scheme its (possibly empty) space of $n$-shifted symplectic structures. This functor is easily seen to have \'etale descent, by using Proposition 1.11 of op. cit., combined with the observation that non-degeneracy of $2$-forms can be checked \'etale-locally.
Denote by $j$ the canonical functor $$\mathbf{DAff}^{\mathit{lfp},\mspace{2mu}{\et}}_k \to \mathbf{DAff}.$$ Then $$j!\Symp[n]_k=:\Symp[n]_{k}^{DM}$$ is a derived Deligne-Mumford stack locally of finite type. By the Yoneda Lemma, the identity map of $\Symp[n]_{k}^{DM}$ classifies an $n$-shifted symplectic structure on $\Symp[n]_{k}^{DM}$ which is universal in the sense that for all derived $k$-schemes $X$ locally of finite type, every $n$-shifted symplectic structure on $X$ arises from pulling back the one on $\Symp[n]_{k}^{DM}$ by an essentially unique \'etale morphism. In fact, the universality holds for all derived Deligne-Mumford stacks locally of finite type over $k$. One should be able to similarly construct a universal $n$-shifted Poisson Deligne-Mumford stack in the sense of \cite{shiftedpoisson}, for any integer $n$. 
\end{ex}

\begin{cor}\label{cor:yoneqDMsp}
Let $\bE$ be an $\Ei$-ring spectrum. There is a canonical equivalence $$\dm^{nSp}\left(\mathbf{LFP}_\bE\right)^{\et} \simeq \Shi\left(\mathbf{nSpAff}^{\mathit{lfp},\mspace{2mu}{\et}}_\bE\right)$$
between the $\i$-category of non-connective spectral Deligne-Mumford stacks locally of finite presentation over $\bE$ and their \'etale morphisms and the $\i$-category of $\i$-sheaves on the site of non-connective spectral affine $\bE$-schemes locally of finite presentation and their \'etale morphisms (with respect to the \'etale topology).\\ \newline
If $\bE$ is connective, there is also a canonical equivalence $$\dm^{Sp}\left(\mathbf{LFP}_\bE\right)^{\et} \simeq \Shi\left(\mathbf{SpAff}^{\mathit{lfp},\mspace{2mu}{\et}}_\bE\right)$$
between the $\i$-category of spectral Deligne-Mumford stacks locally of finite presentation over $\bE$ and their \'etale morphisms and the $\i$-category of $\i$-sheaves on the site of spectral affine $\bE$-schemes locally of finite presentation and their \'etale morphisms (with respect to the \'etale topology).
\end{cor}

\begin{ex}\textbf{Universal Zariski \'Etendues}:\\
Let $\C$ be a small admissibly closed subcategory of commutative $k$-algebras, with respect to the Zariski geometry. Let $A$ be an algebra in $\C$. Consider the Haefliger groupoid $\Ha\left(\Spec\left(A\right)\right),$ of the underling topological space of the spectrum of $A,$ defined analogously to Definition \ref{dfn:haf}, with the role of smooth open embeddings replaced with continuous ones. Since the structure maps of this topological groupoid are local homeomorphisms, one can pull back the structure sheaf on $\Spec\left(A\right)$ to give the arrow space the structure of a (highly non-separated) scheme. Let $\X\left(A\right)$ denote the Zariski \'etale stack associated to this groupoid. As a special case of Theorem 3.3 of \cite{prol}, one has that the Zariski \'etale stack $$\coprod\limits_{A_\alpha \in \C} \X\left(A_\alpha\right)$$
is the functor of points of the universal $\C$-\'etendue.
\end{ex}

\begin{ex}\textbf{Universal \'Etale \'Etendues}:\\
We will now describe how to construct universal \'etendues for $\g=\g_{\acute{e}t}\left(k\right).$ Let $\C$ be a small admissibly closed subcategory of commutative $k$-algebras, with respect to the \'etale geometry. Let $X$ denote the following scheme $$X:=\coprod\limits_{A_\alpha \in \C} \Spec\left(A_\alpha\right).$$ Notice that the $\i$-category of affine $\C$-schemes is canonically equivalent to $\C^{op},$ and by construction, an $\i$-sheaf over $\Aff_\C$ is the same as an $\i$-sheaf over $\C^{op}$ with the induced Grothendieck topology from the \'etale topology on $\Pro\left(\g_{\acute{e}t}\left(k\right)\right)=\Aff_k.$ In particular, one has that the underlying $\i$-topos of the universal $\C$-\'etendue is equivalent to $\Shi\left(\left(\C^{\et}\right)^{op},\acute{e}t\right),$ and so is $1$-localic. Let $\cE=\Sh\left(\left(\C^{\et}\right)^{op},\acute{e}t\right)$ be the associated $1$-topos. Let $$r:\Aff_\C^{\et} \hookrightarrow \mathbf{Sch}^{\et}_k,$$ be the inclusion of the subcategory of affine $\C$-schemes spanned by \'etale morphisms into the analogously defined subcategory of all schemes over $k$. Denote by $$y^{\et}:\left(\C^{\et}\right)^{op} \hookrightarrow \Sh\left(\left(\C^{\et}\right)^{op},\acute{e}t\right)$$ the Yoneda embedding, and denote by $\tilde y^{\et}$ the composite
$$\mathbf{Sch}^{\et}_k \hookrightarrow \Sh\left(\left(\mathbf{Sch}^{\et}_k\right)^{op},\acute{e}t\right) \stackrel{r^*}{\longrightarrow} \Sh\left(\left(\C^{\et}\right)^{op},\acute{e}t\right).$$ Notice that for every object $A_\alpha$ of $\C,$ there exists a canonical map $$y^{\et}\left(A_\alpha\right) \to \tilde y^{\et}\left(X\right).$$ It follows that the canonical map $$\tilde y^{\et}\left(X\right) \to 1$$ is an epimorphism. Notice that $X$ is a $\C$-scheme. It follows that $$\Sh\left(X_{\acute{e}t}\right) \simeq \Sh\left(\left(\C^{\et}\right)^{op},\acute{e}t\right)/\tilde y^{\et}\left(X\right),$$ and that moreover, the associated \'etale epimorphism $$\pi_X: \Sh\left(X_{\acute{e}t}\right) \to \Sh\left(\left(\C^{\et}\right)^{op},\acute{e}t\right)$$ enjoys that property that $\pi_X^*\O_\Uni \cong \O_X,$ that is, the pullback of the structure sheaf of the universal $\C$-\'etendue along $\pi_X$ is isomorphic to the structure sheaf of the small \'etale spectrum of $X.$ Let
$$\xymatrix{\cF \ar[r]^{s} \ar[d]_{t} & \Sh\left(X_{\acute{e}t}\right) \ar[d]^-{\pi_X}\\
\Sh\left(X_{\acute{e}t}\right) \ar[r]^-{\pi_X} & \Sh\left(\left(\C^{\et}\right)^{op}\right)}$$ be a $\left(2,1\right)$-pullback diagram of topoi. Since $\pi_X$ is \'etale, one has a canonical equivalence $$\cF\simeq \Sh\left(X_{\acute{e}t}\right)/\pi_X^*\left(y^{\et}\left(X\right)\right).$$ Since $\pi_X^*\left(y^{\et}\left(X\right)\right)$ is an \'etale sheaf over $X,$ it is represented by an algebraic space $P$ (with no separation conditions) \'etale over $X,$ that is there is an \'etale map $s:P \to X$ and an equivalence $$\Sh\left(X_{\acute{e}t}\right)/\pi_X^*\left(y^{\et}\left(X\right)\right)\simeq \Sh\left(P_{\acute{e}t}\right)$$ under which the morphism $$s:\Sh\left(X_{\acute{e}t}\right)/\pi_X^*\left(y^{\et}\left(X\right)\right) \to \Sh\left(X_{\acute{e}t}\right)$$ corresponds to $s:P \to X.$ Notice that $\pi_X^*\left(y^{\et}\left(X\right)\right)$ assigns to each \'etale map $$\Spec\left(B\right) \to X$$ the set $\Hom^{\et}\left(\Spec\left(B\right),X\right).$ The map $t$ then corresponds to another \'etale map $$t:P \to X$$ and one gets an induced groupoid structure $P \rightrightarrows X$; it is the algebraic analogue of a Haefliger groupoid for the \'etale topology (see Definition \ref{dfn:haf}).  Since $\pi_X$ is an epimorphism it follows that the Deligne-Mumford stack associated to this groupoid is the functor of points of the universal $\C$-\'etendue.
\end{ex}

For each small admissibly closed subcategory $\sU$ of $\Affg,$ denote by $j_\sU$ the composite
$$\sU^{\et} \hookrightarrow \Aff^{\et}_\g \to \Affg.$$ By Proposition \ref{prop:relprolexist}, the restriction functor $$j^*_\sU:\Shi\left(\Affg\right) \to \Shi\left(\sU^{\et}\right)$$ has a left adjoint, called the \textbf{relative \'etale prolongation functor with respect to $\sU.$}

The following theorem is a special case of Theorem \ref{thm:classificationlarge}:

\begin{thm}\label{thm:classificationg}
Let $\g$ be a geometry. An $\i$-sheaf in $\Shi\left(\Affg\right)$ is a Deligne-Mumford $\g$-stack if and only if $\X$ is in the essential image of a relative prolongation functor $j^\sU_!$ for $\sU$ a small admissibly closed subcategory of $\Affg.$
\end{thm}

\begin{rmk}
Here is how this theorem can be used in practice:\\

Let $\sU$ be some small admissibly closed subcategory of affine $\g$-schemes, and let $F$ be any $\i$-sheaf on the $\i$-category of such affine schemes and their \'etale morphisms. Then $F$ gives rise to a Deligne-Mumford $\g$-stack $j^\sU_!\left(F\right)$ with the universal property that for any affine scheme $\Specg\left(A\right)$ in $\sU,$ there is a natural equivalence of $\i$-groupoids
$$\Hom_{\et}\left(\Specg\left(A\right),j^\sU_!\left(F\right)\right) \simeq F\left(U\right),$$ where $$\Hom_{\et}\left(\Specg\left(A\right),j^\sU_!\left(F\right)\right),$$ is the full subcategory of $\Hom\left(\Specg\left(A\right),j^\sU_!\left(F\right)\right)$ on the \'etale morphisms. Moreover, $j^\sU_!\left(F\right)$ can be explicitly constructed from the universal $\sU$-\'etendue as follows: Let $\left(\Uni^\sU,\O_\Uni^\sU\right)$ denote the terminal object in $\sch{\sU}^{\et}.$ Then $j^\sU_!\left(F\right)$ is the functor of points of $\left(\Uni^\sU/F,\O_\Uni^\sU|_F\right).$ Theorem \ref{thm:classificationg} states that any Deligne-Mumford $\g$-stack arises by such a construction. In particular, given any $\i$-sheaf $G$ on the $\i$-category $\Affg$ of affine $\g$-schemes, and any small subcategory $\C$ of $\Affg,$ there exists a Deligne-Mumford $\g$-stack $\X_G$ universal with the property that for any affine $\g$-scheme $\Specg\left(A\right),$ with $A$ in $\C,$ there is a natural equivalence
$$\Hom_{\et}\left(\Specg\left(A\right),\X_G\right) \simeq G\left(U\right),$$ namely $$\X_G:=j^{\widetilde{\C}}_! j_{\widetilde{\C}}^* G,$$ where $\widetilde{\C}$ is the admissible closure of $\C$.
\end{rmk}

\textbf{Convention}: If $\g$ is of the form $\g_{Zar}\left(k\right),$ or $\g^{der}_{Zar}\left(k\right),$ an
admissibly closed subcategory of $\Affg$ will be called \textbf{Zariski closed}. If $\g$ is of the form $\g_{\acute{e}t}\left(k\right),$ $\g^{der}_{\acute{e}t}\left(k\right),$ $\g^{nSp}_{\acute{e}t}\left(\bE\right),$ or $\g^{Sp}_{\acute{e}t}\left(\bE\right),$ an admissibly closed subcategory of $\Affg$ will be called \textbf{\'etale closed}.

The following string of theorems are all special cases of Theorem \ref{thm:classificationg}:

\begin{thm}\label{thm:classificationzariksi}
Let $k$ be a commutative ring. An $\i$-sheaf in $\Shi\left(\Aff_k,Zar\right)$ on affine $k$-schemes with respect to the Zariski topology is an Zariski \'etale $\i$-stack if and only if $\X$ is in the essential image of a relative prolongation functor $j^\sU_!$ for $\sU$ a small Zariski closed subcategory of $\Aff_k.$
\end{thm}

\begin{thm}\label{thm:classificationetalei}
Let $k$ be a commutative ring. An $\i$-sheaf in $\Shi\left(\Aff_k,\acute{e}t\right)$ on affine $k$-schemes with respect to the \'etale topology is a Deligne-Mumford  $\i$-stack if and only if $\X$ is in the essential image of a relative prolongation functor $j^\sU_!$ for $\sU$ a small \'etale closed subcategory of $\Aff_k.$
\end{thm}

\begin{thm}\label{thm:classificationzariksid}
Let $k$ be a commutative ring. An $\i$-sheaf in $\Shi\left(\mathbf{DAff}_k,Zar\right)$ on affine derived $k$-schemes with respect to the Zariski topology is a derived Zariski \'etale $\i$-stack if and only if $\X$ is in the essential image of a relative prolongation functor $j^\sU_!$ for $\sU$ a small Zariski closed subcategory of $\mathbf{DAff}_k.$
\end{thm}

\begin{thm}\label{thm:classificationetaled}
Let $k$ be a commutative ring. An $\i$-sheaf in $\Shi\left(\mathbf{DAff}_k,\acute{e}t\right)$ on affine derived $k$-schemes with respect to the \'etale topology is a derived Deligne-Mumford stack if and only if $\X$ is in the essential image of a relative prolongation functor $j^\sU_!$ for $\sU$ a small \'etale closed subcategory of $\mathbf{DAff}_k.$
\end{thm}

\begin{thm}\label{thm:classificationetalensp}
Let $\bE$ be an $\Ei$-ring spectrum. An $\i$-sheaf in $\Shi\left(\mathbf{nSpAff}_\bE,\acute{e}t\right)$ on affine non-connective spectral $\bE$-schemes with respect to the \'etale topology is a non-connective spectral Deligne-Mumford stack if and only if $\X$ is in the essential image of a relative prolongation functor $j^\sU_!$ for $\sU$ a small \'etale closed subcategory of $\mathbf{nSpAff}_\bE.$
\end{thm}

\begin{thm}\label{thm:classificationetalesp}
Let $\bE$ be a connective $\Ei$-ring spectrum. An $\i$-sheaf in \\$\Shi\left(\mathbf{SpAff}_\bE,\acute{e}t\right)$ on affine spectral $\bE$-schemes with respect to the \'etale topology is a spectral Deligne-Mumford stack if and only if $\X$ is in the essential image of a relative prolongation functor $j^\sU_!$ for $\sU$ a small \'etale closed subcategory of $\mathbf{SpAff}_\bE.$
\end{thm}

Fix a commutative ring $k$. By abuse of notation, for an \'etale closed subcategory $\sU$ of $\Aff_k,$ denote again by $j^\sU_!$ the left adjoint to the restriction functor between stacks of groupoids $$\St\left(\Aff_k,\acute{e}t\right) \to \St\left(\sU^{\et}\right),$$ and call it again the relative \'etale prolongation functor with respect to $\sU.$

\begin{thm}\label{thm:classificationetale}
Let $k$ be a commutative ring. A stack of groupoids $\St\left(\Aff_k,\acute{e}t\right)$ on affine $k$-schemes with respect to the \'etale topology is a Deligne-Mumford stack if and only if $\X$ is in the essential image of a relative prolongation functor $j^\sU_!$ for $\sU$ a small \'etale closed subcategory of $\Aff_k.$
\end{thm}

\begin{proof}
This follows immediately from Theorem \ref{thm:largetruncatedversion} and Theorem \ref{thm:DM1trun}.
\end{proof}

\bibliography{etale}

\begin{thebibliography}{10}

\bibitem{sga4}
{\em Th\'eorie des topos et cohomologie \'etale des sch\'emas. {T}ome 1:
  {T}h\'eorie des topos}.
\newblock Lecture Notes in Mathematics, Vol. 269. Springer-Verlag, Berlin,
  1972.
\newblock S{\'e}minaire de G{\'e}om{\'e}trie Alg{\'e}brique du Bois-Marie
  1963--1964 (SGA 4), Dirig{\'e} par M. Artin, A. Grothendieck, et J. L.
  Verdier. Avec la collaboration de N. Bourbaki, P. Deligne et B. Saint-Donat.

\bibitem{negative}
John~C. Baez and Michael Shulman.
\newblock Lectures on {$n$}-categories and cohomology.
\newblock In {\em Towards higher categories}, volume 152 of {\em IMA Vol. Math.
  Appl.}, pages 1--68. Springer, New York, 2010.

\bibitem{von}
Garrett Birkhoff.
\newblock Von {N}eumann and lattice theory.
\newblock {\em Bull. Amer. Math. Soc.}, 64:50--56, 1958.

\bibitem{breen}
Lawrence Breen.
\newblock On the classification of {$2$}-gerbes and {$2$}-stacks.
\newblock {\em Ast\'erisque}, (225):160, 1994.

\bibitem{shiftedpoisson}
D.~Calaque, T.~Pantev, B.~To{\"e}n, M.~Vaquie, and G.~Vezzosi.
\newblock Shifted poisson structures and deformation quantization, 2015.
\newblock \href{http://arxiv.org/abs/1506.03699}{arXiv:1506.03699}.

\bibitem{etalspme}
David Carchedi.
\newblock An \'etal\'e space construction for stacks.
\newblock {\em Algebr. Geom. Topol.}, 13(2):831--903, 2013.

\bibitem{prol}
David Carchedi.
\newblock {\'E}tale stacks as prolongations, 2013.
\newblock \href{http://arxiv.org/abs/1212.2282}{arXiv:1212.2282}.

\bibitem{homtme}
David Carchedi.
\newblock On the homotopy type of higher orbifolds and {H}aefliger classifying
  spaces, 2015.
\newblock \href{http://arxiv.org/abs/1504.02394}{arXiv:1504.02394}.

\bibitem{DM}
P.~Deligne and D.~Mumford.
\newblock The irreducibility of the space of curves of given genus.
\newblock {\em Inst. Hautes \'Etudes Sci. Publ. Math.}, (36):75--109, 1969.

\bibitem{field}
Daniel~S. Freed and Constantin Teleman.
\newblock Relative quantum field theory, 2012, arXiv:1212.1692.

\bibitem{Giraud}
Jean Giraud.
\newblock {\em Cohomologie non ab\'elienne}.
\newblock Springer-Verlag, Berlin, 1971.
\newblock Die Grundlehren der mathematischen Wissenschaften, Band 179.

\bibitem{rec}
Alexander Grothendieck.
\newblock {\em R\'ecoltes et Semailles}.
\newblock Universit\'é des Sciences et Techniques du Languedoc, Montpellier,
  1985-1987.

\bibitem{Haefliger}
Andr{\'e} Haefliger.
\newblock Homotopy and integrability.
\newblock In {\em Manifolds--{A}msterdam 1970 ({P}roc. {N}uffic {S}ummer
  {S}chool)}, Lecture Notes in Mathematics, Vol. 197, pages 133--163. Springer,
  Berlin, 1971.

\bibitem{morsifold}
Richard Hepworth.
\newblock Morse inequalities for orbifold cohomology.
\newblock {\em Algebr. Geom. Topol.}, 9(2):1105--1175, 2009.

\bibitem{hepworth}
Richard Hepworth.
\newblock Vector fields and flows on differentiable stacks.
\newblock {\em Theory Appl. Categ.}, 22:542--587, 2009.

\bibitem{atomless}
John~R. Isbell.
\newblock Atomless parts of spaces.
\newblock {\em Math. Scand.}, 31:5--32, 1972.

\bibitem{elephant2}
Peter~T. Johnstone.
\newblock {\em Sketches of an elephant: a topos theory compendium. {V}ol. 2},
  volume~44 of {\em Oxford Logic Guides}.
\newblock The Clarendon Press Oxford University Press, Oxford, 2002.

\bibitem{Joyal}
A.~Joyal.
\newblock Quasi-categories and {K}an complexes.
\newblock {\em J. Pure Appl. Algebra}, 175(1-3):207--222, 2002.
\newblock Special volume celebrating the 70th birthday of Professor Max Kelly.

\bibitem{ext}
Andr{\'e} Joyal and Myles Tierney.
\newblock An extension of the {G}alois theory of {G}rothendieck.
\newblock {\em Mem. Amer. Math. Soc.}, 51(309):vii+71, 1984.

\bibitem{kontt}
Maxim Kontsevich.
\newblock Intersection theory on the moduli space of curves and the matrix
  {A}iry function.
\newblock {\em Comm. Math. Phys.}, 147(1):1--23, 1992.

\bibitem{laff}
Laurent Lafforgue.
\newblock Une compactification des champs classifiant les chtoucas de
  {D}rinfeld.
\newblock {\em J. Amer. Math. Soc.}, 11(4):1001--1036, 1998.

\bibitem{dag}
Jacob Lurie.
\newblock Derived algebraic geometry v: Structured spaces.
\newblock \href{http://arxiv.org/abs/0905.0459}{arXiv:0905.0459}, 2009.

\bibitem{htt}
Jacob Lurie.
\newblock {\em Higher topos theory}, volume 170 of {\em Annals of Mathematics
  Studies}.
\newblock Princeton University Press, Princeton, NJ, 2009.

\bibitem{spectral}
Jacob Lurie.
\newblock Derived algebraic geometry vii: Spectral schemes, 2011.
\newblock
  \href{http://www.math.harvard.edu/~lurie/papers/DAG-VII.pdf}{http://www.math.harvard.edu/~lurie/papers/DAG-VII.pdf}.

\bibitem{sheaves}
Saunders MacLane and Ieke Moerdijk.
\newblock {\em Sheaves in Geometry and Logic: A First Introduction to Topos
  Theory}.
\newblock Springer-Verlag, New York, 1st, edition, 1992.

\bibitem{sheavesonorbifolds}
I.~Moerdijk and D.~A. Pronk.
\newblock Orbifolds, sheaves and groupoids.
\newblock {\em $K$-Theory}, 12(1):3--21, 1997.

\bibitem{cont}
Ieke Moerdijk.
\newblock The classifying topos of a continuous groupoid. i.
\newblock {\em Transactions of the American Mathematical Society},
  310(2):629--668, 1988.

\bibitem{Ie}
Ieke Moerdijk.
\newblock Foliations, groupoids and {G}rothendieck \'etendues.
\newblock {\em Rev. Acad. Cienc. Zaragoza (2)}, 48:5--33, 1993.

\bibitem{shifted}
Tony Pantev, Bertrand To{\"e}n, Michel Vaqui{\'e}, and Gabriele Vezzosi.
\newblock Shifted symplectic structures.
\newblock {\em Publ. Math. Inst. Hautes \'Etudes Sci.}, 117:271--328, 2013.

\bibitem{Dorette}
Dorette Pronk.
\newblock Etendues and stacks as bicategories of fractions.
\newblock {\em Compositio Mathematica}, 102(3):243--303, 1996.

\bibitem{satake}
I.~Satake.
\newblock On a generalization of the notion of manifold.
\newblock {\em Proc. Nat. Acad. Sci. U.S.A.}, 42:359--363, 1956.

\bibitem{spivak}
David~I. Spivak.
\newblock Derived smooth manifolds.
\newblock {\em Duke Math. J.}, 153(1):55--128, 2010.

\bibitem{Hagy2}
Bertrand To{\"e}n and Gabriele Vezzosi.
\newblock Homotopical algebraic geometry. {II}. {G}eometric stacks and
  applications.
\newblock {\em Mem. Amer. Math. Soc.}, 193(902):x+224, 2008.

\bibitem{Giorgio}
Giorgio Trentinaglia and Chenchang Zhu.
\newblock Strictification of \'etale stacky {L}ie groups.
\newblock {\em Compos. Math.}, 148(3):807--834, 2012.

\bibitem{stacklie}
Hsian-Hua Tseng and Chenchang Zhu.
\newblock Integrating {L}ie algebroids via stacks.
\newblock {\em Compos. Math.}, 142(1):251--270, 2006.

\bibitem{Wockel}
Christoph Wockel.
\newblock Categorified central extensions, \'etale {L}ie 2-groups and {L}ie's
  third theorem for locally exponential {L}ie algebras.
\newblock {\em Adv. Math.}, 228(4):2218--2257, 2011.

\bibitem{jesse}
Jesse Wolfson.
\newblock Descent for n-bundles, 2013.
\newblock \href{http://arxiv.org/abs/1308.1113}{arXiv:1308.1113}.

\bibitem{zhu}
Chenchang Zhu.
\newblock {$n$}-groupoids and stacky groupoids.
\newblock {\em Int. Math. Res. Not. IMRN}, (21):4087--4141, 2009.

\end{thebibliography}
\bibliographystyle{hplain}

\end{document}